%% file: fsp_rv3.tex
\documentclass[10pt]{amsart}
\usepackage[utf8]{inputenc}

\usepackage{amsfonts}
\usepackage{amssymb}
\usepackage{tikz-cd}
\usepackage{amsthm}
\usepackage{comment}
\usepackage{graphicx}
\usepackage{hyperref}
\usepackage[backend=biber, 
style=alphabetic,
sorting=nyt,
maxnames=99,
maxalphanames=5
]{biblatex}
\usepackage{import}
\usepackage{xifthen}
\usepackage{pdfpages}
\usepackage{transparent}

\newtheorem{lemma}{Lemma}[section]

\newtheorem{theorem}{Theorem}[section]
\newtheorem{corollary}{Corollary}[section]

\newtheorem{proposition}{Proposition}[section]

\theoremstyle{definition}
\newtheorem{definition}{Definition}[section]
\theoremstyle{remark}
\newtheorem{remark}{Remark}[section]

\theoremstyle{remark}

\newcommand{\C}{\mathbb{C}}
\newcommand{\R}{\mathbb{R}}
\newcommand{\Z}{\mathbb{Z}}

\newcommand{\mc}[1]{\mathcal{#1}}
\newcommand{\bb}[1]{\mathbb{#1}}
\newcommand{\mf}[1]{\mathfrak{#1}}
\newcommand{\g}{\Gamma}

\newcommand{\M}{\mathcal{M}}
\newcommand{\cl}[1]{\text{Cliff}(\mathbb{P}^{#1})}
\newcommand{\T}{\text}
\newcommand{\p}{\mathbb{P}}

\newcommand{\I}{\text{Ind}}
\newcommand{\e}{\text{Edge}}

\newcommand{\grad}{\text{grad}}
\newcommand{\crit}{\text{Crit}}

\renewcommand{\e}{\text{Edge}}

\newcommand{\ham}{\text{Ham}}
\newcommand{\lag}{\text{Lag}}
\newcommand{\hol}{\text{Hol}}

\author{Douglas Schultz}\address{Institut f\"ur Mathematik, Humboldt-Universit\"at zu Berlin \\
Unter den Linden 6, 10099 Berlin, Germany} \email{\href{mailto:sultzdou@math.hu-berlin.de}{sultzdou@math.hu-berlin.de} }
\thanks{Author is supported by the ERC grant
\textsf{TRANSHOLOMORPHIC}}
\title{Holomorphic disks and the disk potential for a fibered Lagrangian}
\begin{document}
\begin{abstract}
We consider a fibered Lagrangian $L$ in a compact symplectic fibration with small monotone fibers, and develop a strategy for lifting $J$-holomorphic disks with Lagrangian boundary from the base to the total space. In case $L$ is a product, we use this machinery to give a formula for the leading order potential and formulate an unobstructedness criteria for the $A_\infty$ algebra. We provide some explicit computations, one of which involves finding an embedded $2n+k$ dimensional submanifold of Floer-non-trivial tori in a $2n+2k$ dimensional fiber bundle.
\end{abstract}

\maketitle
\tableofcontents

\section{Introduction}
One of many fundamental problems in modern symplectic topology is finding Lagrangians that are \emph{non-displaceable}. In other words, one can ask which Lagrangians cannot be completely moved from themselves under a Hamiltonian flow. To obstruct such a displacement, one typically defines a type of cohomology theory, called Lagrangian Floer theory.

Lagrangian Floer theory, by its very nature, it extremely hard to compute. The complexity stems from the fact that one typically wants to count isolated holomorphic curves, i.e solutions to the Cauchy-Riemann equation, and often the moduli space of solutions is not smooth or of an expected dimension.  To circumvent this problem, one perturbs the equation in some manner. The perturbation process changes what one would expect to be holomorphic and often ruins any intuition.

In this paper and its siblings, we attempt to make the computation problem easier by developing a theory for \emph{fiber bundles} as in classical algebraic topology. Such an approach is taken in the various flavors of Floer theory as in \cite{perutzgysin,ritter,khanevskybiran,hutchings,amorim,
oancea1,oancea2,oanceabourgeois}  and Sodoge's thesis \cite{sodoge} with this same notion in mind. The underlying question that we  answer is: Let $(F,L_F)$ and $(B,L_B)$ be two symplectic manifolds resp. Lagrangian submanifolds with well-defined and non-trivial Lagrangian Floer cohomology. Can we construct a product Lagrangian $L_B\times L_F$ in the fiber bundle $E$ with base $B$ and fiber $F$ and such that $L_B\times L_F$ has well-defined and non-trivial Floer cohomology?  We try to keep the overarching assumptions (e.g. monotonicity, rationality) loose.

Let $E$ be a compact, smooth fiber bundle
$$F\rightarrow E\rightarrow B$$
where the base and fibers are equipped with a symplectic structure $(F,\omega_F)$ resp. $(B,\omega_B)$. For the purposes of this paper, we take the following definition:
\begin{definition}\label{symplecticfibration1stdef}
A \emph{symplectic fibration} is a fiber bundle as above equipped with a symplectic form $\omega_E$ on the total space such that
\begin{enumerate}
\item $\omega_E\vert_F=\omega_F$ and the transition maps preserve the fiberwise symplectic form, and
\item the connection $H:=TF^{\omega_E\perp}$ has holonomy in Hamiltonian diffeomorphisms of the fibers.
\end{enumerate}
\end{definition}
By Chapter 1 in \cite{guillemin} and the generalization \cite[Thm 6.4.1]{ms1}, the holonomy assumption provides a closed 2-form $a$ on $E$ with $$a\vert_{TF}=\omega_F,\qquad H=TF^{a\perp},$$
and such that $a(v^\sharp,w^\sharp)$ has fiberwise zero average where $\sharp$ indicates the horizontal lift of vector fields from the base. The form $a$ is called the \emph{minimal coupling form} associated to $H$. One can place a symplectic form on $E$
$$\omega_{H,K}:=a+K\pi^*\omega_B$$
where $K>>1$ so that the form is non-degenerate in the horizontal direction. Such a symplectic form is called a \emph{weak coupling form}.

A Lagrangian in $(E,\omega_{H,K})$ can sometimes be realized as a ``lift'' of a Lagrangian $L_B\subset (B,\omega_B)$: In a fiber $\pi^{-1}(p)\in \pi^{-1}(L_B)$, one looks for a Lagrangian $L_F$ that can be made invariant under parallel transport over $L_B$ (see Lemma \ref{fiberedlagrangianlemma} and Definition \ref{fiberedlagrangiandefintro}). A simple example to keep in mind is the projectivization
$$\p^1\rightarrow \p(\mc{O}\oplus \mc{O}_n)\rightarrow \p^2$$ where $L$ is a ``section of great circles'' invariant under the $S^1$ holonomy above some moment tori in $\p^2$.

In \cite{floerfibrations}, we gave a crude way to compute the Lagrangian Floer cohomology of a fibered Lagrangian by mimicking the idea of the classical Leray-Serre spectral sequence of a fiber bundle. For applications this was unsatisfying for several reasons: it is hard to compute the higher pages of the sequence unless one has information about \emph{all} of the holomorphic disks in the base up to a sufficiently high index. Moreover, it was still hard to decide when $L$ was unobstructed, i.e. if the natural Floer operator $\delta$ satisfied $\delta^2=0$. We give solutions to both of these problems in the present paper.

We use a version of Lagrangian Floer self-cohomology based on that of Biran-Cornea's pearl complex \cite{birancorneapearl}. Roughly, one picks a Morse-Smale function on $L$ and attempts to define a Floer differential that counts ``pearly Morse trajectories'', or Morse flows interrupted by boundary marked $J$-holomorphic disks. In order for such a count to be defined, one must perturb the $\bar{\partial}$ equation somehow so that the Moduli space of such trajectories is smooth of expected dimension with an expected compactification. We accomplish this in \cite{floerfibrations}, and explain the results in Section \ref{floerfibrationsresults}.

To make the Floer theory manageable, we have technical assumptions on the Lagrangians and their ambient symplectic manifolds. Recall that a symplectic manifold $(F,\omega_F)$ is called \emph{monotone} if $\exists \lambda>0$ so that
$$\int_{S^2}u^*\omega_F= \lambda \int_{S^2}u^*c_1(TF)$$
for all maps of spheres $u:S^2\rightarrow F$. We say that a symplectic manifold $(B,\omega_B)$ is \emph{rational} if $[\omega_B]\in H^2(B,\bb{Q})$. For the main results of this paper, we consider a special class of symplectic fibrations. Precisely, let $G$ denote the structure group of the fibration in Definition \ref{symplecticfibration1stdef}, and assume that $G$ acts effectively via Hamiltonian diffeomorphisms on $(F,\omega_F)$. Let $J_F$ be a tamed almost complex structure on $(F,\omega_F)$. For $g\in G$, we will also denote the associated diffeomorphism on $F$ by $g$. For the derivative $dg$ of the map $g$, the action on $J_F$ is via conjugation:
$$J_F\mapsto dg\circ J_F\circ dg^{-1}=:g_*J_F.$$
Further, suppose that the action of $G$ extends smoothly to an action of $G_\bb{C}$, the complexification of $G$, on $F$.
\begin{definition}
We say that $(F,\omega_F,J_F)$ is a $G$-invariant triple if $\omega_F$ is $G$-invariant and $J_F$ is $G_\bb{C}$ invariant.
\end{definition}
Thus, for such a triple to be $G$-invariant we require that the action by $G$ extend to a holomorphic action by $G_\bb{C}$. An example of such is $\bb{P}^n$ with the Fubini-Study form, $G=T^n$, and $J_F$ as the integrable toric structure.

\begin{definition}\label{sympkahlerdefintro}
A \emph{symplectic K\"ahler fibration} is
\begin{enumerate}
\item (Monotone/Rational) a symplectic fibration
$$F\rightarrow E\xrightarrow{\pi} B$$ where $(F,\omega_F)$ is monotone, $(B,\omega_B)$ is rational, $E$ is equipped with the symplectic form
$$\omega_{H,K}=a+K\pi^*\omega_B$$
for some minimal coupling form $a$ associated to a connection $H$,
\item \label{finitedimensionalasspintro} the structure group $G$ is a compact Lie group and holonomy in the connection $H$ is $G$-valued, and
\item \label{kahlerassumptionintro}(K\"ahler and $G$-invariant) $(F,\omega_F,J_I)$ is a $G$-invariant triple with $J_I$ an integrable complex structure.
\end{enumerate}
\end{definition}
Items \eqref{finitedimensionalasspintro} and \eqref{kahlerassumptionintro} will allow us to holomorphically trivialize pullbacks of the fibration in Section \ref{liftingsection}, which is important for lifting disks from $B$ to $E$.
\begin{remark}
Items \eqref{finitedimensionalasspintro} and \eqref{kahlerassumptionintro} make this definition finer than that in \cite{floerfibrations}, where we consider fibrations as in \eqref{symplecticfibration1stdef} with Hamiltonian holonomy around contractible loops and only the monotone/rational assumption. The latter fibration is only guaranteed to have structure group contained in $\ham (F,\omega_F)$.
\end{remark}
We say that a Lagrangian $L_F\subset F$ is \emph{monotone} if there is a constant $\alpha>0$ so that
$$\int_C u^*\omega_F=\alpha\mu(u)$$
for all maps of disks $u:(D,\partial D)\rightarrow (F,L_F)$, where $\mu(u)$ is the Maslov index.

We say that a Lagrangian $L_B\subset B$ is \emph{rational} if the set
$$\left\lbrace \int_D u^*\omega_B| u:(D,\partial D)\rightarrow (F,L_F)\right\rbrace$$
is a discrete subset of $\bb{R}$.

\begin{definition}\label{fiberedlagrangiandefintro}
A \emph{fibered Lagrangian} in a symplectic K\"ahler fibration is a Lagrangian $L\subset (E,\omega_{H,K})$ that fibers as
$$L_F\rightarrow L\xrightarrow{\pi} L_B$$
with $L_F\subset (F,\omega_F)$ a monotone Lagrangian with minimal Maslov index $\geq 2$ and $L_B\subset (B,\omega_B)$ a rational Lagrangian. Further, we require that $L_B$ be equipped with a relative spin structure and $L_F$ a spin structure.
\end{definition}
To derive the disk potential, we will also need $L$ to be a product:

\begin{definition}\label{trivfiberedlagrangiandefintro}
An \emph{ambiently trivial fibered Lagrangian} in a symplectic K\"ahler fibration is a fibered Lagrangian together with a fiberwise Hamiltonian trivialization of $\pi^{-1}(L_B)\cong F\times L_B$ and such that $L\cong L_F\times L_B$ in this trivialization.
\end{definition} We will use the above term and \emph{trivially fibered Lagrangian} interchangeably.

An example of trivially-fibered Lagrangian is as follows: Let $\text{Flag}(\bb{C}^3)$ be the manifold of full flags
$$\left\lbrace V_1\subset V_2\subset \bb{C}^3: \dim_\bb{C}=i\right\rbrace.$$
By forgetting $V_1$, we get a fibration
$$\p^1\rightarrow \text{Flag}(\bb{C}^3)\rightarrow \p^2$$
that can be equipped with a weak coupling structure induced by the symplectic structure as a coadjoint orbit.
Let $\text{Cliff}(\p^k)\subset \bb{P}^k$ be the Clifford torus fiber represented by the barycenter of the moment polytope of $\bb{P}^k$. In Subsection \ref{flag3}, we find a Lagrangian 
$$\text{Cliff}(\p^1)\rightarrow L\cong \text{Cliff}(\p^1)\times \text{Cliff}(\p^2)\rightarrow\text{Cliff}(\p^2)$$
by symplectically trivializing the bundle $\pi^{-1}(L_B)$ with a result from \cite{guillemin} (Theorem \ref{moment}). Moreover, one can realize this particular fiber subbundle as Lagrangian with respect to any flat, trivial symplectic connection, so that $L$ can be seen as topologically trivial but having a integer ``Dehn Twist'' in the fiber factor that is induced by the connection. We show that the Floer cohomology of $L$ with field coefficients is non-zero and recover a result of \cite{nish} from a very different perspective. In Subsection \ref{fullflags}, we do a similar thing in the manifold of (initially) complete flags.

So far, we haven't stated any Floer theoretic results, but to entice the reader we give a consequence of the machinery that we develop which is slightly less common. Indeed, in the compact case there are only a handful of examples of continuum families of non-displaceable Langrangians in the literature \cite{fooo4} \cite{bormanquasistates} \cite{viannafamilies} \cite{viannatonkonog} \cite{woodwardwilsonquasimap} and few others. Among other things, we show that there is a real codimension $1$ embedded family of non-displaceable Lagrangians in a compact symplectic fibration of arbitrary real dimension at least $4$. Let $\mc{O}_i$ be the complex line bundle over $\bb{P}^n$ with Chern number $i$ and holomorphic structure arising from the sheaf of degree $i$ homogeneous polynomials.
\begin{theorem}\label{familiestheorem}
There is a symplectic fibration structure on $$\bb{P}^k\rightarrow \bb{P}(\mc{O}\oplus \mc{O}_{i_1}\cdots\oplus \mc{O}_{i_k})\rightarrow \bb{P}^n$$
such that there is an embedded family of pairwise disjoint trivially fibered Lagrangians $L_{\alpha_1\dots \alpha_n}$ parametrized by $B_\epsilon(0)\subset \bb{R}^n$ with $\pi(L_{0})=\cl{n}$ that are non-displaceable by Hamiltonian isotopy.
\end{theorem}
The total dimension of the family as a submanifold is $2n+k$. In particular, for $k=1$ we have the aforementioned real codimension $1$ family. As far as the author knows, we only have such a bound on the codimension in ambient dimension $4$ as given in the above references.

The essential reason for this is as follows: For a non-Clifford torus moment fiber we have $HF(L,\bb{F})=0$ where $\bb{F}$ is a Novikov field, due to the fact that Maslov index 2 disks have unequal area (for instance, see \cite[Ex 5.2]{fooo3}). We show that for each of these disks contributing to the potential function, there is a distinguished vertically constant lift that contributes to the potential of $L_{\alpha_1\dots \alpha_n}$, and the symplectic area of these disks in the total space depends on the holonomy of the symplectic connection around the boundary. By changing the holonomy we can make it so that the areas of the lifted disks become equal in the total space. We then run a typical argument of finding a unital critical point of the representation to show $HF(L_{\alpha_1\dots \alpha_n},\bb{F})\neq 0$.

As suggested by Vianna and Rizell in conversation, this method of deformation is likely a special case of the ``dual'' of bulk deformations as developed in \cite{fooo4}. Thus, it is likely that one can achieve a similar result with the bulk machinery. It still appears to be open as to whether one can find an open set of non-displaceable Lagrangians in the compact smooth case (c.f. the closed orbifold case in \cite[Ex 4.11]{woodwardwilsonquasimap}).
\subsection{Floer theoretic results}
We show that the leading order terms in the potential defined in the context of Fukaya-Oh-Ono-Ohta's $A_\infty$ algebra \cite{fooo} can be written as a sum of the potentials from the base and fiber. To describe the result, let us work with Biran-Cornea's pearl complex \cite{birancorneapearlsurvey} defined over a Novikov ring with $\bb{C}$ coefficients. Let $t$ be a formal variable and define the ring
\begin{equation}\label{uninovikovdef} \Lambda_t:=\left\lbrace \sum_{i} c_{i} t^{\alpha_i} \vert c_{i}\in \bb{C},\,\alpha_i\in \bb{R}_{\geq 0}, \#\lbrace c_{i}\neq 0, \alpha_i\leq N\rbrace <\infty \right\rbrace
\end{equation}

Let $f$ be a Morse-Smale function on $L$ that is adapted to the fibration structure, i.e. it is a sum of functions $\pi^*b+g$ with $b$ Morse on the base and $g\vert_{\pi^{-1}(\crit(b))}$ Morse. We can label the critical points $y^i_j$ where $x_j:=\pi(y_j^i)$ is critical and fixing $j$ describes all of the critical points in the fiber. Define the Floer chains to be
$$CF(L,\Lambda_t):=\bigoplus_{x\in \crit f}\Lambda_t \cdot x.$$

Let $\M_\g(L,D,y_0,\dots, y_n)_0$ denote the moduli space of pearly Morse trees on $L$ of index $0$ with limits $y_1,\dots, y_n$ along the leaves and $y_0$ along the root (see figure \ref{treediskpicture}). In Section \ref{floerfibrationsresults} we discuss why one can find a comeager set of domain dependent almost complex structures adapted to the fibration that make this into a smooth manifold of expected (zero) dimension and realized as a part of the compactification of a higher dimensional moduli space.

Choose a group homomorphism $\rho\in \text{Hom}(\pi_1(L),\Lambda_t^\times)$ and for a disk class $[u]\in H_2(L,\bb{Z})$ let $\mathrm{Hol}_\rho(u)$ denote the evaluation of $\rho$ on $u\vert_{\partial D}$, and 
$$e(u)=\int u^*\omega$$
the symplectic area. The $A_\infty$ maps for a Lagrangian $L$ are defined as 
\begin{align}
&\nu^n_{L,\rho}:CF(L,\Lambda_t)^{\otimes n}\rightarrow CF(L,\Lambda_t),\\
&\nu^n_{L,\rho}(y_1\otimes\cdots y_n)\\
\notag &\quad =
\sum_{y_0, [u]\in\M_\g(L,D,\underline{y})_0} (\sigma(u)!)^{-1}(-1)^{\sum_ii\vert y_i\vert}\varepsilon(u)\mathrm{Hol}_\rho( u)t^{e(u)}y_0.
\end{align}
where $\underline{y}$ is shorthand as above, $\vert\cdot\vert$ denotes degree in a relative grading, $\varepsilon(u)=\pm 1$ depending on the orientation of the moduli space, and $\sigma (u)$ is the number of interior marked points on $u$ required to map to a symplectic divisor as in \cite{CW2}\cite{CM}. The fact that these actually satisfy the $A_\infty$ axioms is a standard result due to the compactness result \ref{compactnessthm} and we discuss this in Section \ref{invariantssection}.

%
The \emph{potential of $L$} is defined as $\nu_{\rho}^0$ and viewed as a function of $\rho$. It lives at the heart of Floer cohomology in the sense that it is both an obstruction to $(\nu^1)^2=0$ as well as a means of computation of the cohomology in the unobstructed case \cite{fooo3}. In the later situation one can sometimes show that the Floer cohomology is isomorphic to the Morse cohomology of $L$ at a critical point of $\nu^0_\rho$.

The lowest degree terms are called the \emph{leading order potential} and we will denote it
$$\mc{W}_L^1(\rho):=\left[ \nu^0_\rho\right]_{\deg_t=\mc{J}_x}$$
where
$$\mc{J}_x= \bigg\lbrace e(u): e(u)=\min \lbrace e(v):v\in \mc{M}(E,L,x)_0\rbrace \bigg\rbrace$$
i.e. the minimal energy terms for each output critical point. 

In this article, we only consider the leading order potential. The rationale behind this is that at a non-degenerate critical point of $\mc{W}_L^1(\rho)$, one can use the non-degeneracy and the adic filtration to solve inductively for a critical point $\nu_\rho^0$. This idea is demonstrated in \cite[Lemma 10.16]{fooo3}, and we apply this reasoning in Theorem \ref{crit}.

In the fibration context, a critical point of the leading order potential may be degenerate for the following reason: the size of the fibers can be varied and the disks contained in a fiber become the lowest order terms as $K$ grows, leaving out information from the base. Thus it is natural to consider first and second order terms:

\begin{definition}
The \emph{second order potential} for a symplectic K\"ahler fibration is
\begin{equation}
\mathcal{W}_L^2(\rho):=
\sum_{\substack{u\in\mathcal{I}_x\\x\in\T{crit} (f)}}(\sigma(u)!)^{-1}\varepsilon(u)\T{Hol}_\rho(u)t^{e(u)}x
\end{equation}
where for each $x$
\begin{align*}\mathcal{I}_x &= \bigg\lbrace u\in \mathcal{M}(E,L,x)_0\,\vert e(\pi\circ u)=0\bigg\rbrace \\
&\quad \bigcup\left\{ u\,\vert e(u)=\min_{v\in \mathcal{M}(E,L,x)_0}\left\{ e(v):e(\pi\circ v)\neq 0\right\}\right\}.
\end{align*}
\end{definition}

Such a definition counts the holomorphic disks contained in a single fiber, in addition to those with minimal total energy among the homology classes with non-zero base energy.

To compute the second order potential in the trivially fibered case, we develop a lifting operator $\mc{L}$ that lifts pearly Morse trajectories from $(B,L_B)$ to configurations in $(E,L)$. In particular the lifted configurations are covariant constant with respect to some symplectic connection and have output as the ``top of the fiber'', i.e. the unique maximal index critical point for the vertical (negative) Morse flow.

To this end, let 
\begin{equation}
\nu^0_{L_B,\rho}= \sum_{x_i, [u]\in\M_\g(B,L_B,x_i)_0} (\sigma(u)!)^{-1}\varepsilon(u)\mathrm{Hol}_{\rho}(u)t^{e_B(u)}x_i
\end{equation}
be the potential for $L_B\subset B$ where $e_B(u):=\int_C Ku^*\omega_B$. The $A_\infty$ algebra for $L_B$ is well defined by the rationality assumption as in \cite{CW2}, as this allows the use of a stabilizing divisor and domain dependent perturbations to achieve transversality and compactness. For each critical point $x_i$, let $x_i^M$ be the fiberwise maximal index critical point for $-g\vert_{\pi^{-1}(x_i)}$. The \emph{lifted potential} is
\begin{equation}
\mc{L}\circ\nu^0_{L_B,\rho}= \sum_{x_i, [u]\in\M_\g(B,L_B,x_i)_0}  (\sigma(u)!)^{-1}\varepsilon(u)\mathrm{Hol}_{L_B}( \mc{L}u)t^{e_B(u)+e_v( \mc{L}u)}x_i^M
\end{equation}
where $e_v(\mc{L}u):=\int_C\mc{L}u^*a$ is the \emph{vertical symplectic area} and $\mc{L}$ is a vertically constant lift of $u$ with output $x^M_i$. We show the existence, uniqueness, and regularity of such a lift in Theorem \ref{liftconfigthm}, as well as establish the fact that the lifts live in a moduli space of expected dimension $0$. Moreover, we show in Proposition \ref{coherentorientations} that one can choose an orientation on $L_F$ so that $\mc{L}$ is orientation preserving. Let 

\begin{equation}\nu^0_{L_F,\rho} =\sum_{x_i, [u]\in\M_\g(F,L_F,x^M)_0} \varepsilon(u) \mathrm{Hol}_{\rho}( u)t^{e_v(u)}x^M
\end{equation}
denote the $0^{th}$ order structure map for $L_F\subset F$, which is a multiple of the maximal index critical point $x^M$ by monotonicity (e.g. see \cite{birancorneapearl,fooo3}, etc.).

Let $i_{M}:F\rightarrow E$ denote the inclusion of $F$ as a fiber above the maximal index base critical point $x_M$, and take the included potential to be
\begin{equation}
i_{M*}\circ \nu^0_{L_F,\rho} =\sum_{[u]\in\M_\g(F,L_F,x^M)_0} \varepsilon(u)\mathrm{Hol}_{\rho}( u)t^{e_v(u)}x_M^M.
\end{equation}

We prove the following recipe for computing the second order potential

\begin{theorem}\label{maintheoremintro}
Let $E$ be a compact symplectic K\"ahler fibration and $L$ a trivially fibered Lagrangian. Let $(P_\g)_\gamma$ be a choice of regular, coherent, $C^\infty$ perturbation datum as per Theorem \ref{transversality} and Theorem \ref{liftconfigthm}. For $K$ large enough in the weak coupling form, the only terms appearing in the second order potential for $L$ are the vertically constant configurations with output a fiberwise maximum and horizontally constant configurations in the maximal critical fiber
\begin{equation}\label{2ndorderpotentialintro}
\mathcal{W}_{L}^2(\rho)=\bigg[\mathcal{L}\circ \nu^0_{L_B,\rho}\bigg]_{\deg_{t}=\mc{K}_x}+i_{x_M*}\circ\nu^0_{L_F,\iota^*\rho}
\end{equation}
where for each generating critical point $x$, we have
$$\mc{K}_x=\min \bigg\lbrace e(u):[u]\in \M(E,L,x)_0,\pi\circ u\neq const.\bigg\rbrace$$
with $e(u)=\int_C u^*(a+K\pi^*\omega_B)$.
\end{theorem}
In the general situation, the most important consequence of this theorem is an unobstructedness criterion. We say that the $A_\infty$ algebra of $L$ is \emph{weakly unobstructed} if $\nu^0_L$ is a multiple of the maximal index critical point. By monotonicity, it follows that $L_F\subset F$ is always weakly unobstructed, and thus Theorem \ref{maintheoremintro} implies that $L$ is weakly unobstructed if $L_B$ is. To justify the name of this definition, in Section \ref{divisorequation} we show a result that is similar to that in the literature (c.f. \cite[Thm 4.10]{fooo3} \cite[Prop 4.36]{CW2}):

\begin{theorem}\label{unobstructednessthmintro}
If $\nu^0_{L_B}$ is a multiple of the unique index zero critical point of $b$ on $L_B$ and $(B,L_B)$ has no holomorphic disks of Maslov index $<2$, then there is deformation $(\nu^n_{L,w})$ of the $A_\infty$ algebra of $L$ such that $(\nu^1_{L,w})^2=0$. Moreover, if $\rho$ is a non-degenerate critical point of \eqref{2ndorderpotentialintro} and $H^*(L,\Lambda_t)$ is\linebreak generated by degree $1$ elements, then there is a representation $\xi\in\linebreak \text{Hom}(\pi_1(L),\Lambda_t^\times)$ such that
$$H(CF(L,\Lambda_t),\nu^1_{\xi,w})\cong H^*(L,\Lambda_t).$$
\end{theorem}			

By a ``deformation'' of an algebra $\nu^n$ we mean the $A_\infty$ algebra defined by
\begin{align}
&\nu^n_w(x_1\otimes\cdots x_n)\\
\notag &\quad :=\sum_{i_0,\dots, i_{n}\in \bb{Z}_{\geq 0}}\nu^{n+i_0+\cdots i_{n}}(w^{\otimes i_0}\otimes x_1\otimes w^{\otimes i_2}\otimes\cdots x_n\otimes w^{\otimes i_{n}})
\end{align}
and in the context of Thm.~\ref{unobstructednessthmintro} we show that we can deform $(CF(L,\Lambda_t),\nu^n)$ to an $A_\infty$ algebra-with-strict-unit as in \cite[Ch 3]{fooo} and that there is a solution $w$ to the weak Maurer-Cartan equation.

Theorem \ref{maintheoremintro} and its corollary answer the aforementioned shortcomings of the spectral sequence from \cite{floerfibrations}.

We close the introduction with some technical comments.
The reason for the trivially fibered assumption in Definition \ref{trivfiberedlagrangiandefintro} is evident in a single Proposition \ref{verticalmaslovprop}.  Any disk $u:(D,\partial D)\rightarrow (E,L)$ induces a section $\tilde{u}$ of $(\pi\circ u)^* E$ (assuming $\pi\circ u\neq const$), and a symplectic trivialization of $\pi\circ u^* E$ gives a ``loop of $L_F$'s'' over $\partial D$ as in \cite{salamonakveld}. The \emph{vertical Maslov index} is defined as the Maslov index of $\tilde{u}$ as a map into $F$. We show non-negativity of this index in the trivial case, where the proof also leans heavily on the fact that the fibers are monotone. The non-negativity allows us to deduce that the projections of configurations from the potential of $L$ appear in the potential of $L_B$. On the other hand, $\tilde{u}$ is equivalent to a map into $F$ which satisfies the perturbed $\bar{\partial}$-problem with moving boundary conditions
\begin{gather}\label{hamiltonianperturbedintro}
\bar{\partial}_{J_F,j}\tilde{u}+X_\delta^{0,1}(\tilde{u})=0\\
\tilde{u}(e^{i\theta})\in L\cap F_{u(e^{i\theta})} 
\end{gather}
where $X_\delta$ is one form that is Hamiltonian vector field-valued which prescribes the parallel transport on the bundle $(\pi\circ u)^*E$ [Section \ref{hamperturbsection}, c.f. \cite{gromov, salamonakveld}]. In general, solutions to this can have negative Maslov index, in which case $u$ will have negative vertical Maslov index.

Of course, transversality is extremely crucial in Floer theory. We show in \cite{floerfibrations} and summarize in Section \ref{floerfibrationsresults} that one can achieve transversality for the Moduli spaces of disks in the total space by considering the pullback of a stabilizing divisor $\pi^{-1}(D_B)\subset E$ and using (lifts of) domain dependent almost complex structures based on \cite{CW1,CW2,CM}. For configurations which might be contained in a single fiber, we rely heavily on the fact that the fibers are monotone.

To lift a disk $v:(D,\partial D)\!\rightarrow\! (B,L_B)$, we use an idea by Donaldson \cite{sdh} that applies heat flow and flattens the connection over $v^*E$. The transformation over the boundary of a disk is guaranteed to be $G$ valued by Donaldson, so boundary conditions are preserved. Once the connection is flat, we pick a covariant-constant section (hence the vertically constant terminology) and show that it is holomorphic in $E$. Moreover, we show transversality for such a section for any fiber almost complex structure; this argument is base upon the argument for transversality of constant disks. As a corollary, we get an existence result for the moving boundary problem \eqref{hamiltonianperturbedintro} in case the Hamiltonian connection is $G$-valued, where the covariant-constant condition says that we've found a periodic orbit of $\delta$ that is contractible along a family of periodic orbits. Such a solution is most likely non-constant if the loop of Lagrangians $\lbrace L\cap F_{v(e^{i\theta})}\rbrace$ is not isotopic to the constant loop (and we note that the lifting result in Section \ref{okasection} does not need the trivially fibered assumption).

\subsection{Outline}
Section \ref{sympfibsection} gives the background on symplectic fibrations and the construction of an adapted Morse flow on fiber bundle.

Section \ref{floerbackgroundsection} gives a precise definition of the domains for the pearly Morse trees and sketches the perturbation system from \cite{floerfibrations}. We summarize the transversality and compactness results that we need in Section \ref{floerfibrationsresults}.

Section \ref{liftingsection} is the beginning of the original content for this paper, where we develop a lifting operator and prove some results on the index.

We define the invariants in Section \ref{invariantssection} and prove Theorem \ref{maintheoremintro} in Section~\ref{trivfiberedsection}.

We compute the mentioned examples, including Theorem \ref{familiestheorem}, in Section~\ref{applicationsection}.

\subsection*{Acknowledgments}
Part of this paper is the second chapter of the author's thesis. As such, I want to thank Chris Woodward for suggesting such an interesting program and providing guidance along the way. My thanks go to Nick Sheridan for his help with many of the computations, the observation and relevance of Remark \ref{hirzebruchremark}, the proof of Theorem \ref{crit}, and his many other insightful comments. I thank Renato Vianna and Georgios Rizell for helpful discussions relating to Theorem \ref{familiestheorem} and bulk deformations. I thank Jean Gutt for a helpful discussion about solutions to the moving boundary problem \eqref{hamiltonianperturbedintro}. Finally, I would like to thank David Duncan and Denis Auroux for illuminating discussions about the flag examples.

\section{Symplectic fibrations}\label{sympfibsection}
To proceed, we review some of the background on the subject of symplectic fiber bundles. Let $(F,\omega_F)$ and $(B,\omega_B)$ be compact, connected symplectic manifolds, and let 
$$F\rightarrow E\xrightarrow{\pi} B$$
be a fiber bundle. A \emph{symplectic fibration} is such a space $E$ where the transition maps are symplectomorphisms of the fibers: 
\begin{align*}
&\Phi_i:\pi^{-1}(U_i)\rightarrow U_i\times F\\
&\Phi_j\circ \Phi_i^{-1}:U_i\cap U_j\times F\rightarrow U_i\cap U_j\times F\\
&(p,q)\mapsto (p,\phi_{ji}(q))
\end{align*}
where $\phi_{ji}:U_i\cap U_j\rightarrow \mathrm{Symp}(F,\omega_F)$ are {\u C}ech co-cycles.

Let $F_p$ denote the fiber at $p\in B$. Suppose that there is a closed $2$-form $a\in C^\infty (E,\wedge^2 T^*E)$ so that
$$a\vert_{F_p}=\omega_F$$
up to Hamiltonian symplectomorphism. The fiberwise non-degeneracy of the form $a$ defines a connection on $E$ via
$$TF^{\perp a}=:H_a$$
so that $TE\cong H_a\oplus TF$. By \cite[Thm 1.2.4]{guillemin}, closedness of $a$ implies that the connection is symplectic. We say that the connection $H_a$ is \emph{Hamiltonian} if the holonomy of around any contractible loop in the base is a Hamiltonian symplectomorphism of the fiber. Such a property guarantees the existence of a natural symplectic form.

\begin{theorem}\cite[Thm 6.21]{ms1}\label{coupling}
Let $H$ be a symplectic connection on a fibration $F\rightarrow E\rightarrow B$ with $\dim F=2n$. The following are equivalent:
\begin{enumerate}
\item The holonomy around any contractible loop in $B$ is Hamiltonian.
\item There is a unique closed connection form $a_H$ on $E$ such that $i^*a_H=\omega_F$ and
\begin{equation}\label{normalizationcondition}
\int_F a_H^{(n+1)} = 0 \in C^{\infty}(B,\Lambda^2(T^*B)),
\end{equation}
\end{enumerate}
where $\int_F a_H^{(n+1)}$ is integration along the fiber. The form $a_H$ is called \emph{the minimal coupling form} associated to the connection $H$.
\end{theorem}

Essentially, $a$ is already defined on $TF\wedge TF$ and $TF\wedge H$, so it suffices to define it on $H\wedge H$. One does this by using the \emph{curvature identity}. For a vector field $v$ on $B$, let $v^\sharp$ denote its horizontal lift to $H$. Denote by $\iota(v)a$ contraction of a vector field with a form. From \cite[Eq 1.12]{guillemin} we have
\begin{equation}\label{curvatureidentity}
-d\iota(v^\sharp)\iota( w^\sharp)a=\iota(\pi_{TF}[v^\sharp,w^\sharp])a \: (mod\, TB)
\end{equation}
for any closed connection form $a$, where $\pi_{TF}$ is the projection onto $TF$. Hence the \emph{curvature} $\pi_{TF}[v^\sharp,w^\sharp]$ is a Hamiltonian vector field on $F$. Thus, we assign to 
$$a(w^\sharp,v^\sharp)$$ the value of the unique zero-average Hamiltonian associated to $\pi_{TF}[v^\sharp,w^\sharp]$. Such an assignment gives a closed form that satisfies the normalization condition \eqref{normalizationcondition}.

For large $K$, we refer to any symplectic form 
$$\omega_{H,K}:=a_H+K\pi^*\omega_B$$ as a \emph{weak coupling form}. Such a form is unique up to symplectic isotopy in the \emph{weak coupling limit}:

\begin{theorem}\cite[Thm 1.6.3]{guillemin}\label{uniquecouplingthm}
Assume $F$ is simply connected. Then for two symplectic connections $H_i$, $i=1,2$, the corresponding forms $a_{H_i}+K\pi^*\omega_B$ are isotopic for large enough $K$.
\end{theorem}

The proof of this theorem uses a Moser-type argument and the fact that $[a_{H_1}]=[a_{H_2}]$.

In Example \ref{flag3}, we used deformation and extension theorems in order to construct a fibered Lagrangian. For reference:

\begin{theorem}\cite[Thm 4.6.2]{guillemin}\label{connectionextensiontheorem}
Let $A\subset B$ be a compact set, $A\subset U$ an open neighbourhood, $H'$ a symplectic connection for $\pi^{-1}(U)$. Then there is an open subset $A\subset U'\subset U$ and connection $H$ on $E$ such that $H=H'$ over $U'$.
\end{theorem}

Often it will be clear if one can construct a trivially fibered Lagrangian from a combination of Theorems \ref{connectionextensiontheorem} and \ref{uniquecouplingthm}. A useful criterion for actually checking that a fiber sub-bundle is Lagrangian is the following:

\begin{lemma}{(Fibered Lagrangian Construction Lemma)}(Lemma 2.1 from \cite{floerfibrations})\label{fiberedlagrangianlemma}
Let $L_F\rightarrow L\rightarrow L_B$ be a connected sub-bundle. Then $L$ is Lagrangian with respect to $a+K\pi^*\omega_B$ if and only if

\begin{enumerate}
\item\label{paralleltransinv} $L$ is invariant under parallel transport along $L_B$ and
\item\label{normalizing} there is a point $p\in L_B$ such that $a_p\vert_{TL_F\oplus H_L}=0$, where $H_L:=H_a\cap TL$ is the connection restricted to $L$.
\end{enumerate}

\end{lemma}
Since $L$ is a lift of a Lagrangian for $\omega_B$ $a\vert_L$ must be zero.

\subsection{Adapted Morse functions and pseudo-gradients}\label{pseudogradients} 

An approach to Morse theory that is natural to use in the fibration setting is that of a pseudo-gradient, as in \cite[Sec 6.3]{hutchings} c.f. \cite[Sec 5]{oancea2}.
\begin{definition}
Let $f$ be a Morse function on a Riemannian manifold $(M,G)$. A \emph{pseudo-gradient} for $f$ is a vector field $X$ such that $X_p=0$ for $p\in \text{crit}(f)$ and
\begin{enumerate}
\item \label{par} $G(\grad_G(f),X)\leq 0$ and equality holds only at critical points of $f$

\item \label{local} In a Morse chart for $f$ centered at $p\in \text{Crit}(f)$, $\grad_ef=X$ where $e$ is the standard Euclidean metric.
\end{enumerate}
\end{definition}

Morse theory for pseudo-gradients is carried out in Chapter 2 \cite[Ch 2]{audin}.

To construct a pseudo-gradient on a fibered Lagrangian $L$, we use the connection $H_L:=TL\cap H$ (which is merely convenient: any connection would do). The Morse input for $L$ consists of the following:

\begin{enumerate}
\item A Morse function $b$ on $L_B$ and a Riemannian metric $G_B$ that is Euclidean in a neighbourhood of each $x\in \crit(b)$, and a pseudo-gradient $X_b$ for $b$,

\item for each critical fiber $\pi^{-1}(x)\cap L$ a choice of Morse function $g_x$ on $L_F$ and a Riemannian metric $G_{F_x}$ which is Euclidean in a neighbourhood of $\crit (g)$, and a pseudo-gradient $X_{g_x}$ on $TL_{F_x}$, and

\item an extension $X_g$ of the $\lbrace X_{g_x}\rbrace $ to all of $TF$ (by zero if necessary)
\end{enumerate}

One can show that the vector field 
$$X_f:=X_g\oplus X_b^\sharp$$ on $TL_F\oplus H_L$ forms a pseudo-gradient for the Morse function 

$$f:=\pi^*b+\epsilon\sum_{x\in \crit(b)}\pi^*\phi_x g_x$$
for $\epsilon <<1$ where $\phi_x$ is an appropriate cutoff function on $L_B$. 

Let $\gamma_X(t,x)$ denote the Morse flow of $X$ starting at $x\in L$. Define the \emph{stable} (+) resp. \emph{unstable} (-) manifold of $x_i$ as
$$W^\pm_X(x_i):=\left\{ x\vert \lim_{t\rightarrow \pm\infty}\gamma_X(t,x)=x_i \right\}.$$ 
Define the \emph{index} of a critical point
$$\I(x_i):=\dim W^-_X(x_i)$$
and the \emph{coindex}
$$co\I(x_i):=\dim W^+_X(x_i).$$

By perturbing the pseudo-gradient $X_f$ outside of a neighbourhood of the critical points, one can achieve the \emph{Smale condition} for $X_f$, which says that the unstable and stable manifold intersect transversely. Hence, $W^+_{X}(x_i)\cap W^-_X(x_j)$ is a smooth manifold of dimension
$$\I(x_j)-\I(x_i).$$

\section{Basic Floer theoretic notions in symplectic fibrations} \label{floerbackgroundsection}

Since the base manifold is rational, we adopt an approach similar to that of \cite{CM} and \cite{CW2} to achieve transversality using domain-dependent perturbations of the almost complex structure. The scheme uses complicated domains,  which are made even more complicated by the requirement that holomorphic curves have to be compatible with the fibration structure. The complete version is spelled out in the first part of the author's thesis \cite[Secs 3,4,5]{floerfibrations}, and we include an abridged version.
\subsection{Moduli of treed disks}\label{treeddisksection}

For a tree $T$, let $\text{Edge}(T)$ resp. $\text{Vert}(T)$ denote the set of edges resp. vertices. To define treed disks, we start with a metric tree $\mc{C}=(T,\ell)$ with $\ell\in[0,\infty]$. For a symplectic manifold $E$ and a Lagrangian $L$, we partition the vertices into two sets and label them as either spherical or disk vertices, and assign maps
\begin{gather*}
[\cdot]_d:\text{Vert}_d(\mc{C})\rightarrow H^{d}_2(E,L,\bb{Z})\\
[\cdot]_s:\text{Vert}_s(\mc{C})\rightarrow H^{s}_2(E,\bb{Z})
\end{gather*}
where $H^{d}_2(E,L,\bb{R})$ resp. $H^{s}_2(E,\bb{R})$ are the homology classes that can be represented by maps of disks resp. spheres. The edges correspond to either interior markings which map to the divisor, interior or disk nodes, or domains of Morse flows. We have the following chart from \cite{floerfibrations} that sums up the labelling, with an example schematic as Figure \ref{treediskpicture}:

\begin{center}
\begin{tabular}{||c | c ||}
\hline
$\text{Vert}_s(\mc{C})$ & spherical vertices (matched with a $\bb{P}^1_v$)\\
\hline
$\text{Vert}_d(\mc{C})$ & disk vertices (matched with a $D_v$)\\
\hline
$\text{Edge}_\rightarrow^\bullet (\mc{C})$ & interior markings \\
\hline
$\text{Edge}_-^\bullet(\mc{C})$ & interior nodes \\
\hline
$\text{Edge}_\rightarrow^\circ (\mc{C})$ & boundary markings \\
\hline
$\text{Edge}_-^\circ(\mc{C})$ & boundary nodes \\
\hline
$\ell:\text{Edge}_-^\circ (\mc{C})\rightarrow [0,\infty]$& boundary node length\\
\hline
$m:\text{Edge}_\rightarrow^\bullet(\mc{C})\rightarrow \bb{Z}_{\geq 0}$ & divisor intersection multiplicity\\
\hline
$[\cdot]_d:\text{Vert}_d(\mc{C})\rightarrow H^d_2(E,L)$ & relative disk/sphere classes\\
\hline
$[\cdot]_s:\text{Vert}_s(\mc{C})\rightarrow H^s_2(E)$ & sphere classes\\
\hline
\end{tabular}
\end{center}
We further require that there are no boundary edges $\text{Edge}^\circ(\mc{C})$ incident at spherical vertices. Let the notation $\text{Edge}_-^{\bullet,\ell}(\mc{C})$ resp. $\text{Edge}_-^{\circ,\ell}(\mc{C})$ denote the set of interior resp. boundary nodes of length $\ell$.

\begin{definition}\cite{floerfibrations} \label{treeddiskdef}A \emph{treed disk} $\mc{C}$ is a triple $(T,\mc{S}_T,\sim)$ consisting of
\begin{enumerate}
\item\label{tree} a labelled metric tree $(T,\ell)$,
\item for each vertex $v$ a class of a marked surface $(\mc{S}_v,\underline{x},\underline{z}) \in \mc{S}_T$ with interior resp. boundary markings denoted $\underline{z}$ resp. $\underline{x}$, where each is identified with a marked surface class $[\bb{P}^1,\underline{z}]/\sim$ or $[D\subset \bb{C},\underline{z},\underline{x}]/\sim$ up to biholomorphism fixing the markings, and
\item\label{specialpoints} a bijective identification of the interior special points $\underline{z}$ with\linebreak $\text{Edge}_\rightarrow^\bullet(v)\cup\text{Edge}_-^\bullet(v)$, together with
\item an identification between the boundary special points $\underline{x}$ and\linebreak $\text{Edge}_\rightarrow^\circ(v)\cup\text{Edge}_-^\circ (v)$ such that the natural relative clockwise ordering of the $\underline{x}$ around $\partial D$ agrees with the relative clockwise ordering of the identified edges in $\text{Edge}_\rightarrow^\circ(v)\cup\text{Edge}_-^\circ (v)$ around the vertex.
\end{enumerate}
\end{definition}
\begin{figure}[h]\label{treediskpicture}
\includegraphics[scale=1]{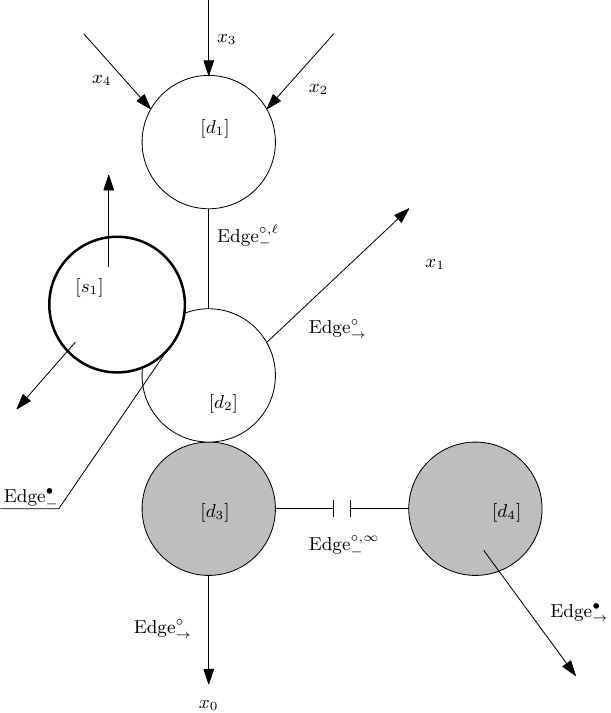}
\caption{The geometric realization of a (stable) treed disk. Shadings are the binary markings discussed in Section \ref{pistabilitysection}.}
\end{figure}
We call a treed disk \emph{broken} if some $e\in \text{Edge}^{\circ}_{-}(\mc{C})$ has infinite length. We identify a broken type with the union of two types glued at their extraneous ends $e_i\in \text{Edge}^{\circ}_\rightarrow(\mc{C}_i)$.
The treed disks with $n$ boundary markings and $m$ interior markings form a moduli space $\mf{M}_{n,m}$. Such a cell-complex is stratified by \emph{combinatorial type}, denoted $\g(\mc{C})$ or simply $\g$, which is the information contained in the above chart with reference to $\ell$ only if $\ell(e)=0,\infty$. We denote a layer of this stratification by $\mf{M}_\g$.

The \emph{geometric realization class} $U$ of $\mc{C}$ is the class of topological space given by replacing the vertices with their corresponding marked surface class and attaching the edges to the corresponding special points, together with the Riemann surface structure on each surface component. The \emph{universal treed disk} is the fiber bundle $\mc{U}_{m,n}\rightarrow\mf{M}_{m,n}$ whose fiber at a point $p$ is the geometric realization class of $p$. Similarly, denote the restriction of the universal treed disk to $\mf{M}_\g$ by $\mc{U}_\g$.

A treed disk is \emph{stable} if and only if each marked surface is stable. In general, we will not require such stability; rather, we later develop a finer notion of stability with the fibration structure in mind .

Since we have a pseudo-gradient, we can label combinatorial types with critical points. Suppose $\lbrace x_0,\dots, x_n\rbrace$ are the critical points of a pseudo-gradient $X$. We define a \emph{Morse labelling} of a combinatorial type $\g$ as a labelling
$$\underline{x}:\text{Edge}^\circ_\rightarrow (\g)\rightarrow\lbrace x_0,\dots, x_n\rbrace $$

Our convention will be that $x_0$ denotes the value of $\underline{x}$ on the root and $x_i$ $i\geq 1$ denotes the value on the leaves.

\subsubsection{Special considerations for the fibered setting}\label{pistabilitysection}

\begin{definition} A \emph{binary marking} $\varrho$ for $\g$ is a subset of the vertices and edges, denoted $m\text{Vert}(\g)$ and $m\text{Edge}(\g)$, for which any map $u: C_\g\rightarrow E$ is required to map the domain for $mv\in m\text{Vert}(\g)$ resp. $me\in m\text{Edge}(\g)$ to a constant under $\pi$. The set of unmarked vertices and edges will be denoted $u\text{Vert}(\g)$ resp. $u\text{Edge}(\g)$. 
\end{definition}

The marked vertices and edges are those that correspond to domains that are mapped to a constant under $\pi$. Such a process can be decomposed into two elementary operations, the first being:

\begin{definition}\cite{floerfibrations}
The combinatorial type $\pi_*\g$ is the combinatorial type $\g$ with the following two modifications:
\begin{enumerate}
\item The homology labelling $[\cdot]$ is replaced replaced by $\pi_*[\cdot]$ 
and,
\item $\ell_{\pi_*\g}(e)=0$ for any $e\in m\text{Edge}^{\circ}_-(\g)$
\end{enumerate}
\end{definition}
followed by:

\begin{definition}\cite{floerfibrations}
The \emph{$\pi$-stabilization map} $\g\mapsto\Upsilon(\g)$ is defined on combinatorial types by forgetting any unstable vertex $v_i$ for which $[v_i]=0$ and identifying edges as follows:
\begin{enumerate}
\item If $v_i$ has two incident unmarked edges, $e_i$ and $f_i$, with $f_i$ closer to the root, then $\text{Vert}(\Upsilon(\g))=\text{Vert}(\g)-\lbrace v_i\rbrace$ and we identify the edges $e_i$ and $f_i$: $$\text{Edge}(\Upsilon(\g))=\text{Edge}(\g)/\lbrace f_i \sim e_i\rbrace$$ and $$\ell(\Upsilon(f_i))=\ell(\Upsilon(e_i))=\ell(e_i)+\ell(f_i).$$
We set $\ell(\Upsilon(e_i))=\infty$ if either $e_i$ or $f_i$ is a semi-infinite edge (hence, so is $\Upsilon(e_i)$ under the broken tree identification).

\item If $v_i$ has one incident edge $e_i$, then $\Upsilon(\g)$ has vertices $$\text{Vert}(\Upsilon(\g))=\text{Vert}(\g)- \lbrace v_i\rbrace$$ and edges 
$$\text{Edge}(\g)-\lbrace e_i\rbrace$$
\end{enumerate}
\end{definition}

$\Upsilon\circ\pi_*$ forgets marked edges and vertices and stabilizes, so it is the combinatorial type of $\pi\circ u$ for $u:U_\g\rightarrow \g$. 

We get an induced map on moduli spaces resp. universal marked disks, denoted 
$$\Omega_\g:\mf{M}_\g\rightarrow \mf{M}_{\Upsilon(\pi_*\g)}$$
resp.
$$\Omega_\g:\mc{U}_\g\rightarrow \mc{U}_{\Upsilon(\pi_*\g)}$$

\begin{definition}
A combinatorial type $\g$ is called \emph{$\pi$-stable} if $\Upsilon(\pi_*\g)$ is stable.
\end{definition}

\subsection{Perturbation data}\label{perturbationsection}
In order to precisely define a Floer configuration, we discuss the types of perturbations that we use. For a rational symplectic manifold and Lagrangian $(B,\omega_B,L_B)$, let $D_B$ denote a symplectic hypersurface in the complement of $L_B$.

\begin{definition}\label{stabilizingdivisordefinition} 
\begin{enumerate}
\item A tamed almost complex structure $J_{D_B}\in \mc{J}(B,\omega_B)$ is \emph{adapted} to $D_B$ if $D_B$ is an almost complex submanifold of $(B,J_B)$.

\item A divisor $D_B$ in the complement of $L_B$ with $[D_B]=k[\omega_B]$ is \emph{stabilizing} for $L_B$ if any map $u:D\rightarrow (B,L_B)$ with $\int_D u^*\omega_B> 0$ has as least one intersection with $D_B$.

\end{enumerate}
\end{definition}

Charest and Woodward develop a Lagrangian Floer theory for rational symplectic manifolds and Lagrangians, inspired by the stabilizing divisor approach of \cite{CM}. To summarize, we have

\begin{theorem}\cite[Sec 4]{CW1}
There exists a divisor $D_B \subset B-L_B$ that is stabilizing for L representing $k[\omega_B]$ for some integer $k>>1$.
\end{theorem}
When there is no confusing with the notation for ``disk'', we will denote $D:=\pi^{-1}(D_B)$ as the pullback of a stabilizing divisor.
\subsection{Hamiltonian perturbations}\label{hamperturbsection}
Let $\ham(F,B):=\Lambda^1 (B,C^\infty_0(F_p,\bb{R}))$ denote the space of one forms $\sigma$ on $B$ with values in $C^\infty (E,\bb{R})$ such that
$$\int_{F_p} \iota(v)\sigma \omega_{F_p}^{\dim F/2}=0$$
for every fiber $F_b$ of $E$ and $v\in T_p B$ (i.e. Hamiltonians with zero fiber-wise average). In \cite{floerfibrations} we achieve transversality for \emph{non-vertical disks} (disks not contained in a single fiber) by the theory of Hamiltonian perturbation. That is, to every Hamiltonian one-form $\sigma\in \ham (F,B)$, we get another connection form
$$a_\sigma:=a-d\sigma$$
which is a minimal coupling form by the zero-average condition. The new connection $H_\sigma:=TF^{\perp a_\sigma}$ has holonomy that differs infinitesimally from $H_0$ by $\sigma$.

We have a fiber bundle $\mc{J}^{\text{vert}}\rightarrow B$ where the fiber at $p$ is the space of (sufficiently differentiable) tamed almost complex structures on $(F_p,\omega_{F_p})$. Let $\mc{J}^{l,\text{vert}}(B,F)$ be the space of $C^l$ sections of this fiber bundle. To make sense of holomorphic lifts, we have the following lemma:

\begin{lemma}\cite[eq 8.2.8]{ms2}\label{connectionacs}
To a connection $H_\sigma$ and an assignment of vertical almost complex structures $J_F\in\mc{J}^{\text{vert}}(B,F)$ there is a unique almost complex structure $J_\sigma$ that:
\begin{enumerate}
\item makes $\pi:E\rightarrow B$ holomorphic,
\item preserves $H_\sigma$, and
\item agrees with the section $J_F$.
\end{enumerate}
\end{lemma}
Such a structure in the $TF\oplus H_0$ splitting is given by
\begin{equation}\label{acsontotalspace}
J_\sigma=
\begin{bmatrix}
J_F & J_F\circ X_\sigma-X_\sigma\circ J_B\\
0 & J_B
\end{bmatrix}
\end{equation}
where $v\mapsto X_{\sigma_v}$ is the one-form with values in fiberwise Hamiltonian vector fields, prescribed by $d\sigma_v=\iota_{X_{\sigma_v}}\omega_F$.

For each critical fiber $F_i$ of the Morse function $\pi^*b$, choose a tamed almost complex structure $J_F^i$, and let 
$$\mc{J}^{l,vert}(J_F^1,\dots,J_F^k):=\left\{J_F\in \mc{J}^{l,\text{vert}}(B,F): J_F\vert_{F_i}=J_F^i\right\}$$ be the space of sections of $\mc{J}^{l,\text{vert}}\rightarrow B$ that agree with the chosen $J_F^i$. Evaluation of the a.c.s. at each fiber is a regular map, so such a space is a Banach manifold. To save space take the notation $\underline{J_F}$ to denote the choice of tamed almost complex structures on the critical fibers.
Let 
\begin{align}\label{totalspaceperturbationdata}
\mc{JH}^l(E,\omega_{H,K}):= \bigg\lbrace& (J_\sigma,\sigma)\in \mc{J}^{l,vert}(\underline{J_F})\times \mc{J}^l(B,\omega_B) \times \ham (F,B):\\ 
&J_\sigma \text{ tames }a_\sigma+K\pi^*\omega_B\bigg\rbrace .
\end{align}
The latter is an open condition about $\sigma=0$ and tamed almost complex structures $J_B$ of the base assuming that $K>>1$ by an elementary estimate \cite[Lem 5.1, Rem 5.11]{floerfibrations}. Thus, $\mc{JH}^l$ is a Banach manifold.
Choose neighbourhoods $\mc{S}_\g^o\subset \Upsilon\pi_*\mc{U}_\g$ resp. $\mc{T}_\g^o\subset \mc{U}_\g$ not containing the special points and boundary on each surface component resp. not containing $\infty$ on each edge.

\begin{definition}\cite{floerfibrations}\label{perturbationdatadef} Fix a stabilizing $J_{D_B}$ for $L_B$. Let $(\g,\underline{x})$ be a $\pi$-stable combinatorial type of treed disk with a Morse labelling. A class $C^l$ \emph{fibered perturbation datum} for the type $\g$ is a choice of $\underline{J_F}$ together with piecewise $C^l$ maps 
\begin{gather*}
(J_\sigma,\sigma):\Upsilon\pi_*\mathcal{U}_\g\rightarrow \mathcal{JH}^l(E,\omega_{H,K}),\\
X:\mc{T}_\g\rightarrow \text{Vect}^l(TL_F)\oplus \text{Vect}^l(H_L)
\end{gather*}
denoted $P_\g$, that satisfy the following properties:
\begin{enumerate}
\item $\sigma\equiv 0$ and $J_B\equiv J_{D_B}$ on the neighbourhoods $\mc{S}_\g - \mc{S}_\g^o$,

\item $\sigma\equiv 0$, $J_B\equiv J_{D_B}$ on $\Upsilon\pi_*\mc{T}_\g$, and

\item $X\equiv X_f$ in the neighbourhood $\mc{T}_\g-\mc{T}_\g^o$ of $\infty$.
\end{enumerate}

Let $\mc{P}^l_\g(E,D)$ denote the Banach manifold of all class $C^l$ fibered perturbation data (for a fixed $J_{D_B}$).
\end{definition}

Any perturbation datum $P_\g$ lifts to the universal curve via pullback $\tilde{P}_\g:=(\Upsilon\pi_*)^* P_\g$. As such, we will only distinguish between $\tilde{P}_\g$ and $P_\g$ when necessary.

\begin{definition}
A \emph{perturbation data} for a collection of combinatorial types $\gamma$ is a family $\underline{P}:=(P_\g)_{\g\in\gamma}$
\end{definition}

\subsection{Holomorphic treed disks}
\begin{definition}\cite{floerfibrations} \label{holotreeddisk} Given a fibered perturbation datum $P_\g$, a $P_\g$-\linebreak \emph{holomorphic configuration} of type $(\g,\underline{x})$ in $E$ with boundary in $L$ consists of a choice of representative $C=S\cup T$ of a fiber of the universal treed disk $\mathcal{U}_\g$, together continuous map $u:C\rightarrow E$ such that
\begin{enumerate}
\item $u(\partial S\cup T)\subset L$.
\item On the surface part $S$ of $C$ the map $u$ is $J_\sigma$-holomorphic for the given perturbation datum: If $j$ denotes the complex structure on $S$, then
\begin{equation*}
J_{\sigma}\mathrm{d}u|_S = \mathrm{d}u|_S j.
\end{equation*}
\item On the tree part $T\subset C$ the map $u$ is a collection of flows for the perturbation vector field:
\begin{equation*}
\frac{d}{ds}u|_T = X(u|_T)
\end{equation*}
where $s$ is a local coordinate with unit speed so that for every $e\in\text{Edge}^\circ(\g)$, we have $e\cong [0,\ell(e)]$, $e\cong [0,\infty)$, or $e_0\cong (-\infty,0]$ via $s$.
\item If $e_i$ $i\geq 1$ denote the leaves of $\g$ and $e_0$ the root,
\begin{gather*}
u(e_i)\in W^-_X(x_i)\\
u(e_0)\in W^+_X(x_0).
\end{gather*}

\item \label{multintdef} The map $u$ takes the prescribed multiplicity of intersection $m_{z_e}(u)$ to $D$ for $z_e$ corresponding to $e\in \text{Edge}^\bullet_\rightarrow (\g)$. That is, let $\gamma$ be a small loop around $z_e$: As in \cite{CW1} $m_{z_e}(u)$ is the winding number of $u(\gamma)$ in the complement of the zero section in a tubular neighbourhood of $D$ and~$u$.

\end{enumerate}
\end{definition}

Proposition 7.1 from \cite{CM} shows that \eqref{multintdef} is related to the ``order of tangency'': In a coordinate neighbourhood that sends $D$ to $\bb{C}^{n-1}\subset \bb{C}^n$ define the order of tangency of $u$ to $D$ at $z_e$ to be
\begin{equation}\label{orderoftangencydef}
tan_{z_e}(u)=\max_l d^l_{z_e} u\in T\bb{C}^{n-1}.
\end{equation}
By the aforementioned proposition, we have $tan_{z_e}(u)+1=m_{z_e}(u)$ for a holomorphic map. In particular, $m_{z_e}(u)\geq 0$.

\section{Transversality and compactness results}\label{floerfibrationsresults}
In this section, we review the basic compactness and transversality results that were shown in \cite{floerfibrations} for the fibered setting expanding upon results of \cite{CM} \cite{CW1} \cite{CW2}. These results are essential to defining the $A_\infty$ algebra and computing the potential later on.

\subsection{Transversality}
In order to achieve transversality for the aforementioned configurations, we allowed domain dependent Hamiltonian perturbations and almost complex structures in our definition of perturbation datum, as in \cite{CM}, \cite{CW1}, \cite{CW2}. However, in order to achieve compactness we need to select our perturbation data in a way that is compatible with certain actions on treed disks.

Actions on graphs $\tilde{\Pi}:\g\rightarrow \g'$, such as cutting edges, collapsing edges and identifying vertices, or changing the length of an edge away from $0$ or $\infty$ gives rise to maps on the moduli of combinatorial types $\Pi:\mf{M}_\g\rightarrow \mf{M}_{\g'}$ and also their universal disks $\Pi:\mc{U}_{\g}\rightarrow \mc{U}_{\g'}$. With the exception of the last axiom, the following is from \cite[Def 2.11]{CW2}.

\begin{definition}\label{coherentdefinition}
Let $\underline{P}$ be a family of fibered perturbation data. We say that $\underline{P}$ is \emph{coherent} if it is compatible with morphisms on the universal treed disk in the following sense:
\begin{enumerate}
\item\emph{(Cutting edges)} If $\Pi:\mc{U}_\g\rightarrow \mc{U}_\g'$ cuts an edge of infinite length then $P_{\g'}=\Pi_*P_\g$.

\item\emph{(Products)} Let $\g=\g_1\cup \g_2$ via cutting an infinite edge of $\g$ so that $\mf{M}_\g\cong \mf{M}_{\g_1}\times\mf{M}_{\g_2}$ with projections $\pi_i$ $i=1,2$. Then $P_\g$ agrees with the pullback of $P_{\g_i}$ on $\pi_i^* \mc{U}_{\g_i}$.

\item\emph{(Collapsing an edge or making edge length finite/non-zero)} If $\Pi:\mc{U}_\g\rightarrow \mc{U}_{\g'}$ collapses an edges or changes an edge length from $0$ or $\infty$, then $P_\g=\Pi^*P_{\g'}$.

\item\emph{(Ghost-marking independence)} If $\Pi:\mc{U}_\g\rightarrow \mc{U}_{\g'}$ forgets an interior\linebreak marking on a component with $[v]=0$, then $P_{\g}=\Pi^*P_{\g'}$.

\item\emph{($\pi$-stabilization)} Finally, if $\g$ is a $\pi$-stable type, then $P_\g=\Omega^*P_{\Upsilon\circ\pi_*\g}$.
\end{enumerate}

\end{definition}

There is an important notion of an adapted holomorphic configuration:
\begin{definition}\label{adapteddef}\cite{floerfibrations}
We say that a Floer trajectory $u:C_\g\rightarrow (E,L)$ is $\pi$-adapted to $D$ if
\begin{enumerate}
\item\emph{(Stable domain)} $\g$ is $\pi$-stable,

\item\emph{(Non-constant spheres)} each component of $C_\g$ that maps entirely to $D$ is constant, and

\item\emph{(Markings)} each interior marking $z_i$ maps to $D$ and each component $u^{-1}(D)$ contains an interior marking.
\end{enumerate}
\end{definition}

Let $\widetilde{\mc{M}}_\g(P_\g,\underline{x})$ be the moduli space of $P_\g$-holomorphic configurations on $(E,L)$ which are adapted to $D$ with limits $\underline{x}$, and
$$\mc{M}_\g(P_\g,\underline{x}):=\widetilde{\mc{M}}_\g(P_\g,\underline{x})/\sim$$
where $u\sim u'$ if there is a map $\phi:C_\g\rightarrow C_\g$ such that $u=u'\circ\phi$ that is comprised only of biholomorphisms of the surface components which fix the marked points.

We say that a perturbation datum $P_\g$ is \emph{regular} if $\mc{M}_\g(P_\g,\underline{x})$ can be given the structure of a smooth manifold of dimension
\begin{align}\label{indexdefinition}
&\I (\g,\underline{x}):= \mathrm{dim}W^+_{X}(x_0)-\sum_{i=1}^{n}\mathrm{dim}W^+_{X}(x_i) + \sum_{i=1}^{m} I(u_i)+n-2 \\
\label{indexdefinitionline2} &\quad-|\text{Edge}^{\circ,0}_-(\g)|-|\text{Edge}_-^{\circ,\infty}(\g)|-2|\text{Edge}^\bullet_-(\g)| -2\sum_{e\in \text{Edge}^{\bullet}_\rightarrow} m(e).
\end{align}

The $I(u_i)$ signifies the either the Maslov index or $2c_1(u_i^*TE)$ for a single disk/sphere component $u_i$. By the Riemann-Roch theorem, $I(u_i)+\dim L-3$ is the index of the linearized Cauchy-Riemann operator that describes the space of unparametrized holomorphic Maslov index $I(u_i)$ disks (or spheres) with boundary in $L$, see for example \cite[Thm C.1.1.10]{ms2}. The other terms arise from expected codimension counts given by the assumption that each condition is cut-out transversely.

Finally, we say that a combinatorial type $\g$ is \emph{uncrowded} if each vertex with $[v]=0$ has at most one interior marking.

We have the following theorem about the existence of regular perturbation data to partially legitimize the counting of holomorphic configurations.

\begin{theorem}[Transversality, \cite{floerfibrations}; Theorem 6.1]\label{transversality}
Let $E$ be a symplectic (K\"ahler) fibration and $L$ a fibered Lagrangian. Let $\gamma$ be a finite indexing set and suppose we have a finite collection of uncrowded, $\pi$-adapted, and possibly broken types $\lbrace\g\rbrace_{\g\in \gamma}$ with $$\I(\g,\underline{x})\leq 1.$$ 
Then there is a comeager subset of smooth regular data for each type $$\mc{P}^{\infty,reg}_\g(E,D)\subset\mc{P}^\infty_\g$$
and a selection in the product space 
\begin{equation*}(P_\g)_\gamma\in \underset{\g\in \gamma}{\Pi} \mc{P}^{\infty,reg}_\g
\end{equation*} that forms a regular, coherent datum. Moreover, we have the following results about tubular neighbourhoods and orientations:

\begin{enumerate}
\item\label{tubularneighborhoods}\emph{(Gluing)} If $\Pi:\g\rightarrow\g'$ collapses an edge or makes an edge finite/non-zero, then there is an embedding of a tubular neighbourhood of\linebreak $\mc{M}_\g (P_\g,\underline{x})$ into $\overline{\mc{M}}_{\g'}(P_{\g'},\underline{x})$, and
\item \emph{(Orientations)} if $\Pi:\g\rightarrow\g'$ is as in \eqref{tubularneighborhoods} and $L$ is relatively spin, then the inclusion $\mathcal{M}_{\g} (P_\g,\underline{x})\rightarrow \overline{\mathcal{M}}_{\g'} (P_{\g'},\underline{x})$ gives an orientation on $\mc{M}_\g$ after choosing an orientation for $\mc{M}_{\g'}$ and the outward normal direction on the boundary. 
\end{enumerate}

\end{theorem}

We give a short sketch of the proof, with the appropriate references: One shows that the \emph{universal moduli space} (the moduli space of all holomorphic configurations of type $\g$ for every perturbation data) is a smooth manifold when we allow domain dependent almost complex structures in the base and we are allowed to vary the connection over the interior of each surface component. Transversality for surface components that are contained in a single fiber is handled by the usual argument for monotone Lagrangians as in \cite{oh}. A Sard-Smale argument then shows that for a given configuration type $\g$ there is a comeager set of perturbations $J_\g$ for which the linearized $\bar{\partial}_{j,J_\g}$ operator is surjective, thus making the moduli space associated to such a perturbation datum into a smooth manifold of expected dimension. A selection from each comeager set is then pieced together to form a coherent system for low index types. 

The usual gluing argument is used to show \eqref{tubularneighborhoods}, where one constructs approximate solutions $\tilde{u}_\delta$ based on $\g'$ associated to a pre-gluing parameter $\delta$, and one shows that the linearized operator at $\tilde{u}_\delta$ is surjective given that it is surjective at an index $0$ curve $u$. Then, after showing some estimates as in \cite[Thm 4.1]{CW1} one applies the quantitative version of the Implicit Function Theorem \cite[Prop A3.2]{ms2} to obtain a unique holomorphic configuration $u_\delta$ for each $\delta$ which converges to $u$ in the sense of Gromov.

Finally, orientations on $\mc{M}_\g(P_\g,\underline{x})$ can be realized after choosing a relative spin structure as in \cite[Thm 8.1.1]{fooo}, and such an orientation only depends on the choice of spin structure on the orientation on $\mf{M}_\g$. As in the remark \cite[Rem 4.25]{CW2} orientations on $\mc{M}_\g(P_\g,\underline{x})$ for an index $0$ type can be realized in the following way: One chooses an orientation on each $W^\pm_X(x_i)$ so that we have an orientation preserving isomorphism
$$\det (T_{x_i}W^+_X(x_i)\oplus T_{x_i}W^-_f(x_i))\cong \det T_{x_i}L.$$ Moreover, one can orient $\mf{M}_{m,n}$ in a consistent way which induces orientation on the strata $\mf{M}_\g$. For the linearized operator $D_u$ at some holomorphic configuration $u$ we have an isomorphism of determinant lines
\begin{equation}\label{orientationiso}
\det D_u\cong \det \mf{M}_\g\otimes \det TL\otimes \det TW^+_X(x_0)\otimes \bigotimes_{i=1}^n \det TW^-_X(x_i)
\end{equation}
\cite[45]{CW2} after choosing a relative spin structure on $L$ and applying a trivialization argument (as in \cite[Thm 8.1.1]{fooo}). A coherent choice of orientations thus depends on a coherent choices on $\mf{M}_{\g}$ and $\mf{M}_{\g'}$. The former can be induced via a choice on $\mf{M}_{\g'}$ and the opposite choice on the stratum $\mf{M}_{\g''}$ where $\g''$ is obtained from $\g$ by the complimentary morphism to $\Pi$ (i.e. collapsing an edge or making an edge length non-zero).

\subsection{Compactness}
The transversality result together with the coherence conditions pave the way for a compactness result, which we review in this section.

We say that an almost complex structure $J$ of the form \ref{acsontotalspace} is \emph{adapted} to $D$ if $D_B$ is an almost complex submanifold with respect to $J_B$. ``Pulling back'' the definition from \cite{CW2}:
\begin{definition}\label{stabilizedbyDdef} For a $\pi$-stabilizing divisor $D=\pi^{-1}(D_B)$ we say that an adapted almost complex structure $J_{D_B}$
is \emph{$\varrho$-stabilized} by $D$ if
\begin{enumerate}
\item $D_B$ contains no non-constant $J_{D_B}$-holomorphic spheres of energy less than $\varrho$,
\item each non-constant $J_{D_B}$-holomorphic sphere $u: S^2\rightarrow B$ with energy less than $\varrho$ has $\# u^{-1}(D_B)\geq 3$
\item every non-constant $J_{D_B}$-holomorphic disk $u:(D,\partial D)\rightarrow (B,L_B)$ with energy less than $\varrho$ has $\# u^{-1}(D_B)\geq 1$.
\end{enumerate}
\end{definition}

One can show \cite[Cor 8.14, Prop 8.12]{CM} that the space of $\varrho$-stabilized almost complex structures forms an open and dense set for each energy. Thus, we can require a family of perturbation data $(P_\gamma)$ to satisfy the conclusion of Theorem \ref{transversality} as well as to take values in such a set. In particular, $\pi_*\g$ has a well defined energy by homology class, so we can take $\varrho> e_B( [\pi_*\g])$. In general we will take $P_\g$ to be $n(\g)$-stabilized where $n(\g)>e_B( [\pi_*\g])$.

Define the energy of a map $u$ of type $\g$ based on a treed disk $C_\g$ as
$$e(u):=\int_{C_\g}u^*\omega_{H,K}.$$
By a general theorem, $e(u)$ only depends on $\g$ when $u$ is holomorphic.

With these notions, we have the following theorem:

\begin{theorem}[Compactness, \cite{floerfibrations}; Theorem 7.1]\label{compactnessthm}
Let $\g$ be an uncrowded, unbroken type with $\I(\g,\underline{x})\leq 1$ and let $\mc{P}=(P_\Xi)_{\Xi\in\gamma}$ be a collection of coherent, regular, $\pi$-stabilized fibered perturbation data that contains data for all types $\Xi$ from which $\g$ can be obtained by some combination of (collapsing an edge/making an edge finite or non-zero) or (forgetting a ghost component). Then the compactified moduli space $\overline{\mc{M}}_\g(P_\g,\underline{x})$ of adapted configurations contains only regular configurations $\gamma$ with broken edges and disk vertices. In particular, there are no sphere bubbles. Moreover, moduli corresponding to vertical disk bubbles come in pairs or can be given opposite orientations so that they cancel in an algebraic sense. 
\end{theorem}
The following is included from \cite{floerfibrations}:
\begin{proof}
Let $u_\nu\in\mc{M}_\g (P_\g,\underline{x})$ be a sequence of bounded energy adapted $P_\g$-holomorphic configurations each based on a treed disk $C_\nu$. By Gromov and Floer there is a convergent subsequence to a configuration $u:\tilde{C}\rightarrow (E,L)$. By forgetting unstable ghost components we obtain a curve $u:C\rightarrow (E,L)$ based on a type $\Xi$ that is related to $\g$ via the assumed operations on treed disks, and thus $u$ is $P_\Xi:=\Pi^* P_\g$-holomorphic. We will show that $\Xi$ is $\pi$-stable, that $u$ is $\pi$-adapted in the sense of Definition \ref{adapteddef}, and that $\Xi$ lacks sphere and vertical disk bubbles.

First, we have the (markings) property in Definition \ref{adapteddef}. Each interior marking on $\Xi$ maps to $D$ since this is a closed condition. On the other hand, the fact that every intersection contains a marking follows from topological invariance of intersection number; see the proof of \cite[Thm 4.27]{CW2}.

Next, we show that $\Xi$ is $\pi$-stable. Assume that there is an unstable vertex for $\Upsilon\pi_*\Xi$ with corresponding surface component $S$. It follows that $u_S$ is a bubble, is $J$-holomorphic for some $n(\g)$-stabilized almost complex structure, and has energy at most $n(\g)$. From the definition of $\pi$-stabilized data, the (markings) property, and the fact that $\Upsilon\pi_*$ does not forget edges that have positive intersection multiplicity with $D$, it follows that $S$ is actually stable. This is a contradiction.

To show the (non-constant spheres axiom), we notice that in the limit, any sphere bubble must have energy less than $n(\g)$. Since the almost complex structure takes values in stabilized data, it follows from item $(1)$ in Definition~\ref{stabilizedbyDdef} that any sphere component $S$ contained in $D$ must be contained in a single fiber. Thus, any sphere of $u$ mapping to $D$ is horizontally constant.

It follows that $\Xi$ is $\pi$-stable and that $u$ is $\pi$-adapted. Since $\g$ can be obtained from $\Xi$ by the assumed operations on treed disks, Theorem \ref{transversality} implies that (after restricting to a smaller set of perturbation data $P_\g$) the pullback perturbation data $P_\Xi$ is regular for $\Xi$. It follows that $u:C\rightarrow (E,L)$ lives in a smooth moduli space $\mc{M}_\Xi$ of expected dimension.

Now, we show the absence of sphere bubbles and vertical disk bubbles. Any non-constant sphere bubble gives a configuration of expected dimension two less than $\g$. By regularity of the type $\Xi$, this contradicts the index assumption $\I(\g,\underline{x})\leq 1$.

Suppose there is a marked disk bubble $\underline{D}$ attached to a component $C$. It suffices to check the case when $\I (\Xi,\underline{x})=0$, since otherwise $u$ will be contained in an open set in a moduli space of dimension $1$ and will not effect the properties of the $A_\infty$-algebra. We have that $\mu(u_{\underline{D}})\geq 2$ by assumption. Suppose $[u_C]\neq 0$. Then the configuration $\tilde{u}$ of type $\tilde{\Xi}$ with $\underline{D}$ removed is $\pi$-stable, $\pi$-adapted, and whose moduli space can be given the same perturbation data by the assumption that data is defined on $\Upsilon\pi_*\mc{U}_\g$. Thus, $\tilde{u}$ lives in a smooth moduli space of expected dimension. On the other hand the expected dimension is $\leq -1$, which shows that $\tilde{u}$ is non-existent. 

Finally, we analyze a vertical disk bubble when $C$ is a ghost component. If $u_C$ is not mapped to a critical point, then the configuration without $D$ is regular and of index $\leq -1$. By the same argument as in the above paragraph, this is a contradiction. Thus, $u_C$ must be mapped to a critical point $x_0$, and $u_{\underline{D}}$ is contained in a critical fiber. Since $u_C(\bar{z})$ is also holomorphic, we can give $C$ the opposite complex structure from $\bb{C}$ and form a distinct configuration. This gives the opposite orientation of the determinant line bundle of the linearized operator over $C$, which gives the opposite orientation of the moduli space $\mc{M}_\Xi$. Thus such configurations always cancel in pairs.
\end{proof}

\section{Lifting holomorphic trajectories}\label{liftingsection}
Here begins some of the main technical ideas leading to the main result. Notably, we give existence and some characteristics of lifts of $J_B$-holomorphic disks $v:(D,\partial)\rightarrow (B,L_B)$ to the total space $(E,L)$. 

\subsection{Hamiltonian holonomy and covariant energy}
In this section, we give background following \cite[Sec 1.4]{gromov} \cite{salamonakveld} \cite[Ch 8]{ms2} to state a known important formula that expresses the $L^2$ norm of the covariant derivative of a holomorphic section in terms of a semi-topological quantity.

The pullback bundle $v^*E$ over a (holomorphic) disk is topologically trivial and has an induced non-trivial connection via the pullback coupling form $v^*a_\sigma$. Rather than keep track of the connection for $a$ \emph{and} $\sigma$ we absorb this information into a single Hamiltonian connection form $\delta\in \Lambda^1(D,C^\infty_0(F,\bb{R}))$ so that
\begin{equation}\label{deltadefinition}
\omega_F-d\delta=v^*(a_\sigma).
\end{equation}
where the connection form is exact by the Hamiltonian assumption and by abuse of notation we mean $\omega_F:=\pi_{F}^*\omega_F$.

We get an associated connection
$$Tv^*E\cong TF\oplus TF^{\perp v^*a_\sigma}=:TF\oplus H_\delta$$
along with a pullback almost complex structure $$J_\delta:=v^*J_\sigma=\begin{bmatrix}
J_F & J_F\circ X_\delta-X_\delta \circ j \\
0 & j
\end{bmatrix}$$
in the $TF\oplus TD$ coordinates. Note that this a.c.s. satisfies the uniqueness criterion from Section \ref{hamperturbsection}.

By Lemma \ref{fiberedlagrangianlemma} and the vanishing of the a.c.s. perturbation $\sigma$ \eqref{perturbationdatadef} on $\partial D$, the \emph{Lagrangian boundary} $$v^*L:= v\vert_{\partial D}^*L\rightarrow \partial D$$ can be viewed as a path of Lagrangians $L_t$ parametrized by $[0,2\pi]$ that becomes a loop if we mod out by $\mathrm{Diff}(L_t)$ (see \cite{salamonakveld} for a complete viewpoint of this).

We consider $J_\delta$-holomorphic sections $u$ of $v^*E$ with the boundary condition
$$u(\partial D)\subset v^*L.$$
There is a clear correspondence between such sections and $J_\sigma$-holomorphic lifts of $v$ with values in $(E,L)$.

\begin{remark}\label{hamiltonianperturbremark}There is a further correspondence between $J_\delta$-holomorphic sections of $v^* E$ with Lagrangian boundary and smooth maps $u:D\rightarrow (F)$ that satisfy
\begin{equation}\label{perturbeddbarequation}
\bar{\partial}_{J_F,j}u+X_\delta^{0,1}(u)=0
\end{equation}
and the \emph{moving Lagrangian boundary} condition 
\begin{equation}\label{perturbedboundarycond}
u(e^{i\theta})\in L_{\theta}.
\end{equation}

To see the correspondence, let $\tilde{u}=z\times u(z):D\rightarrow D\times F$ be the $J_\delta$-holomorphic section. We have
$$d\tilde{u}\circ j=j\oplus du\circ j=j\oplus [J_F\circ du+J_F\circ X_\delta-X_\delta\circ j]=J_\delta\circ d\tilde{u}.$$
\end{remark}
\subsubsection{Curvature}
There is a corresponding energy identity for solutions to the Hamiltonian perturbed $\bar{\partial}$ equation \eqref{perturbeddbarequation} that is crucial to our argument. The identity's formulation involves the notion of the curvature of $H_\delta$:
 
We have that a horizontal lift of $w\in TD$ to $H_\delta$ has the form 
$$-X_{\delta_w}\oplus w\in TF\oplus TD$$
Let $w^\sharp$ and $x^\sharp$ be two horizontal lifts of tangent vectors in $TD$, and define the curvature two form $\kappa\in \Lambda^2(D,TF)$ to be 
$$\kappa(w,x):=  [w^\sharp, x^\sharp]^{vert}$$
where $vert:TF\oplus H\rightarrow TF$.

Following \cite{ms2}, we write the curvature as 
\begin{equation}\label{curvatureasafunction}
R_\delta  d vol
\end{equation}
where by the curvature identity \eqref{curvatureidentity} $R_\delta:D\times F\rightarrow \bb{R}$ in each fiber is given by the zero average Hamiltonian associated to $\frac{1}{d vol(v,w)}[v^\sharp,w^\sharp]^{vert}$ (whenever $v$ is not a multiple of $w$).

\begin{definition}
For a $J_\delta$ holomorphic section $u$ of $v^*E$, the \emph{covariant derivative} is
\begin{equation}\label{covariantderivative}
\nabla^\delta u:=\partial_{J_F,j} (u)+X^{1,0}_\delta(u),
\end{equation}
i.e., the $(1,0)$ part of $du+X_\delta(u)$ as an element of $\Lambda^{1,0}_{J_F,j}(D,TF)$. 
\end{definition}
Choose coordinates $(s,t)$ on $D$ so that $\delta=f ds+ g dt$ and $X_\delta=X_fds+X_gdt$, and \eqref{covariantderivative} becomes 
$$\nabla^\delta u=(\partial_s u+X_f)ds+(\partial_tu+X_g)dt$$
for $u$ satisfying the Hamiltonian perturbed equation \eqref{perturbeddbarequation}.

Finally, let $\vert\cdot \vert_{J_F}=\omega_F(\cdot,J_F \cdot)$. The relevant notion of \emph{vertical energy} of a $J_\delta$-holomorphic section is given by 
\begin{equation}\label{verticalenergy}
e_\delta(u):=\frac{1}{2}\int_D \vert du+X_\delta\vert_{J_F}^2\text{d} vol=\frac{1}{2}\int_D \vert \nabla^\delta u\vert_{J_F}^2 \text{d} vol.
\end{equation}
It follows that $u$ has $0$ energy precisely when it is a covariant constant section, which motivates the following definition:

\begin{definition}\label{vertconst}
A $J_\delta$-holomorphic section $u$ is \emph{vertically constant} with respect to $H_\delta$ if $e_\delta(u)=0$.
\end{definition}
Also in the $(s,t)$ coordinates we have
\begin{equation}\label{coordexpforddelta}
d\delta=\text{d}' f\wedge d s+\text{d}' g\wedge d t+(\partial_s g-\partial_t f)d s\wedge d t
\end{equation}
where $\text{d}'$ is the exterior differential on the fibers.
By the curvature identity \eqref{curvatureidentity} one can write
$$\kappa=\big(\partial_sg-\partial_t f+\lbrace f,g\rbrace_{\omega_F}\big) d s\wedge d t. $$
We use this in the following lemma:
\begin{lemma}\cite[Lem 5.2]{salamonakveld}{The Covariant Energy Lemma.}\label{energylemma}
Let $u$ be a $J_\delta$-holomorphic section of $v^*E$. We have
\begin{equation}\label{energyequation}e_\delta(u)=\int_D u^*\omega_F + \int_D R_\delta \text{d} vol - \int_{\partial D} u^*\delta.
\end{equation}
\end{lemma}

\begin{proof}
Writing everything in coordinates:
\begin{align*}
e_\delta(u) &=\frac{1}{2} \int_D \omega_F( \partial_s u + X_f,\partial_t u +X_g)-\omega_F (\partial_s u + X_f,\partial_s u +X_f)\\
&\qquad -\omega_F(\partial_t u+X_g,\partial_s u+ X_f)+\omega_F(\partial_s u+X_f,\partial_t u + X_g)d s\wedge d t\\
&=\int_D \omega_F(\partial_s u + X_f,\partial_t u +X_g)d s\wedge d t \\
&= \int_D \big[ \omega_F(\partial_s u,\partial_t u) +\text{d}' f(\partial_t u)-\text{d}' g(\partial_s u)+\lbrace f,g\rbrace_{\omega_F}\big] d s \wedge d t
\end{align*}

From \eqref{coordexpforddelta} we have
$$d\delta\big( \partial_s u\oplus \partial s,\partial_t u\oplus \partial t\big)=\partial_s g-\partial_t f -\text{d}'f(\partial_t u)+\text{d}'g(\partial_s u).$$

Thus, the main calculation continues as 
\begin{align*}
\cdots &= \int_D u^*\omega_F -\int_D u^*d\delta +\int_D\kappa
\end{align*}
by Stokes' theorem and the identification of $\kappa$ with \eqref{curvatureasafunction}, 
\begin{equation*}
=\int_D u^*\omega_F -\int_{\partial D}u^*\delta+\int_D R_\delta d vol
\end{equation*}
\end{proof}

\subsection{A relative h-principle}\label{okasection}
One thing decidedly lacking from the above Hamiltonian point of view is that it is not clear if there is a covariant constant section of $v^*E$, or any holomorphic section at all. To solve such a problem, we use heat flow following Donaldson \cite{sdh} to flatten the $G$-connection via holomorphic gauge transformation.

We reiterate assumptions needed for this section: $E$ has a compact structure group $G$ and $F$ has an integrable $G_\bb{C}$-invariant complex structure $J_I$. Parallel transport in the inherent symplectic connection, denoted $H_G$, is $G$-valued.

\begin{definition}
A \emph{fibered Lagrangian} in a symplectic K\"ahler fibration is a Lagrangian $L\subset (E,\omega_{H,K})$ that fibers as
$$L_F\rightarrow L\xrightarrow{\pi} L_B$$
with $L_F\subset (F,\omega_F)$ a monotone Lagrangian with minimal Maslov index $\geq 2$ and $L_B\subset (B,\omega_B)$ a rational Lagrangian. 
\end{definition}
Any bundle of the form $$\p(V)\rightarrow \p^n$$ where $V$ is equipped with a Hermitian metric is a example of a symplectic K{\"a}hler fibration in which one can find fibered Lagrangians. We also have the flag manifold \eqref{fullflags} which has structure group $SU(n-k+1)$ as a fiber bundle.

Associated to the pull-back connection $v^*H_G$, we can instead require that $J_\delta$ fiberwise agrees with $J_I$, so that it has the following form in the trivial splitting $TF\oplus TD$:
$$J^G:=\begin{bmatrix}
J_I & J_I\circ X_\delta-X_\delta\circ j\\
0 & j
\end{bmatrix}$$
where $\delta$ in this case describes the $G$-valued parallel transport. 

Since the action of $G$ on $F$ is effective, $E$ has an associated principal $G$-bundle. Let $P_{v^*E}$ be the associated bundle to $v^*E$. $H_G$ defines a connection on $P_{v^*E}$, which we also denote by $H_G$.

In this section we draw out the following:

\begin{theorem}\label{oka} There is a gauge transformation $\mathcal{G}\in C^\infty(D,G_\C)$ that preserves $J_I$ in the fibers, is $G$-valued along $\partial D$, and such that $\mathcal{G}(H_G)=:H_0$ is a flat connection on $v^*E$.
\end{theorem}

We will see that this theorem is a corollary of:

\begin{theorem}\cite[Thm 1]{sdh}\label{heat} Let $V\rightarrow \bar{Z}$ be a holomorphic vector bundle over a compact K\"ahler manifold $(\bar{Z},\omega)$ with non-empty boundary $\partial Z$. For any Hermitian metric $f$ on the restriction of $V$ to $\partial Z$ there is a unique Hermitian metric $h$ on $V$ such that:\begin{enumerate}\item[1)] $h\vert_{\partial Z} = f$\item[2)] $i\Lambda F_h = 0$ in $Z$ \end{enumerate}
\end{theorem}
Here, $F_h$ is the curvature of the Chern connection associated to $h$, and $$\Lambda:\Omega^{1,1}(\bar{Z})\otimes V\rightarrow \Omega^0(\bar{Z})$$ is the K\"ahler component in the decomposition $\Omega^{1,1}(\bar{Z)}\otimes V\cong \Omega^{1,1}_{\omega\perp}(\bar{Z})\otimes V \oplus \Omega^{0}(\bar{Z})\omega\otimes V$, where $\Omega^{1,1}_{\omega\perp}$ are the $1,1$ forms defined on the dual of the kernel of $\omega$.

One calls a solution to Theorem \ref{heat} a Hermitian Yang-Mills metric. When $\bar{Z}=\bar{D}$ the associated connection is merely flat, so we can trivialize via parallel transport.

\subsubsection{Gauge transformation background} The space of diffeomorphisms of $P_{v^*E}$ covering the identity forms an infinite dimensional Lie group, called the group of gauge transformations $\mathcal{G}(P_{v^*E})$. Let $\phi\in \mathcal{G}(P_{v^*E})$, and let $g_p$ be the unique element in $G$ such that $\phi(p)=p\cdot g_p$. We require that $\phi$ preserve the $G$-structure of $P_{v^*E}$: $\phi(p\cdot g)=\phi(p)\cdot g$. From this we get
$$p\cdot g\cdot g_{p\cdot g}=p\cdot g_{p}\cdot g$$
so that $g_{p\cdot g}=g^{-1}g_pg$. Thus $\mathcal{G}(P_{v^*E})\cong C^{\infty}(P_{v^*E},G)^{G}$, where the superscript denotes the equivariance $f(p\cdot g)=g^{-1}f(p)g$. When the domain of $v$ is a disk, we have that $P_{v^*E}\cong D\times G$ as a smooth $G$-bundle. As a gauge transformation is completely determined by how it acts on the identity section, we have $\mc{G}(P_{v^*E})\cong C^\infty(D,G)$.

In a similar manner, we denote the complexified gauge group as $\mathcal{G}_\C:=C^\infty (D,G_\C)$. The complexified gauge group acts on $G$-connections via $H\mapsto \mathcal{G}(H)$. 

\begin{proof}[Proof of \ref{oka}]
By assumption the connection $H_G$ on $v^*E$ is induced from a connection on $P_{v^*E}$, and in turn $H_G$ induces a unique complex connection on the complexification $P^\bb{C}_{v^*E}$ by the invariance property of connections under right multiplication. Since $G$ is compact it admits a faithful representation as a matrix group, which induces a faithful matrix representation $\rho:G_\bb{C}\rightarrow GL_\bb{C}(V)$ for a complex vector space $V$. Let $\mc{V}:=V\times_\rho P_{v^*E}^\bb{C}$ be the associated vector bundle, which admits a corresponding connection denoted $\nabla^{H_G}$. This is indeed the Chern connection associated to a Hermitian metric $h_0$ and some fiberwise complex structure $J_\mc{V}$ due to the fact that we can realize $\rho(G)\subset U(k)$ with $k=\dim V$.

By Theorem \ref{heat} there is a flat Hermitian metric $h_\infty$ that agrees with $h_0$ on $\partial D$. Thus, there is a map $\mathcal{G}\in C^\infty(D,G_\C)$ so that $\mathcal{G}^*h_0=h_\infty$.

We have the following, where we use the same notation for the three situations:
\begin{enumerate}
\item $h_0=h_\infty$ on $\partial D$, so we can assume $\mc{G}$ to be $G$ valued on the boundary.
\item By the $G_\bb{C}$-triviality of $P^\bb{C}_{v^*E}$, there is an induced complex gauge transformation (still denoted $\mc{G}$) such that $H_0:=\mc{G}(H_G)$ is a flat connection.
\item $\mc{G}$ induces a transformation $\mc{G}$ of $v^*E$ that gives a flat connection $H_0$
\item By $G_\C$ invariance $\mathcal{G}_*J_I=J_I\mc{G}_*$.\vspace{-2em}
\end{enumerate}
\end{proof}

\begin{remark} 
$v^*E$ can now be holomorphically trivialized by parallel transport along $H_0$, denoted by
\begin{equation}\label{triv}
\Phi:v^*E\xrightarrow{\cong}D\times F.
\end{equation}
This map corresponds to an isomorphism $\mathcal{V}\xrightarrow{\cong}D\times V$ given by an element $\mathcal{G}_\Phi\in C^\infty (D,G_\bb{C})$, which is also $G$ valued on $\partial D$.
\end{remark}

\subsection{Existence of lifts} The prime candidate for a lift of the disk $v$ is a constant section $p:D\rightarrow D\times (F,L_F)$ and the associated section $u_p:=(\mc{G}_\Phi\circ \mc{G})^{-1}p$ of $v^*E$. We ask a few questions about such a section:

\begin{enumerate}
\item\label{vertconst?} Is $u_p$ covariant constant w.r.t. $H_G$?
\item Is $u_p$ $J^G$-holomorphic?
\item Does $u_p$ satisfy the Lagrangian boundary condition?
\item \label{regular?} Is $J^G$ regular for $u_p$ and part of a coherent datum in the sense of Theorem \ref{transversality}?
\end{enumerate}

First we answer \eqref{vertconst?}. Let $\mc{F}=\mc{G}^{-1}$ from Theorem \ref{oka} and $H_0$ the flat connection. The covariant derivative  transforms as $$\nabla^{\mathcal{F}(H_0)}s=\mathcal{F}\nabla^{H_0}(\mathcal{F}^{-1}s)$$ \cite[2.1.7]{dk}. If $s$ is a covariant constant section with respect to $\nabla^{H_0}$, then 
\begin{align*}
\nabla^{H_G}\mc{F}s &=\nabla^{\mc{F}(H_0)}\mc{F}S  \\
&=\mc{F}\nabla^{H_0} \mc{F}^{-1}\mc{F}s\\
& = 0.
\end{align*}
Hence, the definition of vertically constant \ref{vertconst} is independent of the choice of representative of the gauge orbit.
\begin{proposition}\label{hologaugeprop}
$\mc{G}$ from Theorem \ref{oka} preserves holomorphic sections. That is, $\mc{G}:(v^*E,J^G)\rightarrow (v^*E,J_0)$ is a biholomorphism, where $J_0=\Phi^* (J_I\oplus j)$ is pulled back from the trivialization \eqref{triv}.
\end{proposition}
\begin{proof}
$d\mc{G}$ restricts to isomorphisms
\begin{gather*}
d\mc{G}_v:TF\rightarrow TF\\
d\mc{G}_h:H_G\rightarrow H_0
\end{gather*}
and we have $d\mc{G}_v\circ J_I=J_I\circ d\mc{G}_v$ on $TF$ by $G_\bb{C}$ invariance. Thus, it suffices to check that
$$d\mc{G}_h\circ J^G=j\circ d\mc{G}_h.$$

As demonstrated previously, $J^G$ acts on $H_G$ as
$$(-X_y,y)\mapsto (-X_{jy},jy)$$
for $y\in TD$. As $\mc{G}$ is a bundle isomorphism and covers the identity on $D$, it must preserve horizontal lifts of base vector fields. Since $jy$ lifts to $jy$ in $H_0$ and $(-X_{jv},jv)$ in $H_G$, we have
$$\mc{G}_h(-X_{jv},jv)=(0,jv).$$
Hence, 
\begin{equation}\label{holomorphicityofgaugetransformation}
\bar{\partial}_{J_0,J^G}\mc{G}=0.
\end{equation} The fact that $\mc{G}$ preserves holomorphic sections follows.
\end{proof}
Finally, the main application is the following:
\begin{corollary}\label{liftdiskcor}
Given $p\in \pi^{-1}(1)\cap L_F$ and a Hamiltonian $G$-connection $H_G$, there exists a unique, vertically constant section $u$ of $(v^*E,v^*L)$ with $u(1)=p$ which is $J^G$-holomorphic.
\end{corollary}

\begin{proof}
Define a section $\mc{G}u$ of $\mc{G}v^*E$ by parallel transport over $D$ starting at $\mc{G}p\in \pi^{-1}(1)$, which is well defined by the flatness of the connection. Uniqueness for constant maps in the trivialization gives us that $u$ is the unique covariant-constant section starting at $p$. 

It follows from Proposition \ref{hologaugeprop} that $u$ is $J^G$-holomorphic. We check that the boundary condition $u(\partial D)\subset L$ is satisfied: Since $\mc{G}$ is $G$-valued on $\partial D$, it follows that $\mc{G}(L)$ is also Lagrangian with respect to $\mc{G}^*a\vert_{\partial D}$. By Lemma \ref{fiberedlagrangianlemma}, $\mc{G}(L)$ is invariant under parallel transport, so the disk $\mc{G}u$ has $\mc{G}(L)$ boundary conditions.
\end{proof}
We hereafter refer to such a vertically constant section as a \emph{Donaldson lift}.
\subsection{The vertical Maslov index}
To show regularity for these lifts, we require a short discussion on the vertical Maslov index. The Maslov index of a relative class with boundary values in a fibered Lagrangian splits as a sum of horizontal and vertical terms, both of which are topological invariants. It will be crucial to the argument of our main theorem that the ``vertical'' Maslov index is non-negative.

Let $u:D\rightarrow (E,L)$ be a $J_\delta$-holomorphic lift of a holomorphic disk in the base. The Maslov index $I(u)\in 2\bb{Z}$ is defined as the boundary Maslov index of the pair $I(u^*TE,u^*TL)$. It follows that
$$I(u^*TE,u^*TL)=I(u^*TF,u^*TL_F)+I(u^*H,u^*H_L)$$
By the isomorphism $d\pi:(H,H_L)\rightarrow (TB,TL_B)$ we have
$$I(u)=I(u^*TF,TL_F)+I(\pi\circ u^*TB,\pi\circ u^*TL_B)=:I_F(u)+I_B(u).$$
In particular, we have a well defined topological quantity:

\begin{definition}\label{vertmaslovdef}
The \emph{vertical Maslov index} $I_F(u)\in 2\bb{Z}$ for a relative disk class is the vertical part of the Maslov index from the splitting $(TF,TL_F)\oplus (H_\delta,H_L)$. If $u$ is a sphere, we set $I_F(u)=\langle 2c_1(u^*TF),[S^2]\rangle$. 
\end{definition}
\begin{proposition}{Non-negativity of the vertical Maslov index.}\label{verticalmaslovprop}
Let $u:\linebreak (D,\partial D)\rightarrow (E,L)$ be a $J_\delta$-holomorphic disk with $L$ a trivially-fibered Lagrangian. If $u$ is not a Donaldson lift, then there is an open selection of Hamiltonian perturbation data so that $I_F(u)\geq 0$. Moreover if $u$ is a Donaldson lift as in Theorem \ref{liftconfigthm}, $I_F(u)=0$
\end{proposition}
\begin{remark}
Without this proposition, the argument for the main theorem fails due to the fact that the index of a no-input configuration may increase under projection. Hence, there is no way to tell if a configuration counted in the potential for $L$ can be projected to a counted configuration for $L_B$. 
One will also see how the monotonicity of $L_F$ is crucial to our argument.
\end{remark}
\begin{proof}
First, we argue for the vertically constant case, and then in case $u$ is not vertically constant. In the vertically constant case, we argue in the proof of corollary \ref{liftdiskcor} that $u$ is given by a constant map after applying the appropriate trivializing gauge transformation. This provides a trivialization of the pair $(u^*TF,u^*TL_F)\cong (D\times TF,D\times TL_F)$. Thus, the vertical Maslov index is zero in this case.

Let $v:=\pi\circ u$ and without confusion let $u$ also denote the induced section of $v^*E$. Let $H_\delta$ denote the induced connection on $v^*E$. In the case $u$ is not a vertically constant section, we are free to use a more general perturbation of the connection: First we show the proposition for when the holonomy above $L_B$ vanishes. For $S\subset D$ and $f$ a zero fiberwise-average function, let
$$\Vert f\Vert_S:=\int_S \max_{p\in F}f(z,p)-\min_{p\in F}f(z,p) d vol$$
denote the Hofer norm. From Lemma \ref{energylemma}, we have
$$0\leq e_\delta(u)\leq \int_D u^*\omega_F+\Vert R_{\delta}\Vert_D+\Vert \delta \Vert_{\partial D}=\int_D u^*\omega_F+\Vert R_{\delta}\Vert_D$$
where the last equality holds since the holonomy of the connection vanishes over $L_B$ (the trivial assumption).

Hence,
$$-\Vert R_{\delta}\Vert_D\leq \int_D u^*\omega_F.$$
Let $\lambda>0$ be so that $\int_D u^*\omega_F=\lambda I_F(u)$ for all relative disk classes $u\in H^{disk}_2(F,L_F,\bb{Z})$, and take 
$$\Sigma_{L_F}=\min\left\{ I_F(u)>0|u\in H^{disk}_2(F,L_F,\bb{Z})\right\}$$
to be the minimum Maslov number for $L_F$ (we have $\Sigma_{L_F}\geq 2$ by assumption). To show the first statement, we can choose the $\delta$ close enough to $0$ in our perturbation data over the domain of $u$ so that 
$$\Vert R_{\delta}\Vert_D < \lambda \Sigma_{L_F}$$
and it follows that $I_F(u)\geq 0$. The second statement follows from the same choice of perturbation data.

For a general holonomy, we use the triviality of $L$ followed by a lemma from \cite{salamonakveld}. As in the previous sections, let $v^*L\subset S_1\times F$ denote the fiber subbundle with fiber $L_F$. By the ambiently trivial assumption, there is a fiberwise Hamiltonian diffeomorphism $\phi:S^1\times F\rightarrow S^1\times F$ that takes the pulled-back connection to the trivial one. In particular $\phi(v^*L)=S^1\times L_F$.

Following the proof of \cite[Lem 3.2(ii)]{salamonakveld} we can extend this Hamiltonian diffeomorphism to a fiberwise Hamiltonian diffeomorphism over the disk
$$\Phi:v^*E\rightarrow v^*E.$$
On the target connection $\Phi_*H_\delta$, define $$\Phi^*J_\delta:=d\Phi\circ J_\delta \circ d\Phi^{-1}.$$
Evidently, Lemma \ref{connectionacs} tells us that $\Phi^*J_\delta$ is the unique almost complex structure on $v^*E$ that fiberwise agrees with $\Phi^*J_F$ and preserves $\Phi_*H_\delta$. Moreover, $\Phi\circ u$ is $\Phi^*J_\delta$-holomorphic. The energy identity \eqref{energyequation} for such a disk tells us that
$$\int\vert d\Phi\circ u\vert^2_{\Phi^* J_\delta}=\int_D (\Phi\circ u)^*\omega_F + \int_D R_{(\Phi^{-1})^*\delta} \text{d} vol$$
which does not include the integral around the boundary since $\Phi_*H_\delta$ has trivial holonomy here. By \cite[Lem 3.2(iii)]{salamonakveld} we have $(\Phi^{-1})^*R_{\delta}=R_{(\Phi^{-1})^*\delta}$. In the construction in \cite{salamonakveld}, $\Phi$ only depends on the value of $\delta$ over the boundary. Therefore, as in the vanishing holonomy case, the vertical Maslov index of $\Phi\circ u$ is non-negative for $\delta$ small enough over the interior of the disk. Since $\Phi\circ u$ is isotopic to $u$, the proposition follows.
\end{proof}

\begin{corollary}\label{zeroenergycorollary}
Let $L_F\rightarrow L\rightarrow L_B$ be a trivially fibered Lagrangian in a symplectic K\"ahler fibration. Let $v:D\rightarrow (B,L_B)$ be a $J_B$-holomorphic disk and $\mc{L}v:D\rightarrow (E,L)$ be a Donaldson lift as in Theorem \ref{liftconfigthm} in the given $G$-connection. Then the vertical symplectic area of $\mc{L}v$ is equal to (minus) the integral of the holonomy Hamiltonian around the boundary:
\begin{equation}
e_v(\mc{L}v):=\int_D \mc{L}v^*a=-\int_{\partial D} \mc{L}v^*\delta.
\end{equation}
\end{corollary}

\begin{proof}
From the definition of $a$ and Stokes' theorem, we have
$$\int_D \mc{L}v^*a=\int_D \mc{L}v^*\omega_F-\int_D \mc{L}v^*d\delta=\int_D \mc{L}v^*\omega_F-\int_{\partial D} \mc{L}v^*\delta.$$
By the final statement in Proposition \ref{verticalmaslovprop} and monotonicity of $(F,L_F)$,
$$\int_D \mc{L}v^*\omega_F=\lambda I_F(\mc{L}v)=0$$
from which the corollary follows.
\end{proof}

\subsection{Existence of regular lifts}

We show we can do the same as Corollary \ref{liftdiskcor} for more general configurations, as well as achieve transversality for such sections:

\begin{theorem}\label{liftconfigthm}
Let $v:C\rightarrow B$ be a regular $J_B$-holomorphic configuration of type $(\g,q)$ with output $v(x_0)=\tilde{q}$, no sphere components, and no broken edges, and let $H_G$ be a $G$-connection on $v^*E$. Then for any $\tilde{p}\in \pi^{-1}(\tilde{q})\cap L_F$, there is a unique lift $\mc{L}_{\tilde{p}}v:C\rightarrow E$ with $\mc{L}_{\tilde{p}}v(x_0)=\tilde{p}$ that is vertically constant with respect to $H_G$ on disk components and $J^G$-holomorphic. Moreover, if $\I(\g,q)=0$, then $\mc{L}_{\tilde{p}}v$ lives in a smooth moduli space of expected dimension $\dim W^+_{X_g}(p)$ where $p$ is the limit of the gradient flow starting at~$\tilde{p}$.
\end{theorem}
Notation: Given such a type $(\g,q)$, we denote by $(\pi^*\g,p)$ its lift through $p$, also called a \emph{Donaldson lift}.

\begin{proof}
We construct a lift by matching the chain of boundary conditions, and then we prove transversality. Let $v_0$ be the restriction of $v$ to the disk component $D_0$ closest to the output, let $E_{01}$ denote the edge between $D_0$ and some adjacent component $D_1$, and let $v_1$ denote the restriction of $u$ to $D_1$ that meets $E_{01}$ at $x_2$. By Lemma \ref{liftdiskcor}, there is a unique vertically constant lift $\hat{v}_0$ with $\hat{v}_0(x_0)=p$. Let $x_1$ be the boundary point corresponding to where $E_{01}$ connects to $D_0$. The projection of the flow of $X$ starting at $\hat{v}(x_1)$ agrees with the flow of $X_b$, so flowing $X$ for time $\ell (E_{01})$ lands at a point $p_1\in\pi^{-1}(v(x_2))$. Take the unique lift of $\hat{v}_1$ of $v_1$ with $\hat{v}(x_2)=p_1$, and continue in this fashion until a lift of $v$ is constructed on every disk component. This gives a $J^G$-holomorphic Floer trajectory $\mc{L}_{\tilde{p}}v:C\rightarrow E$ with boundary in $L$.

To see the expected dimension of $\pi^*\g$, we have the following lemma:
\begin{lemma}\label{indexofliftlemma}
Let $(v,\g,q)$ be a regular configuration of index $0$ in $B$. Then the Donaldson lift has expected dimension
\begin{equation}
\I(\mc{L}_{\tilde{p}} v,\pi^*\g,p)=\I(v,\g,q)+\dim W^+_{X_g}(p).
\end{equation}

\end{lemma}
\begin{proof}[Proof of lemma]
Let $v_i$ denote a surface component. We have that $I_F(\mc{L} v_i)=0$ by Proposition \ref{verticalmaslovprop}. Moreover, $m_v(z_0)=m_{\mc{L}v}(z_0)$ for an interior marked point $z_0$ as the derivative of $d\mc{L}v_i(z_0)$ vanishes to all orders in the vertical direction. Since $$\dim W^+_{X}(p)=\dim W^+_{X_b^\sharp}(p) +\dim W^+_{X_g}(p)=\dim W^+_{X_b}(q) +\dim W^+_{X_g}(p),$$
the lemma follows.
\end{proof}

Finally, we show the transversality statement, and to start we restrict to a single disk component $v:D\rightarrow (B,L_B)$. Since the divisor is stabilizing for $L_B$, we can assume that $D$ has an interior marked point and the domain is stable as a marked disk. Choose a metric $g$ on the vertical sub bundle $TF$ which makes $TL_F\vert_{\partial D}$ totally geodesic. First, we achieve transversality for the disk $\mc{L}_pv$ as a section of $v^*E$, and then explain how transversality follows for the disk in the total space. Let $p\in L_F$, and for $\ell, k$ integers with $k\ell> \dim_{\bb{R}} F$, $k\geq 1$, take $\text{Map}^{k,\ell}_{sec}(D,u^*E,u^*L,p)_{[\mc{L}_pv]}$ to be the space of (continuous) sections $u$ of $v^*E$ with boundary values in $v^*L$ in the class $[\mc{L}_pv]$, that have $k$ weak derivatives which are $\ell$ integrable and such that $u(1)=p$. Such a space is a Banach manifold that is locally modelled on $W^{k,\ell}(D,u^*TF,u^*TL_F)_1$, vector fields of Sobolev class $(k,\ell)$ that vanish at $1$, via geodesic exponentiation in the metric g.

This space admits a vector bundle $\mc{E}_{J^G,j}^{k-1}$ whose fibers are given by
$$(\mc{E}_{J_I,j}^{k-1})_{u}=\Lambda^{0,1}_{J_I,j}(D,u^*TF)_{k-1,\ell}$$
where $j$ is the standard complex structure on $D$. We have a section
\begin{equation}\label{dbaroperator}\bar{\partial}_{X_\delta,j}u:=\pi^{X_\delta}_{TF}[du+J^G\circ du\circ j]
\end{equation}
where $\pi_{TF}$ is the projection onto $TF$ in the splitting $TF\oplus H_\delta$.

Let $D_u\bar{\partial}_{X_\delta,j}$ be the linearization of \eqref{dbaroperator} at $u$. In coordinates \eqref{dbaroperator} has the form 
$$du+J_I\circ du\circ j + X^{0,1}_\delta(u)$$
so the linearization is a real linear Cauchy-Riemann operator and hence it is Fredholm (see \cite[Prop 3.1.4]{ms2} for an explicit formula). We show that the linearization is surjective at the Donaldson lift $\mc{L}_pv$.

Let $\mc{G}$ be the trivializing gauge transformation from Theorem \ref{hologaugeprop}. By Proposition \ref{hologaugeprop},  $\mc{G}$ induces an isomorphism 
\begin{equation}\label{gaugeiso}
\begin{tikzcd}
\mc{E}_{J_I,j}^{k-1} 
\arrow[d] \arrow[r,"\mc{G}_*"]& \mc{E}_{J_I,j}^{k-1}
\arrow[d]\\
\text{Map}^{k,\ell}_{sec,[\mc{L}_pv]}
\arrow[r,"\mc{G}"]& \text{Map}^{k,\ell}_{sec,\mc{G}_*[\mc{L}_pv]}
\end{tikzcd}
\end{equation}
Moreover, we have that 
$$\bar{\partial}_{0,j}\mc{G}\circ u=\mc{G}_*\bar{\partial}_{X_\delta,j}u$$
by the following three facts: $\mc{G}\circ \pi^{X_\delta}_{TF}=\pi^{0}_{TF}\circ \mc{G}$ since the connection on the target is $\mc{G}(H_\delta)$, the equation \eqref{holomorphicityofgaugetransformation}, and the chain rule for complex derivatives.  From this one gets that the linearizations commute with $\text{d}\mc{G}$, so it suffices to show that $D_{\mc{G}\mc{L}_pv}\bar{\partial}_{0,j}$ is surjective. Since $J_I$ is integrable, the latter is given by
$$\xi\mapsto (\nabla\xi)^{0,1} $$
(see \cite[Rem 3.1.2]{ms2}) where $\nabla$ is some Chern connection induced by $J_I$, which is the standard $\bar{\partial}$-operator associated with the trivial vector bundle $(\mc{G}\mc{L}_pv)^*TF$. By construction $\mc{G}(\mc{L}_pv)$ is given as a constant section, and we have an isomorphism of tangent spaces that starts at the right side of the above diagram
\begin{equation}\label{tangiso}
\begin{tikzcd}
\Lambda^{0,1}_{J_I,j}(D,\mc{G}\mc{L}_p v^*TF)_{k-1,p} 
\arrow[d] \arrow[r,"\cong"] & \Lambda^{0,1}_{J_I,j}(D,w^*TF)_{k-1,\ell}
\arrow[d]\\
W^{k,\ell}(D,\mc{G}\mc{L}_p v^*TF,\mc{G}\mc{L}_p v^*TL_F)_1
\arrow[r,"\cong"]& W^{k,\ell}(D,w^*TF,w^*TL_F)_1
\end{tikzcd}
\end{equation}
where $w:D\rightarrow L_F$ is a constant map. $w$ extends to a map $\underline{w}:\bb{P}^1\rightarrow F$ and by Schwarz reflection along the boundary conditions and gives an injection $W^{k,\ell}(D,w^*TF,w^*TL_F)_1\hookrightarrow W^{k,\ell}(\bb{P}^1,\underline{w}^*TF)_1$ that preserves holomorphic vector fields. However, if we ignore the vanishing condition for a moment, the only holomorphic vector fields in the latter are constant since it is the trivial bundle, so 
$$\dim_\bb{R} \ker D_w\bar{\partial}_{0, j}=\dim_\bb{R}F.$$ By the Riemann-Roch theorem, the index of $D_w\bar{\partial}_{0, j}$ (without the vanishing condition) is also $\dim_\bb{R}F$, so it follows that it is surjective. From the above isomorphisms it follows that $D_{\mc{L}_p v}\bar{\partial}_{X_\delta,j}$ is surjective. 

Next, we need to show that $D_u$ is surjective at any holomorphic section in the class $[\mc{L}_pv]$. We have the following rigidity statement:
\begin{lemma}\label{onlyrepinclasslemma}
$\mc{L}_pv$ is the only $J^G$-holomorphic section in the class $[\mc{L}_pv]$ with $u(1)=p$.
\end{lemma}
\begin{proof}[Proof of lemma]
Let $u$ be some $J^G$-holomorphic section in the class $[\mc{L}_pv]$. Via the trivializing gauge transformation and Proposition \ref{hologaugeprop}, $u$ induces a section $\tilde{u}$ of $\mc{G}(v^*E)$ that is $J_I\times j$-holomorphic. Sections of $\mc{G}(v^* E)$ are in one-to-one correspondence with $J_I$-holomorphic disks $D\rightarrow (F,L_F)$. By the energy identity in $F$ with norm induced by $J_I$, it follows that if $\tilde{u}$ is not a constant section then
$$\int_D \tilde{u}^*\omega_F>0.$$
Clearly $\int_D \mc{G}(\mc{L}_pv)^*\omega_F=0$, so it follows that $\tilde{u}$ cannot be in the class $\mc{G}_*[\mc{L}_pv]$. Since we are prescribing the condition $u(1)=p$, the lemma follows.
\end{proof}

By elliptic regularity every element in the kernel of $\bar{\partial}_{X_\delta,j}$ is smooth for smooth $J^G$. By the implicit function theorem, the moduli 
$$\mc{M}_{J^G,j,[\mc{L}_pv]}^{sec}(v^*E,v^*L)_0:=\bar{\partial}_{X_\delta,j}^{-1}(0)$$ of smooth holomorphic sections with $u(1)=p$ in the class $[\mc{L}_pv]$ is smooth of expected dimension.

Next we explain why we can achieve transversality for such a lift when considered as a class in the total space, given the fact that we have transversality for the map in the base.   Again, we focus on a single disk component: We have a Banach vector bundle
$$\tilde{\mc{E}}_{J^G,j}^{k-1}\rightarrow \text{Map}^{k,\ell}_{\pi^*\g}(D,u^*E,u^*L,p)$$
whose base is the space of $(k,\ell)$ treed disks of type $\pi^*\g$ with output $p$ modelled on $W^{k,\ell}(D,u^*TE,u^*TL)_1$ and whose fiber at a map $u$ is 
$$\Lambda^{0,1}_{J^G,j}(D,u^*TE)_{k-1,\ell},$$
 where $J^G$ is now an almost complex structure on the total space of the form~\eqref{acsontotalspace}.

The bundle mentioned above depends on the $J_B$ we choose on the base. In fact, there is an extension of this bundle to a base $$\mc{B}(E,L)_{k,\ell,l}:=\text{Map}^{k,\ell}(E,L)\times \mc{P}^l_{\pi^*\g}(E,D)_{J_I,H_G}$$
with $\mc{P}^l_{\pi^*\g}(E,D)_{J_I,H_G}$ the space of domain dependent perturbation data as in definition \ref{perturbationdatadef} that fix the vertical complex structure and Hamiltonian connection. From now on we denote an element of this space by $J$. We have a $\bar{\partial}$ operator as a section
\begin{gather*}
\bar{\partial}_{J,j}:\mc{B}(E,L)_{k,\ell,l}\rightarrow \tilde{\mc{E}}_{J,j}^{k-1}\\
\bar{\partial}_{J,j}=du+J\circ du\circ j.
\end{gather*}
Since each $J$ preserves $H_G$, we have a splitting above surface components that does not depend on $J$:
\begin{equation}\label{splitting}
\Lambda^{0,1}_{J,j}(D,u^*TE)_{k-1,\ell}\cong \Lambda^{0,1}_{J_I,j}(D,u^*TF)_{k-1,\ell}\oplus \Lambda^{0,1}_{J,j}(D,u^*H_G)_{k-1,\ell}.
\end{equation}
$\Lambda^{0,1}_{J,j}(D,u^*H_G)_{k-1,\ell}$ is isomorphic to $\Lambda^{0,1}_{J_B,j}(D,\pi\circ u^*TB)_{k-1,\ell}$ by a $(J,J_B)$ equivariant isomorphism given by ``horizontal lifting'', e.g.
\begin{gather}\label{projectioniso}
\Lambda^{0,1}_{J_B,j}(D,\pi\circ u^*TB)_{k-1,\ell}\rightarrow \Lambda^{0,1}_{J_\sigma,j}(D,u^*H_G)_{k-1,\ell}\\
\eta\mapsto (-X^{0,1}_\sigma\circ\eta,\eta)
\end{gather}
in a trivialization.

The linearized operator can be written as
\begin{equation}\label{linearizedoperatortotalspace}
D_{u,J}\bar{\partial}_{J,j}=D_{u}\bar{\partial}_{J,j}+\mf{J}\circ du\circ j,
\end{equation} where we explain the form of $\mf{J}$: The space $\mc{P}^l_{\pi^*\g}(E,D)_{J_I,H_G}$ admits a pushforward $\pi_*\mc{P}^l_{\pi^*\g}(E,D)_{J_I,H_G}$ by Lemma \ref{connectionacs}. In words, this is the space of domain dependent complex structures with values in $\mc{J}^l(B,\omega_B)$ satisfying some additional conditions according to Definition \ref{perturbationdatadef}. For $K_B\in T\pi_*\mc{P}^l_{\pi^*\g}(E,D)_{J_I,H_G}$, we can construct a canonical $\mf{J}$ which looks like
\begin{equation}\label{basevariation}K_B\mapsto \mf{J}=\begin{bmatrix}
0 & X_\sigma\circ K_B\\
0 & K_B
\end{bmatrix}
\end{equation}
in some coordinate chart. The linearized operator splits as well:
\begin{align*}
&D_{u}^{TF}\oplus D_{u,J}^{H_G}: W^{k,\ell}(D,u^*TF,u^*TL_F)\\
&\qquad \oplus W^{k,\ell}(D,\pi\circ u^*TB,\pi\circ u^*L_B)\times T\mc{P}^l_{\pi^*\g}\\
&\rightarrow \Lambda^{0,1}_{J_I,j}(D,u^*TF)_{k-1,\ell}\oplus \Lambda^{0,1}_{J,j}(D,\pi\circ u^*TB)_{k-1,\ell}.
\end{align*}
Namely, the $D_u$ part splits and the summand depending on $\mf{J}$ lands in the second factor.

From the above argument for transversality of holomorphic sections in the class $[\mc{L}_pv]$ it follows that the linearized operator is surjective onto the first summand.

For the second factor in the splitting, it follows from \cite[Thm 2.22]{CW2} that variations of the $K_B$ \eqref{basevariation} are sufficient to achieve surjectivity of $D^{H_G}$ onto the second factor.

By the implicit function theorem it follows that $(\bar{\partial}_{J,j})^{-1}(0)$ is a smooth manifold. By a Sard-Smale argument, there is a comeager set of $J$ at which the linearization of the projection $\Pi:(\bar{\partial}_{J,j})^{-1}(0)\rightarrow \mc{P}^l_{\pi^*\g}(E,D)_{J_I,H_G}$ is surjective. However, the cokernel of $\Pi$ at such a $J$ is isomorphic to the cokernel of $D_{u,J}\bar{\partial}_{J^G,j}$ at the same $J$. By the implicit function theorem, it follows that moduli space of disks
$$\mc{M}_{J,j,[\mc{L}_p v]}(E,L)$$
is smooth and of expected dimension for the comeager set of $J$. By assumption and the splitting of the linearized operator, our chosen $J$ lies in this set.

Transversality along the edges and at the matching conditions is a standard argument and done in \cite[Thm 6.1]{floerfibrations}. Transversality at the intersections with the divisor is done in detail \cite[Lem 6.6]{CM}.
\end{proof}

Regularity of the projection $E\rightarrow B$ follows from the proof of Theorem~\ref{liftconfigthm}. We will need this result in the proof of the main theorem:
\begin{lemma}\label{regular proj}
Let $u_\g:U_\g\rightarrow E$ be a regular $P_\g$-holomorphic configuration $P_\g$ as in Theorem \ref{transversality}. Then $\pi\circ u_\g$ is a regular $J_B$-holomorphic configuration. 
\end{lemma}

\begin{proof}
The fact that $\pi\circ u$ is $J_B$ holomorphic is clear. Furthermore, by the choice of pseudo-gradient perturbation data and the discussion in Section \ref{pseudogradients}, we have that $\pi\circ u$ is a Morse flow on edges. For almost complex structures of the form \eqref{acsontotalspace} the splitting \eqref{splitting} holds we have the commutative diagram on surface components:
\begin{equation}\label{commutativediagram}
\begin{tikzcd}
\text{Map}^{k,\ell}_{[u]}(E,L) 
\arrow[twoheadrightarrow]{dd}{\pi_*}\arrow[twoheadrightarrow]{rd}{D_{u,J}} & \\
& \Lambda_{j,J}^{0,1}(D,u^*TF\oplus H)_{k-1,\ell} 
\arrow[two heads]{d}{d\pi_*}\\
\text{Map}^{k,\ell}_{[\pi\circ u]}(B,L_B)
\arrow[twoheadrightarrow]{r}{D_{\pi\circ u,J_B}} &\Lambda_{j,J_B}^{0,1}(D,\pi\circ u^*TB)_{k-1,\ell}&
\end{tikzcd}
\end{equation}
By assumption, $D_{\tilde{v},J^G}$ is surjective, and thus must be surjective onto the second part of the splitting \eqref{splitting}. By the isomorphism \eqref{projectioniso}, it follows that $d\pi_*$ is surjective, so it follows from the diagram that $D_{\pi\circ u,J_B}$ is surjective.

There is a similar splitting and diagram along the edges, from which regularity follows.
\end{proof}

\subsection{The lifting operator}\label{liftingopsection}

With Theorem \ref{liftconfigthm} in place, we can finally define a map on moduli spaces. We begin to follow the notation from the introduction, with $x$ some critical point in $L_B$ and $x^i$ a lift of $x$ to a critical point in $L$. Given a combinatorial type $(\g,x)$ in $B$ with index $0$ and an $x^i\in \crit(\pi^{-1}(x))$, let
\begin{equation}\label{donaldsonmodulispace}
 \mathcal{M}_{\pi^*\g} (E,L,J^G,x^i)
 \end{equation}
be the moduli space of holomorphic configurations of the same combinatorial type as a Donaldson lift as in Theorem \ref{liftconfigthm}. We showed that this is regular and of expected dimension, and by Lemma \ref{onlyrepinclasslemma} it contains only Donaldson lifts.

\begin{definition}
Choose a combinatorial type $(\g,x)$ for $(B,L_B)$ of expected dimension $0$ and a lift $x^i$ of $x$. For each $u_j\in \mathcal{M}_{\g} (B,L_B,J_B,x)$, choose a $q_j\in \pi^{-1}(u(1))\cap W^+_X(x^i)$ and denote by $\mathbf{q}= (q_1,\dots)$ the set of choices across all such $u$. For a coherent regular perturbation datum $(P_\g)$ that uses $J^G$ for all types with $I_F(\g)=0$, a \emph{lifting operator} is defined as
\begin{equation}\label{liftingop}
\mathcal{L}_{\bf{q}}^{\g}:\mathcal{M}_{\g} (B,L_B,J_B,x)\rightarrow \mathcal{M}_{\pi^*\g} (E,L,J^G,x^i)
\end{equation}
that lifts a configuration to the unique vertically constant configuration through $q_j$.
\end{definition}

We thus have a collection of lifting operators for each choice $\bold{q}$ and we get a surjection when we consider a collection $\lbrace  \mathcal{L}_{\bf{q}}^{\g} \rbrace$ for which each $q_j$ hits all points in $\pi^{-1}(u_j(1))\cap W^+_X(x^i)$.

Note that the choice of $\bold{q}$ is unique when $x^i=x^M$, the unique maximal index critical point for the fiber. In such a case we will simply denote the lifting operator as $\mc{L}^\g_{x^M}$.

\begin{remark}
The Donaldson moduli space \ref{donaldsonmodulispace} is contained in the larger moduli space
$$\mc{M}_{I_F=0}(E,L,J^G,x^i)$$
consisting of all $J^G$-holomorphic aspherical types with $0$ vertical Maslov index on each disk component. However, Lemma \ref{onlyrepinclasslemma} can be rephrased as
\begin{lemma}\label{vertconst0maslemma}
$u$ is vertically constant for $u\in \mc{M}_{I_F=0}(E,L,J^G,x^i)$.
\end{lemma}
which shows that these two moduli are actually the same.
\end{remark}

\subsubsection{Coherence perturbation data}
The next technical detail is to realize the almost complex structure $J^G$ as a part of a coherent system as in Theorem \ref{transversality}, which allows the compactness argument \ref{compactnessthm} to run. In general, the answer is no since a vertically constant configuration (say $\g_1$) may be a piece of a broken configuration $\g_1\cup \g_2$ at the boundary of the one dimensional moduli space. Let $\g$ denote the configuration obtained by contracting an edge between $\g_1$ and $\g_2$ and $S$ denote the newly formed surface component: Then $\g$ is index $1$ and we need to allow a more general perturbation datum on $S$ to rule out sphere and vertical disk bubbling in line with Theorem \ref{compactnessthm}. However, it would not be possible for $J^G$ to be given by pullback as per the \emph{(contracting an edge)} and \emph{(products)} axioms, since $J^G$ would not be regular for spheres or vertical disks that may bubble off of $S$. However, it suffices to use $J^G$ to compute holomorphic representatives of index $0$ vertically constant types by using the notion of a regular homotopy and showing that the moduli space $\mathcal{M}_{\pi^*\g} (E,L,J^G,p)$ is in bijection with one defined using a perturbation datum from Theorem \ref{transversality}. In this section, we briefly outline why this is true, but the reader may skip ahead.

\begin{definition}
For a combinatorial type $\g$ with $\I (\g,\bar{x})\leq 1$, and two regular fibered almost complex structures $J_i$ $i=0,1$ as in Definition \ref{perturbationdatadef} for a fixed (domain dependent) $J_B$, define a \emph{regular smooth homotopy fixing $J_B$}, denoted $J_t\in\mathcal{J}^{reg}_{\g,J_B}(J_0,J_1)$, as
\begin{align*}
J_t:[0,1]&\rightarrow \mc{P}^l_\g(E,D)\\
&J_0=J_0\\
&J_1=J_1
\end{align*}
with each $J_t$ making $\pi$ a holomorphic map with respect to $J_B$ on the base, and such that the linearized operator
$$D_{u,J_t}+\frac{\partial}{\partial t}J_t\circ du\circ j$$
is surjective at each $t\in [0,1]$ and $u\in \mathcal{M}_\g(L,J_t)$ (here, the first summand is defined as in \eqref{linearizedoperatortotalspace}).
\end{definition}

Let $\mathcal{J}_{\g,J_B}(J_0,J_1)$ be the Banach manifold of (not necessarily regular) smooth homotopies fixing $J_B$. We use the following adaptation of \cite[Thm 8.4.1]{ms2}:
\begin{theorem}\label{homotopy}
For a combinatorial type $\g$ with $\I(\g,\bar{x})\leq 1$ and a regular $J_B$, let $J_i$, $i=0,1$ be regular upper triangular almost complex structures for type $(\g,\bar{x})$ such that $\I(\g,\bar{x})\leq 1$. Then there is a Baire set of smooth homotopies $\mathcal{J}_{\g,J_B}^{reg}(J_0,J_1)\subset \mathcal{J}_{\g,J_B}(J_0,J_1)$ such that if $J_t\in \mathcal{J}_{\g,J_B}^{reg}(J_0,J_1)$, then there is a parametrized moduli space $\mathcal{W}_{\g,J_B} (J_t)$ that is a smooth oriented manifold with boundary
$$\partial\mathcal{W}_\g (J_t)=\mathcal{M}_\g (L,J_0)^-\sqcup \mathcal{M}_\g(L,J_1)$$
so that these two moduli spaces are oriented compact cobordant.
\end{theorem}
\begin{proof}
The proof is the same as an argument of Theorems \ref{transversality} and \ref{compactnessthm}. The projection from the universal moduli space $\mathcal{M}^{univ}_\g(L,J_t)\rightarrow \mathcal{J}_{\g,J_B}(J_0,J_1)$ has the same Fredholm index and cokernel dimension as the linearized $\bar{\partial}$ operator, and the points where the projection is surjective are precisely the regular homotopies. After showing that the universal space is a $C^l$ Banach manifold, one then uses the Sard-Smale theorem to find a Baire set where the projection is submersive and applies an argument due to Taubes to show a comeager set of smooth homotopies.
\end{proof}

Let $0<\epsilon<\epsilon_0<1$. The above theorem tells us that $\mathcal{M}_\g (L,J_0)$ and $\mathcal{M}_\g (L,J_\epsilon)$ compact cobordant, and thus are diffeomorphic for $\epsilon<<1$. Indeed, any regular homotopy induces a cobordism $\mathcal{W}_\g^\vee (J_t)\rightarrow [0,1]$ that is a submersion at $0$. Since the property of being a submersion is an open condition, this must be a submersion in a neighbourhood $[0,\epsilon)$. It follows $\mathcal{M}_{\pi^*\g} (E,L,J^G,p)$ is diffeomorphic to $\mathcal{M}_{\pi^*\g} (E,L,J_\delta,p)$ for $J_\delta$ from Theorem \ref{transversality} close enough to $J^G$. 

\subsection{Spin and relative spin structures on $L$}
We want to have coherent orientations on the moduli of holomorphic disks as in Theorem \ref{transversality} so that the lifting operator \eqref{liftingop} is orientation preserving. Thus we must discuss these structures for trivially fibered $L$. Note that since $L$ is topologically trivial, the existence of a spin structure on any two of $L$, $L_F$, and $L_B$ gives the existence of such on the third by naturality and the product formula for the Stiefel-Whitney class $w(L)$. In the converse direction, a choice of spin structures on $L_B$ and $L_F$ gives one on $L$.

More generally, we have the following notion in the literature (c.f.\linebreak \cite[Def 8.1.2]{fooo}):
\begin{definition}
A \emph{relative spin structure} on an oriented Lagrangian $L$ is a choice of an oriented real vector bundle $V$ on $M$ with $w_2(V\vert_L)=w_2(TL)$.
\end{definition}
Given a relative spin structure $V$ for $L$, we can determine one on $L_F$ as follows: We have $w_2(TL_F\oplus TL_B\oplus V)=0$ so that for $\iota_p:L_F\rightarrow F_p$ the inclusion we get $w_2(TL_{F_p}\oplus \iota^*_p V)=0$ since $\iota^*_p TL_B$ is the trivial bundle.

On the other hand, there is no clear way to induce a relative spin structure for $L$ given that of $L_B$ and $L_F$ since it would require extending a vector bundle from a fiber to the total space. Thus, we consider only trivially fibered Lagrangians such that $L_B$ is relatively spin and $L_F$ is spin. Thus, a relative spin structure $V$ for $L_B$ induces one for $L$ by pullback $\pi^*V$ after choosing one for $L_F$. We call the relative spin structure induced on $L$ this way a \emph{relative spin-spin} structure.
\begin{proposition}\label{coherentorientations}
Equip $L$ with a relative spin-spin structure. Given a collection of types with $\I(g,y)\leq 1$ as in Theorem \ref{transversality} containing Donaldson lifts, there exists a system of coherent orientations on the associated moduli spaces so that $\mc{L}_{x^M}^\g$ \eqref{liftingop} is orientation preserving.
\end{proposition}
\begin{proof}
The existence of coherent orientations for no-leaf configurations in the base follows from \cite[Thm 4.21]{CW2}. Choose orientations on the Lagrangians so that $\det L_F\otimes \det L_B\cong \det L$. We have that $W^+_{X_b}(x)$ is naturally diffeomorphic to $W^+_X(x^M)$, since they are in a neighbourhood of their respective critical points by construction and we can extend such a local diffeomorphism via the flow. Thus choose orientations on $W^\pm_X(x^i_j)$ so that this diffeomorphism is orientation preserving in addition to satisfying $\det (T_{x^i_j}W^+_X(x^i_j)\oplus T_{x^i_j}W^-_X(x^i_j))\cong \det T_{x^i_j}L$. Moreover, $\mf{M}_\g\cong \mf{M}_{\pi^*\g}$ since we only change relative homology classes when we lift. From the orientation isomorphism \eqref{orientationiso} we have
$$\det D_{\mc{L}_{x^M}^\g u}\cong \det\mf{M}_\g\otimes \det TL_F\otimes\det TL_B\otimes \det T W^+_{X_b}(x).$$
To see the orientation preserving claim, it suffices to check that the $\det TL_F$ factor does not depend on the map $u$. We may assume that $\mc{L}_{x^M}^\g u(1)=x^M$. For fixed $u$, an orientation of $L_F$ at $x_M$ determines the above isomorphism with $\det D_{\mc{L}_{x^M}^\g u}$ uniquely by the spin property of $L_F$. On the other hand, each $\mc{L}_{x^M}^\g u$ has vertical Maslov index $0$, so the isomorphism only depends on $\det TL_F\vert_{x^M}$. Since $\det L_F\vert_{x_M}\cong \det L_F\vert_{x_M'}$ by triviality of $L$, the lift induces the same sign across all $x^M$ and $u$. Finally, one can begin with an orientation on $L_F$ so that this sign is positive.
\end{proof}

\section{Floer invariants}\label{invariantssection}
\subsection{An $A_\infty$-algebra}

While the disk potential is the main focus of this paper, we define the $A_\infty$ algebra of a fibered Lagrangian to actually use the power of the potential.

Let $E$ be a symplectic K\"ahler fibration, and $L$ a fibered Lagrangian. For a relative homology class $u\in H^{disk}(E,L)$, let 
$$e_v(u):=\int_D u^*a$$
denote the \emph{vertical symplectic area}. To avoid mentioning ``$K$'' too many times, we let
$$e(\pi\circ u):=\int_D K \pi\circ u^*\omega_B.$$ 

Define the ring
\begin{align*}
\Lambda^2:= \biggl\lbrace& \sum_{i,j} c_{ij} q^{\rho_i} r^{\eta_j} \vert c_{ij}\in \bb{C},\,\rho_i\geq 0, (1-\epsilon)\rho_i+\eta_j\geq 0\\
&\#\lbrace c_{ij}\neq 0, \rho_i+\eta_j\leq N\rbrace <\infty \biggr\rbrace.
\end{align*}
The epsilon can be arbitrarily small, and is to ensure that in a certain degree we have a lower bound on the negativity of the power of $r$. We include it partially to be consistent with \cite{floerfibrations}.

\subsubsection{Grading}
Let $\mc{L}(B)\rightarrow B$ be the fiber bundle with fiber $\mc{L}(B)_b=\text{Lag}(T_b B,\omega)$ the Lagrangian subspaces in $T_b B$. Following Seidel \cite{seidelgraded} an $N$-fold \emph{Maslov covering} of $B$ is the fiber bundle
$$\mc{L}(B)^N\rightarrow E$$
with fibers the unique $N$-fold Maslov cover $\lag^N(T_b B,\omega)$ of $\text{Lag}(T_b B,\omega)$. Such a cover is unique up to isomorphism if $H_1(B)=0$ \cite[Lem 2.6]{seidelgraded}. For an orientable Lagrangian $L_B$, there is a canonical section $s:L_B\rightarrow \mc{L}(B)\vert_L$. A $\bb{Z}_N$-\emph{grading} of $L_B$ is a lift of $s$ to a section 
$$s^N:L_B\rightarrow \mc{L}^N(B)\vert_{L_B}.$$

Hence, we fix a grading on $L$ as a fibered Lagrangian via 
$$\vert x\vert:=\vert \pi(x) \vert_{s^N} +\dim W^+_{X_g}(x) \quad (mod\,\gcd(N,\Sigma))$$
where $\Sigma$ is the minimal Maslov index for $(F,L_F)$.

The \emph{$A_\infty$-algebra} of a fibered $L$ is the family $$A(L)=(CF(L,\Lambda^2),\mu^n)$$ with 
\begin{align}\label{ainfntymaps}
&\mu^n:CF(L,\Lambda^2)^{\otimes n}\rightarrow CF(L,\Lambda^2)\\
&\mu^n(x_1\otimes\cdots x_n)\\
\notag &\quad =
\sum_{x_0, [u]\in\M_\g(L,D,\underline{x})_0} (\sigma(u)!)^{-1}\varepsilon(u)(-1)^{\sum_i i\vert x_i\vert} \mathrm{Hol}_L( u)r^{e_v(u)}q^{e(\pi\circ u)}x_0
\end{align}
where $\varepsilon(u)$ is $\pm 1$ depending on the orientation of the moduli space, $\sigma(\g)$ is the number of interior marked points on $u$, and $\underline{x}:=x_0\otimes x_1\otimes\cdots x_n$ is shorthand. By Theorems \ref{transversality} and \ref{compactnessthm}, there is a comeager selection of fibered perturbation data so that the sum is well defined with coefficients in $\Lambda^2$.

As a typical result, we will prove that the multiplication maps satisfy the $A_\infty$ axioms
\begin{align}\label{ainftyrelations}
0=\sum_{\substack{n,m\geq 0 \\ n+m\leq d}} &(-1)^{n+\sum_{i=1}^n\vert a_i\vert}\mu^{d-m+1}\\
\notag &(a_1\otimes\cdots a_n\otimes\mu^m(a_{n+1}\otimes\cdots a_{n+m})\otimes a_{n+m+1}\otimes \cdots a_d)
\end{align}

\begin{theorem}\label{ainftythm}
For a coherent, regular, stabilized fibered perturbation system, the products in \eqref{ainfntymaps} satisfy the axioms of a $A_\infty$-algebra.
\end{theorem}

\begin{proof}
We show the proof modulo signs and up to reordering the interior marked points, with the subject of orientations taken up precisely in \cite[Thm 4.33]{CW2}. For bounded energy, Theorems \ref{transversality} and \ref{compactnessthm} say that the compactification of the 1-dimensional component of the moduli space 
$$\mathcal{M}(L,D,\underline{x},e(u)\leq k)_1$$
 is a compact 1-manifold with boundary. Thus,
\begin{equation}\label{boundrelation}
0=\sum_{\g\in \mathfrak{M}_{m,n}}\sum_{[u]\in\partial\M_\g(L,D,\underline{x},e(u)\leq k)_1} (-1)^\dagger(\sigma(u)!)^{-1}\mathrm{Hol}_L( u)r^{e_v(u)}q^{e(\pi\circ u)}
\end{equation}
Each boundary combinatorial type is obtained by gluing two types $\g_1$, $\g_2$ along a broken edge that is a root for $\g_1$ resp. leaf for $\g_2$. Since our perturbation is coherent with respect to cutting an edge, we have that
\begin{align*}
\partial \M_\g(L,D,\underline{x})_1 &\cong \bigcup_{y,\g_1,\g_2}\M_{\g_2}(L,D,x_0,\dots,x_{i-1},y,x_{i+1+k}\dots x_n)_0\\
&\quad \times\M_{\g_1}(L,D,y,x_i,\dots,x_{i+k})_0
\end{align*}
Thus, for each boundary $[u]=[u_1]\times[u_2]$, we have that
\begin{align*}
&\mathrm{Hol}_L(u)=\mathrm{Hol}_L(u_1)\mathrm{Hol}_L(u_2)\\
&e_v(u)=e_v(u_1)+e_v(u_2)\\
&e(\pi\circ u)=e(\pi\circ u_1)+e(\pi\circ u_2)
\end{align*}
as well as the analogous statement for orientation signs. Let $m_i=\sigma(u_i)$. Then for each $[u_1]\times [u_2]$ of combinatorial type $\g_1\times\g_2$, there are $m\choose m_1$ ways to distribute the interior markings to each unbroken piece. These observations with \eqref{boundrelation} give us the formula
\begin{multline}
0=\sum_{y,\g_1,\g_2}\sum_{\substack{[u_1]\in \M_{\g_1}(L,D,y,x_i,\dots,x_{i+k})_0 \\ [u_2]\in \M_{\g_2}(L,D,x_0,\dots,x_{i-1},y,x_{i+1+k},\dots, x_n)_0}} (-1)^{\dagger}(m!)^{-1}{m \choose m_1}\\
\times \mathrm{Hol}_L(u_1)\mathrm{Hol}_L(u_2)r^{e_v(u_1)}r^{e_v(u_2)}q^{e(\pi\circ u_1)} q^{e(\pi\circ u_2)}
\end{multline}
which shows \eqref{ainftyrelations} up to signs.
\end{proof}

\subsection{The disk potential}
\label{diskpotentialsection}
We use our lifting operator to explore the $0^{th}$ structure map in \eqref{ainfntymaps}. One can often prove Hamiltonian non-displaceability of a Lagrangian by simply finding critical points of $\mu^0$. We prove a formula that expresses the low energy terms of $\mu^0$ as a sum of terms coming from the base and fiber. This section is a set-up for the main ideas involved in the statement of the main theorem.

\begin{definition}\label{2ndpotdef}
The \emph{second order potential} for a fibered Lagrangian in a symplectic K\"ahler fibration is
\begin{equation}\label{2ndpot}
\mc{W}_L^2[q,r](\rho):=
\sum_{\substack{u\in\mathcal{I}_x\\x\in\T{crit} (f)}} (\sigma(u)!)^{-1}\varepsilon(u)\T{Hol}_\rho(u)q^{e(\pi\circ u)}r^{e_v(u)}x
\end{equation}
where for each $x$
\begin{align*}\mathcal{I}_x &= \bigg\lbrace u\in \mathcal{M}(E,L,x)_0\,\vert e(\pi\circ u)=0\bigg\rbrace \\
&\quad \bigcup\left\{ u\in\mathcal{M}(E,L,x)_0\,\vert e(u)=\min_{v\in \mathcal{M}(E,L,x)_0}\left\{ e(v):e(\pi\circ v)\neq 0\right\}\right\}.
\end{align*}
\end{definition}In words: the second order potential is a sub-sum of $\mu^0$ that counts isolated holomorphic disks contained in a single fiber, along with the holomorphic disks of minimal energy in the total space that project to non-constants. 

For a monotone Lagrangian, one can write down the entirety of $\mu^0$ (i.e. see \cite{oh} or \cite{fooo3} for the toric case). Following notation in the introduction, let $x^M$ be the unique maximal critical point of $g\vert_{L_F}$. Since $L_F$ is monotone and has minimal Maslov number at least $2$, we have that
$$\mu^0_{L_F}(\theta)[r]=\sum_{u\in\mathcal{M}(F,L_F,x_F)_0}\varepsilon(u)\T{Hol}_\theta(u)r^{2\lambda} x^M$$
where $\lambda\mu_F(u)=\int u^*\omega_F$ and the sum is finite (or empty if the minimal Maslov number is greater than $2$). 

For dimensional reasons, the only terms in $\mu_L^0$ that project to constants appear as coefficients of $x^M_M$, the unique maximal index critical point of $f$.

One can define an arbitrary order potential 
$$\mc{W}_L^k [t]:=
\sum_{\substack{u\in\mathcal{I}^k_x\\x\in\T{crit} (f)}} (\sigma(u)!)^{-1}\varepsilon(u)\T{Hol}_\rho(u)t^{e(u)}x$$
with 
\begin{align*}
\mathcal{I}_x^k = \Big\{& u\in \mathcal{M}(L,x)_0\,\vert \nexists \lbrace u_i\rbrace_{1\leq i\leq k}\in \mathcal{M}(L,x)_0:\\
& 0<e(u_1)<\cdots <e(u_k)<e(u) \Big\}.
\end{align*}
This definition will be used to find a Floer-non-trivial Lagrangian in the complete flag manifold of arbitrary dimension \eqref{fullflags}.

Following \cite{CW2}, one can define the $0^{th}$ multiplication map for a rational $L_B$, which fits into the $A_\infty$-algebra for said Lagrangian:
\begin{definition}
For generating critical points $x\in CF(L_B,\Lambda_q)$ and a coherent, stabilized, regular perturbation system of domain dependent data $P_B$, the potential for a rational Lagrangian in a rational symplectic manifold is
\begin{equation}\label{basepot}
\mu^0_{L_B}(\theta)[q]=
\sum_{\substack{x\in \crit b\\u\in \mathcal{M}(B,L_B,y, P_B)_0}}(\sigma(u)!)^{-1}\varepsilon(u)\T{Hol}_\theta(u)q^{e(u)}x
\end{equation}
\end{definition}

Given a representation $\rho \in \text{Hom}(\pi_1(L),(\Lambda^2)^\times)$ and a critical point $x^j$ corresponding to a point in $\pi^{-1}(u(1))$, define a representation $$u\mapsto\T{Hol}_\rho(\mathcal{L}_{x^j}u)$$ for elements in $u\in\mathcal{M}_{\g} (B,L_B,J_B,x)$.

Enumerate the critical points resp. fibers $x_i$ resp. $F_i$, and let $x_i^M$ be the unique index zero critical point of $g\vert_{L_{F_i}}$ in the fiber above $x_i$. Using the lifting operator from Section \ref{liftingopsection}, we perform a transformation on the base potential:
\begin{definition}\label{liftpot}
The \emph{lifted potential} for $L_F\rightarrow L\rightarrow L_B$ is
\begin{equation*}
\mathcal{L}\circ \mu^0_{L_B}(\theta)[q,r] := \sum_{i,\, u\in \mathcal{M}(B,L_B,x_i,J_B)_0}(\sigma(u)!)^{-1}\varepsilon(u)\T{Hol}_{\rho}(\mathcal{L}_{x^M_i} u)q^{e(u)}r^{e_v(\mc{L}u)}x_i^M.
\end{equation*}
\end{definition}

In words, we take each configuration $u\in \mathcal{M}(B,L_B,y_i)$ and compute its Donaldson lift through the stable manifold of $x_i^M$.

\begin{remark}
In the presence of orientations and more general coefficients, there may be cancellations between configurations in \eqref{basepot} that may not occur in the lifted potential. We interpret the lifted potential as first looking at the Moduli spaces, taking lifts, and then performing any algebraic operations. In this sense, it is not an algebraic operation on $\mu^0_{L_B}$.
\end{remark}

\section{Statement and proof of the main theorem}\label{trivfiberedsection}
To relate all of these concepts from Section \ref{diskpotentialsection}, we want to understand how the index changes when we project a single-output configuration, and when we lift such a configuration. First, we prove a short proposition:

\begin{proposition}\label{indexdecprop}
$$0\leq \I (\Upsilon\circ\pi_*\g, \pi(x))\leq \I (\g,x)$$ for a regular type $\g$ for which $\mc{M}_\g\neq \emptyset$.
\end{proposition}

\begin{proof} We compare terms in formula \eqref{indexdefinition}. By Proposition \ref{regular proj}, we know that $0\leq \I[\Upsilon\pi_*\g]$. For the other inequality, we check each component of the index formula: By construction of the pseudo-gradient we have that flow lines project to flow lines, hence $$\dim W^+_{X_g\oplus X_b^\sharp}(x)=\dim W^+_{X_g}(x)+\dim W^+_{X_b}(\pi(x)).$$

Recall the definition of $m_{z_0}(u)$ as the intersection multiplicity of $u$ at $z$ with the divisor. We show the following lemma:
\begin{lemma}\label{multiplicityincreaselemma}
For an interior marked point $z_0\in C_i$ for $u$, we have $m_{z_0}(u) = m_{z_0}(\pi\circ u)$.
\end{lemma}
\begin{proof}[Proof of lemma]
We use the equivalence with the Cieliebak-Mohnke definition \eqref{orderoftangencydef}. Let $U_B$ be a neighbourhood of $\pi\circ u(z_0)$ and $U_F$ a fiber neighbourhood in $F$ of $u(z_0)$ so that $\phi:U_B\times U_F\rightarrow \bb{C}^k\times \bb{C}^l$ is holomorphic and sends $(D_B\cap U_B)\times U_F$ to $\bb{C}^{k-1}\times \bb{C}^l$. In an open set $V$ around $z_0$, we have $u\vert_V=u_B\times u_F$ in these coordinates, and the multiplicity of $u$ at $z_0$ is 
$$m_{z_0}(u)=j+1:\min_{j\geq 0}d^{(j)}_{z_0}u_B\oplus d^{(j)}_{z_0} u_F\in \bb{C}^{k-1}\times \bb{C}^l$$
while the degree of multiplicity of $\pi\circ u$ is
$$m_{z_0}(\pi\circ u)=j+1:\min_{j\geq 0}d^{(j)}_{z_0}u_B\in \bb{C}^{k-1}$$
but $d_{z_0}^{(j)}u_F\in \bb{C}^l$ by definition, so clearly $m_{z_0}(u)= m_{z_0}(\pi\circ u))$.
\end{proof}

To compare indices directly, we have:

\begin{lemma}
For an unbroken and aspherical configuration $\g$, the expected dimension of the moduli space $\mc{M}_{\Upsilon\pi_*\g}(B,L_B,\pi(x))$ for the projected configuration is:
\begin{align}\label{projectedindexformula}
\I(\Upsilon\pi_*\g,\pi(x)) &= \I(\g,x) + [\# \text{forgotten nodes}-\sum_i I_F(u_i)]\\
 \label{projectedindexformulaline2} &\quad +\sum_{z_e\in \text{Edge}^\bullet_\rightarrow}[m_{z_e}(u)-m_{z_e}(\pi\circ u)]-\dim W^+_{X_g}(x).
\end{align}
\end{lemma}
\begin{proof}[Proof of lemma]
This is the result the Riemann-Roch Theorem applied to the $\bar{\partial}$ problem on $\pi\circ u_i^*(TB,TL_B)$ on each surface component, together with assuming that the nodal and tangency conditions are cut out transversely and counting their codimension. Indeed, the index of the linearized operator at $\pi\circ u_i$ is given by $I_B(\pi\circ u_i)+\dim L_B-3$, which is the expected dimension of the space of solutions assuming surjectivity. The requirement to evaluate to $W^+_{X_b}(\pi(x))$ is codimension $\dim W^-_{X_b}(\pi(x))$. On the other hand, the index of the linearized operator at $u_i$ is $I(u_i)+\dim L-3$ where evaluation is codimension $W^-_{X}(x)$. We have $$I_B(\pi\circ u_i)+\dim W^+_{X_b}(\pi(x))=I(u_i)+\dim W^+_X(x)-I_F(u_i)-\dim W^+_{X_g}(x)$$ by definition of the vertical Maslov index and the adapted pseudo gradient, and this accounts for the last terms in lines \eqref{projectedindexformula} and \eqref{projectedindexformulaline2}.

The presence of a boundary node on $\g$ with zero length edge is a codimension $1$ restriction, so forgetting any such nodes after stabilizing with $\Upsilon$ removes this restriction. This accounts for $\#\text{forgotten nodes}$.

Each tangency to the divisor is a codimension $m_u (z_e)$ condition, which accounts for the sum in line \eqref{projectedindexformulaline2}. Since each component with an interior marked point is already stable, it follows that $\Upsilon$ forgets none of these.

Finally, by assumption the middle two terms in \eqref{indexdefinitionline2} are $0$ for both configurations, so they do not appear here.
\end{proof}
Each forgotten node must have at least one associated forgotten surface component. By the fact that $\Sigma_{L_F}\geq 2$ and non-negativity of $\mu_F$ for any $u_i$ (Proposition \ref{verticalmaslovprop}) we have
$$\# \text{forgotten nodes}-\sum_i\mu_F(u_i)\leq 0.$$

Thus, by Lemmas \ref{multiplicityincreaselemma}, \ref{projectedindexformula}, and non-negativity of $W^+_{X_g}(x)$, the proposition follows.
\end{proof}
\begin{remark}
Note that \eqref{projectedindexformula} does not hold as soon as we include any input on our configuration.
\end{remark}

We are now ready to state and prove the main theorem. If $x_M$ is the unique maximal index critical point for $b$ on $L_B$, let $i_{x_M*}:CF(L_F,\Lambda[r])\rightarrow CF(L,\Lambda^2)$ be the map induced from the inclusion $i_{x_B}:L_F\rightarrow L$ of $L_F$ as the fiber of $x_M$.

Let $i^*:\T{Hom}(\pi_1(L),(\Lambda^2)^\times)\rightarrow \T{Hom}(\pi_1(L_F),(\Lambda^2)^\times)$ be the map on representations induced from $i_{x_M*}$.

\begin{theorem}[Main Theorem \ref{maintheoremintro}]\label{potthm}
Let $E$ be a compact symplectic K\"ahler fibration and $L$ a trivially fibered, relatively spin-spin Lagrangian. Let $(P_\g)_\gamma$ be a choice of regular, coherent, stabilized, smooth perturbation datum as per Theorems \ref{transversality}, \ref{compactnessthm}, and \ref{liftconfigthm}. The only terms appearing in the second order potential for $L$ are the vertically constant lifts of index $0$ configurations with output a fiberwise maximum and horizontally constant configurations in the maximal critical fiber
\begin{equation}\label{peq}
\mathcal{W}_{L}^2[q,r](\rho)=\bigg[\mathcal{L}\circ \mu^0_{L_B}(\rho)\bigg]_{\deg_{qr}=\mc{K}_x}+i_{x_B*}\circ\mu^0_{L_F}(\iota^*\rho)
\end{equation}
where for each generating critical point $x$, we have
$$\mc{K}_x=\min \bigg\lbrace e(u):[u]\in \M(E,L,x)_0,\pi\circ u\neq const.\bigg\rbrace.$$
\end{theorem}
It follows that
\begin{corollary}\label{unobstructed}
If $\mu^0_{L_B}(\theta)$ is a multiple of $x_M$ on $L_B$, then $\mu^0_{L}(\rho)$ is a multiple of $x_M^M$ for an appropriate Morse function $g+\pi^*b$ on $L$.
\end{corollary}

\begin{remark}\label{unitremark}
The maximal index critical point $x_M^M\in CF(L,\Lambda^2)$ often functions like a unit in the $A_\infty$ algebra of $L$, and the above corollary is the first step towards showing that one has $(\mu^1_L)^2=0$. To actually achieve this with the domain dependent perturbation system, one has to introduce additional copies of $x_M^M$ to the $A_\infty$ algebra of $L$ and restrict to perturbation data that is coherent with respect to forgetting incoming edges with labels among these copies. We explain this in more detail in Subsection \ref{maurercartansection} and give a simple criterion for Floer cohomology to be unobstructed at the end of Section \ref{divisorequation}.
\end{remark}

\begin{proof}[Proof of theorem]
If we can realize every term in the second order potential \eqref{2ndpot} as either a Donaldson lift of an index $0$ configuration or horizontally constant, then we will be done due to the fact that the second summand on the right hand side of \eqref{peq} always appears in \eqref{2ndpot} and are the only horizontally constant terms appearing for index reasons. The strategy is to take a non-horizontally constant term appearing in \eqref{2ndpot}, project it, lift it, and show that we get the same configuration back.

Let $u$ correspond to horizontally non-constant configuration appearing in \eqref{2ndpot}. By Proposition \ref{indexdecprop} and Lemma \ref{regular proj}, $\pi\circ u$ appears in the potential for $L_B$ \eqref{basepot} before cancellation. By Lemma \ref{indexofliftlemma}, we have that $\mc{L}_{x^i}(\pi\circ u)$ is counted in $\mu^0_{L}$ precisely when $x^i$ is the maximum of a critical fiber (more precisely, when the evaluation of $\mc{L}_{x^i}(\pi\circ u)$ at the output lies in the stable manifold of such a critical point). Thus, $\I(\g,x^M)=\I(\Upsilon\pi_*\g,\pi(x^M))=0$.

\begin{lemma} $u=\mc{L}_{x^M}(\pi\circ u)$ for such a horizontally non-constant $u$ counted in the second order potential \eqref{2ndpotdef}.
\end{lemma}
\begin{proof}[Proof of lemma] First, we rule out that $u$ has any marked (horizontally constant) components. If $u$ has non-constant marked components, then by formula \eqref{projectedindexformula} and the fiber monotonicity, $\pi\circ u$ is of type $\Upsilon\circ \pi_*\g$ with negative index. This contradicts Lemma \ref{regular proj}. Thus, $u$ does not have any marked components. It follows from the non-negativity of all the terms in the RHS of \eqref{projectedindexformula} that they all must be $0$ for $u$, and in particular $I_F(u)=0$. 

It follows from this vertical Maslov index observation, Lemma \ref{vertconst0maslemma}, and regularity from Theorems \ref{liftconfigthm}, \ref{homotopy} that we can use a perturbation data of the form $J^G$ for the combinatorial type $\g$. However, this forces $u$ to be a Donaldson lift and we are done.
\end{proof}
From Proposition \ref{coherentorientations}, it follows that every lifted disk contributes $(+1)$ or $(-1)$ depending on the orientation we choose for $L_F$. We assume that the contribution is $(+1)$.
\end{proof}

The main point of the potential is to allow us to calculate Floer cohomology in simple way. The general theorem that one aims to prove is along the lines of \cite[Thm 4.10]{fooo3}:  If the assumptions in Corollary \ref{unobstructed} hold and we have a critical point of this potential at a particular representation, then the differential at the higher pages of the Morse-to-Floer spectral sequence vanishes and the Floer cohomology is isomorphic to the Morse cohomology. This is detailed in the next subsection.

\section{Unobstructedness and non-displaceability criteria}\label{divisorequation}
We explain how to use corollary \ref{unobstructed} to provide a criteria for the Floer cohomology of $L$ to be defined, and use an idea due to \cite{fooo3}\cite{birancornearigitity} to show when the Floer cohomology is isomorphic to the singular cohomology as a module.

\subsection{The Maurer-Cartan equation and strict units}\label{maurercartansection}
Given an $A_\infty$ algebra $(A,\mu^n)$, a \emph{strict unit}  is an element $e$ such that
\begin{gather*}
(-1)^{\vert x\vert}\mu^2(x,e)=\mu^2(e,x)=x\quad \forall x\in A\\
\mu^n(x_1\otimes\cdots e\otimes\cdots x_n)=0 \quad for\; n\neq 2.
\end{gather*}
By the $A_\infty$ relations, we have that $(\mu^1(x))^2=(-1)^{1+\vert x\vert}\mu^2(x,\mu^0)+\mu^2(\mu^0,x)$, so if $\mu(1)=e$ then $(\mu^1)^2=0$. One can also take an element $b$ of positive $q$ valuation and produce a deformed $A_\infty$ algebra $(A,\mu^n_b)$:
\begin{align}\label{ainftyperturbedrelations}
&\mu^n_b(x_1\otimes\cdots x_n)\\
\notag&\quad :=\sum_{i_0,\dots, i_{n}\in \bb{Z}_{\geq 0}}\mu^{n+i_0+\cdots i_{n}}(b^{\otimes i_0}\otimes x_1\otimes b^{\otimes i_2}\otimes\cdots x_n\otimes b^{\otimes i_{n}}).
\end{align}

Consider the \emph{weak Maurer-Cartan} equation
\begin{equation}\label{mcequation}
\mu(b):=\sum_{n\geq 0}\mu^n(b\otimes\cdots b)=0 \qquad mod \, e.
\end{equation}
For a solution we have $(\mu^1_b)^2=0$.

Expanding upon Remark \ref{unitremark}, we give an overview of the treatment of units in Floer theory from \cite{CW2}. For the sake of simplicity, we work with the more traditional products:
\begin{definition}\label{singlevardef}
For $e(u)=\int_C u^*\omega_{H,K}$, define the \emph{single variable $A_\infty$-maps} $\nu^n:CF(L,\Lambda_t)\rightarrow CF(L,\Lambda_t)$ as
\begin{align}\label{ainftymap1}
&\nu_\rho^n(x_1\otimes\cdots x_n)\\
\notag &\qquad =\sum_{x_0, [u]\in\M_\g(L,D,\underline{x},P_\g)_0} (\sigma(u)!)^{-1} \varepsilon (u) (-1)^{\sum_i i\vert x_i\vert}\mathrm{Hol}_\rho(u)t^{e(u)}x_0.
\end{align}

\end{definition}
We denote the single variable $A_\infty$ maps of a Lagrangian $L$ in the form \eqref{ainftyperturbedrelations} by $\nu^n_{\rho,b}$ or simply $\nu^n_b$ when there is no confusion. Note that the ring homomorphism
\begin{equation}\label{2to1variablemap}
\mathfrak{\Lambda}:\Lambda^2\rightarrow \Lambda_t
\end{equation}
given by
$$\sum_{i,j\geq 0}c_{ij} q^{\alpha_i}r^{\eta_j}\mapsto \sum_{i,j\geq 0}c_{ij}t^{\alpha_i+\eta_j}$$
makes $\Lambda_t$ into a $\Lambda^2$-algebra and induces a linear map $\mf{f}:CF(L,\Lambda^2)\rightarrow\linebreak CF(L,\Lambda_t)$ such that $\mf{f}\circ\mu^n=\nu^n\circ\mf{f}$.

We will call $x_M^M$ the \emph{geometric unit} and define $CF^u (L,\Lambda_q)$ to be the module generated by $(\crit{f}-x_M^M)\cup \lbrace x^\triangledown,x^\blacktriangledown,x^\vee\rbrace$ where
\begin{gather*}
\I (x^\triangledown)=\I(x^\blacktriangledown)=0\\
\I(x^\vee)=-1
\end{gather*}
and we refer to these collectively as the algebraic maxima, and called the associated labelled edges forgettable resp. unforgettable resp. weighted, and take the degree of $x^\vee$ to be the (mod $N$ reduction of) $-1$.  We consider treed disks as in Section \ref{treeddisksection} but with weights $\rho\in [0,\infty]$ attached to edges labelled with $x^\vee$ (in this way, the expected dimension of the moduli space is still \eqref{indexdefinition}). The limit along the edges labelled with one of the algebraic maxima is required to be $x_M$, and are only allowed to be output edges according to the following axiom from \cite{CW2}:

\begin{enumerate}
\item The outgoing edge $e_0$ is weighted with label $x^\vee$ only if there are two leaves $e_1$, $e_2$ exactly one of which is weighted with the same weight, and the other is labelled $x^\triangledown$, and there is a single disk S with no interior markings.

\item The outgoing edge $e_0$ can only be forgettable ($x^\triangledown$) if
either
\begin{enumerate}
\item there are two forgettable leaves, or
\item there is a single leaf $e_1$ that is weighted and the configuration has no interior markings.
\end{enumerate}
\end{enumerate}

It follows that in the context of Corollary \ref{unobstructed} that the potential must be a multiple of $x^\blacktriangledown$ in $CF^u(L,\Lambda_t)$ in case is it a multiple of the geometric unit in $L_B$.

The most important aspect of this setup is that perturbation data is required to be coherent with respect to forgetting incoming edges labelled with $x^\triangledown$. This assumption and an expected dimension argument as in \cite[Thm 4.33]{CW2} shows that $x^\triangledown$ is a strict unit for the $A_\infty$ algebra of $L$ \eqref{ainfntymaps}.

\begin{remark}
An orientation for $W^\pm_X(x^\vee)$ is assigned to agree with that of $W^\pm_X(x_M)\times \bb{R}$, where the second factor is oriented according to the weight parameter on the underlying space of weighted treed disks. In addition, as in \cite[Rem 4.25]{CW2}, constant trajectories connecting $x^\vee$ with $x^\triangledown$ resp. $x^\blacktriangledown$ are assigned orientations that agree resp. disagree with the standard orientation on $W^+_X(x_M)$.
\end{remark}

The addition of the weighted maximum provides a solution for the Maurer-Cartan equation in some cases. We have the following:

\begin{lemma}\cite[Lem 4.43]{CW2} \label{mclemma} Suppose that $\nu_\rho^0\in \Lambda_t x^\blacktriangledown$ and every non-constant disk in $(E,L)$ has positive Maslov index. Then there is a solution to the Maurer-Cartan equation \eqref{mcequation} for the $A_\infty$ algebra of $L$.
\end{lemma}

\begin{proof}[Sketch of proof]
Let $\nu^0(1)=\mc{W}x^\blacktriangledown$. In general we have the formula
\begin{equation}
\nu^1(x^\vee)=x^\triangledown-x^\blacktriangledown+\sum_{\substack{x, [u]\in \mc{M}_\g(L,D,x^\vee,x)_0\\e(u)>0}}\pm t^{E(u)}x
\end{equation}
(\cite[49]{CW2}, c.f. \cite[eq 3.3.5.2]{fooo}). By the Maslov index assumption there are no $t$ terms, and moreover $\nu^n(x^\vee\otimes\cdots x^\vee)=0$. We have
\begin{gather*}
\nu(\mc{W}x^\vee)=\nu^0(1)+\nu^1(\mc{W}x^\vee)=\mc{W}x^\blacktriangledown+\mc{W}(x^\triangledown-x^\blacktriangledown)=\mc{W}x^\triangledown\in \Lambda_t x^\triangledown
\end{gather*}
which proves the lemma.
\end{proof}

\subsection{The weak divisor equation}

We use the critical points of the potential as a tool for computing $HF(L,\Lambda_t)$.

Certain moduli spaces defined via divisorial perturbations satisfy a weak form of the \emph{divisor equation} from Gromov-Witten theory. With this principle one can use some intersection theory to count the number of holomorphic disks in a certain moduli, which allows us to get a relation between $\mu^1$ and the derivative of $\mu^0$.

The relation involves forgetting the incoming edge from a Morse 1-cocycle on configurations counted by $\mu^1$. First, we must check that there is perturbation data which is coherent with respect to this action: Let $\gamma=\sum_{i=1}^k x^k$ be a Morse $1$-cocycle. Given a non-Morse configuration counted in $\mu^1(\gamma)$, one can forget the incoming edge and obtain a configuration counted by $\mu^0$ for \emph{some} perturbation datum. Thus, we could ask if it is possible to assume that there is regular perturbation datum which is coherent with respect to forgetting the single incoming edge from Morse 1-cocycles. In the rational case, this is summed up in \cite[Lem 2.35]{CW2}, which we now adapt to our case. First, let $\g$ be the above combinatorial type and $f(\g)$ be the type obtained by forgetting the incoming edge and stabilizing in such a way that $f(\g)$ is $\pi$-stable. Note that this process would only involve possibly forgetting ghost components on $\pi_*\g$ by the assumption that $L_B$ is rational and $D_B$ is stabilizing.

Given a perturbation datum on $f(\g)$, we obtain one on $\g$ via pullback:
$$P_\g=f^*P_{f(\g)}.$$
We have the following adaptation of \cite[Lem 2.35]{CW2} to our case
\begin{lemma}\label{forgetfulperturbationlemma}
Let $\g$ be a combinatorial type of $\pi$-adapted pseudoholomorphic treed disk with a single input of expected dimension at most one and suppose $J_{D_B}\in \mc{J} (B,\omega_B)$ is such that all $J_{D_B}$-holomorphic disks $u$ in $B$ with boundary in $L_B$ have positive Maslov index. For a comeager set of perturbations $\mc{P}_{f(\g)}^{reg}$, if $P_{f(\g)}\in \mc{P}_{f(\g)}^{reg}$ then the moduli space $\mc{M}_\g(E,L,f^*P_{f(\g)})$ is smooth of expected dimension.
\end{lemma}
The proof is almost exactly the same as in \cite{CW2}. By Proposition \ref{verticalmaslovprop}, the positive Maslov assumption in the base implies the same in the total space. One argues by cases on the whether the incoming disk component is constant or not. If it is not constant, then $f(\g)$ does not forget any disk components and the perturbation data as in Definition \ref{perturbationdatadef} is sufficient to achieve transversality. If the incoming disk is constant and must be forgotten, then one argues that a domain dependent Morse function is sufficient to achieve transversality on the remaining edge.

The second step in the derivation of the divisor equation is to write $\partial_\rho\mu^0_\rho$ in terms of $\mu^1(\gamma)$. To this end, we define a smooth manifold structure on $\text{Hom}(\pi_1(L),\Lambda^{2\times})$. In order for the exponential sum to converge, we have to introduce a topology as in \cite{fooocyc}.
\begin{definition}
Let $$\textbf{x}_k=\sum_{i,j}c^k_{ij}q^{\rho_i}r^{\eta_j}\in\Lambda^2,\quad \textbf{x}=\sum_{i,j}c_{ij}q^{\rho_i}r^{\eta_j}$$ be in $\Lambda^2$. We say that $\textbf{x}_k$ converges to $\textbf{x}$ if each $c_{ij}^k$ converges to $c_{ij}$ in $\C$.

\end{definition}

\begin{lemma}
For $p\in\Lambda^2$, the sequence $\sum_{n=0}^N\frac{p^n}{n!}$ converges in the above topology.
\end{lemma}
\begin{proof}
Define $$\exp(p)=\sum_{n=0}^\infty\frac{p^n}{n!}.$$
If this is well defined in $\Lambda^2$, then it is clear that this is the limit of the sequence. Let $p=\sum_{i,j\geq 0}c_{ij} q^{\rho_i} r^{\eta_j}\in\Lambda^2$ with $\rho_0$ and $\eta_0$ such that $\rho_0+\eta_0= \min_{i,j} \rho_i+\eta_j$. Then $$\exp(p)=1+\sum_{n=0}^\infty c_{00}^n q^{n\rho_0}r^{n\eta_0}/n! +\T{terms with higher valuation}.$$ From $\rho_0+\eta_0\geq 0$ and $(1-\epsilon)\rho_0+\eta_0\geq 0$, we have that $\rho_0+\eta_0=0$ if and only if $\rho_0=0$ and $\eta_0=0$, so we take two cases. First let us assume that $\rho_0+\eta_0>0$. Then $n\rho_0+n\eta_0\rightarrow \infty$, so there are finitely many non-zero coefficients below any bounded degree.\\

If $p=c_{00}+\T{higher degree terms}$ with $c_{00}\neq 0$, then $$\exp(p)=1+\sum_{n=1}^\infty c_{00}^n/n!+\sum_{n=1}^\infty c_{00}^{n-1}q^{\rho_1}r^{\eta_1}/n!+\cdots$$ It is also clear that $\exp(p)\in\Lambda^2$ in this case.
\end{proof}

It follows that we can endow the space of representations 
$$\T{Hom}(\pi_1(L),\Lambda^{2\times})\cong H^1(L,\Lambda^{2\times})$$
 with the structure of a smooth manifold using $\exp$. Indeed, the exponential allows $\Lambda^{2\times}$ to inherit a $\Lambda^2$-module structure and we can form the tensor product
$$H^1(L,\Lambda^2)\otimes_{\Lambda^2}\Lambda^{2\times}.$$

If we assume that $H^1(L,\Lambda^2)$ is free, then 
$$H^1(L,\Lambda^2)\otimes_{\Lambda^2}\Lambda^{2\times}\cong H^1(L,\Lambda^{2\times}),$$
 and we have a surjective map given by $H^1(L,\Lambda^2)\rightarrow H^1(L,\Lambda^2)\otimes_{\Lambda^2}\Lambda^{2\times}$ that is locally an isomorphism. By universal coefficients followed by the Hurewicz homomorphism we have 
$$H^1(L,\Lambda^{2\times})\cong \T{Hom}(H_1(L),\Lambda^{2\times})\cong \T{Hom}(\pi_1(L),\Lambda^{2\times}).$$
 This gives us a tangent space at each point. It only remains to check that the transition maps are smooth, which we leave to the reader.

From the above discussion, both $\mu^0_L$ and the second order potential define smooth maps $H^1(L,\Lambda^{2\times})\rightarrow CF(L,\Lambda^2)$. The derivative of $\mu^0_\rho$ in the direction $\gamma\in H^1(L,\Lambda^2)$ becomes evaluation of the $1$-cycle $[\partial u]$ at $\gamma$, i.e.
\begin{equation}
\partial_\gamma \mu^0_\rho= \sum_{u\in \mathcal{M}(L,D,x)_0}(\sigma(u)!)^{-1}\varepsilon (u)\T{Hol}_\rho(u)\gamma([\partial u]) q^{e(\pi\circ u)}r^{e_v(u)}x.
\end{equation}

Turning to $\mu^1_\rho$, \cite[Prop 2.36]{CW2} explains that 
\begin{align}
-\mu^1_\rho(\gamma)=&\sum_{u\in \mathcal{M}(L,D,x)_0}(\sigma(u)!)^{-1}\varepsilon (u)\T{Hol}_\rho(u)\langle[\partial u]^{PD},\gamma\rangle q^{e(\pi\circ u)}r^{e_v(u)}x
\end{align}
which follows from using compatible perturbations as in Lemma \ref{forgetfulperturbationlemma}, and transversality of $\partial u$ to unstable manifolds of critical points. Thus, the \emph{weak divisor equation}
\begin{equation}
-\mu^1_\rho(\gamma)=\partial_\gamma \mu^0_\rho
\end{equation}
shows that at a critical point for $\mu_\rho^0$, the Floer differential vanishes on all Morse $1$-cocycles.

In practice, it is much easier to search for critical points when the coefficients come from the single variable Novikov ring $\Lambda_t$ \eqref{uninovikovdef}: Similar to the above discussion, the weak divisor equation
$$-\nu^1_\rho(\gamma)=\partial_\gamma \nu^0_\rho$$ holds for Morse 1-cocycles in \eqref{singlevardef}. However, in our setup we work with a non-trivial solution to the Maurer-Cartan equation (Lemma \ref{mclemma}). Thus, we need to check the following:

\begin{lemma}
Suppose all non-constant holomorphic disks have Maslov index at least $2$. Then $\nu^1_{\rho,\mc{W}x^\vee}(\gamma)=\nu^1_{\rho}(\gamma)$ on Morse 1-cocycles. Hence 
\begin{equation}\label{wdequationmc}
\nu^1_{\rho,\mc{W}x^\vee}(\gamma)=\partial_\gamma \nu^0_\rho.
\end{equation}
\end{lemma}
\begin{proof}
Assume we have chosen appropriate perturbation data so that all holomorphic disks for index $0$ configurations intersect the divisor transversely or not at all. For $i_0+\cdots i_n\geq 1$ in the definition \eqref{ainftyperturbedrelations}, transversality and the index formula \eqref{indexdefinition} tells us that no non-Morse configurations contribute, with Morse configurations only possible for $i_0+\cdots i_n=1$. In the latter case, the only possibility is a configuration with output $x^\vee$, which is impossible since there are no flow lines from $\gamma$ to the geometric unit. The lemma follows.
\end{proof}

\subsubsection{Relation between Floer cohomologies for different rings}
We make sure that the Floer cohomology induced from $\mu^1$ can be related to that of $\nu^1$. For questions of displaceability, we only need to consider $HF(L,\Lambda_t^{(0)})$ where the coefficients come from the universal Novikov field:
\begin{equation}
\Lambda_t^{(0)}:=\left\lbrace \sum_{i} c_{i} t^{\alpha_i} \vert c_{i}\in \bb{C},\alpha_i\in \bb{R}, \#\lbrace c_{i}\neq 0, \alpha_i\leq N\rbrace <\infty \right\rbrace.
\end{equation}

We have
$$HF(L,\Lambda^2)\otimes_{\Lambda^2}\Lambda^{(0)}_t\cong HF(L,\Lambda_t^{(0)})$$
where $\Lambda_t^{(0)}$ is a $\Lambda^2$-algebra under the map $\exp\circ \mf{f}$.

Let 
$$\mathcal{W}_L^{2}[t]:=\mathcal{W}_L^2[q,r]\otimes 1\in CF(L,\Lambda_t)$$ denote the \emph{single variable second order potential} for $L$. We will use the following theorem extensively in our applications, c.f. \cite[Thm 4.10]{fooo3} \cite[Prop 2.37]{CW2}:

\begin{theorem}\label{crit}
Suppose that $\mu^0$ and hence $\nu^0$ is a multiple of the geometric unit, $\nu^0=\mc{W}x^M_M$. Suppose that for $\rho\in \T{Hom}(\pi_1(L),\Lambda_t^\times)$, $D_\rho W_L^{2}[t]=0$, the Hessian $D^2_{\rho}W_L^{2}[t]$ is surjective, and further that $H^*(L,\Lambda_t)$ is generated by degree one elements via the cup product. Then there is a representation $\xi\in\T{Hom}(\pi_1(L),\Lambda_t^\times)$ so that
$$H(CF(L,\Lambda_t),\nu^1_{\xi,\mc{W}x^\vee})\cong H^{Morse}(L,\Lambda_t).$$
\end{theorem}

\begin{proof}
Let $K$ in the minimal coupling form be large enough so that for all classes $[u]\in H_2(E,L)$ with holomorphic representatives and $\pi_*[u]\neq 0$ we have
$$\int_D u^*\omega_{H,K}>\lambda\Sigma_{L_F}$$
where $\lambda\Sigma_{L_F}$ is the energy quantization for $(L_F,\omega_F)$. Then $W^2_L[t]$ consists of the first and second order terms of $\nu^0_{L}$. Therefore, we can apply the induction argument from the proof of the strongly non-degenerate case of \cite[Thm 10.4, Lemma 10.16]{fooo3} to produce a critical point $\xi$ of $\nu^0$: here we use that the Hessian is non-degenerate and solve for $\xi$ order-by-order. By the weak divisor equation \eqref{wdequationmc} we have that the Floer differential $\nu_{\xi,\mc{W}x^\vee}^1$ vanishes on Morse 1-cocycles. 

To argue that $\nu_{\xi,\mc{W}x^\vee}^1$ vanishes on all Morse cocycles, we use an induction on cohomological degree and symplectic area, similar to \cite[Lem 13.1]{fooo3}, \cite[Prop 2.31]{CW2}. Recall the existence of a Morse-to-Floer spectral sequence that is induced by the energy filtration on $CF(L,\Lambda_t)$ \cite[Sec 8]{ohspec}. The $0^{th}$ differential is simply the Morse differential.
\begin{lemma}
$\nu_{\xi,\mc{W}x^\vee}^1(\gamma)=0$ for any Morse cocycle $\gamma$ on the first page of the Morse-to-Floer spectral sequence.
\end{lemma}
\begin{proof}
By the $A_\infty$ relations,
\begin{align*}
\nu^1(\nu^2(\gamma_1\otimes\gamma_2))&=\nu^2(\nu^1(\gamma_1)\otimes\gamma_2))\pm\nu^2(\gamma_q\otimes\nu^1(\gamma_2))\\ 
&\quad\pm \nu^3(\gamma_1\otimes\gamma_2,\nu^0)\pm \nu^3(\gamma_1\otimes\nu^0\otimes\gamma_2)\pm\nu^3(\nu^0\otimes\gamma_1\otimes\gamma_2)
\end{align*}
where we shorten notation $\nu^n_{\xi,\mc{W}x^\vee}=:\nu^n$. We deduce that the $\nu^3$ terms are zero due to the assumption $\mu^0$ is a multiple of the $x^\blacktriangledown$ and hence $\nu^0$ is a multiple of a strict unit by Lemma \ref{mclemma}. Let $\nu^n_\beta$ be the sum of terms in $\nu^n$ that are a $\bb{C}$ multiple of $t^\beta$ for $\beta\geq 0$. We have
\begin{align*}\sum_{\beta_1+\beta_2=\beta}\nu^1_{\beta_1}(\nu^2_{\beta_2}(\gamma_1\otimes\gamma_2))&=\sum_{\beta_1+\beta_2=\beta}\nu^2_{\beta_1}(\nu^1_{\beta_2}(\gamma_1)\otimes\gamma_2)\\
&\quad +\sum_{\beta_1+\beta_2=\beta}\nu^2_{\beta_1}(\gamma_1\otimes\nu^1_{\beta_2}(\gamma_2))
\end{align*}

Note that $\nu^2_0(\gamma_1\otimes\gamma_2)=\gamma_1\cup \gamma_2$ since $\mc{W}x^\vee$ is either zero or has positive $t$-valuation. Thus
\begin{align*}\nu^1_\beta(\gamma_1\cup \gamma_2)=&\sum_{\beta_1+\beta_2=\beta}\nu^2_{\beta_1}(\nu^1_{\beta_2}(\gamma_1)\otimes\gamma_2))+\sum_{\beta_1+\beta_2=\beta}\nu^2_{\beta_1}(\gamma_1\otimes\nu^1_{\beta_2}(\gamma_2))\\&+\sum_{\beta_1+\beta_2=\beta,\, \beta_2>0}\nu^1_{\beta_1}(\nu^2_{\beta_2}(\gamma_1\otimes\gamma_2)).
\end{align*}
Using this expression, we proceed by strong induction on Morse cohomological degree of $\gamma_1\cup\gamma_2$ and the energy $\beta$. For $d_i\in\mathbb{N}_0$ and $\alpha_i\in\R_{\geq 0}$, say that $(\alpha_1,d_1)\leq (\alpha_2,d_2)$ if $\alpha_1 <\alpha_2$ or $\alpha_1 =\alpha_2$ and $d_1<d_2$. The base step $(\beta,\deg\gamma_1\cup\gamma_2)=(0,d)$ is given, since the computation takes place on the first page of the Morse-to-Floer spectral sequence and the $\gamma_i$ are cocycles. In the induction step the first two terms on the right hand side vanish by the induction hypothesis, since $(\beta_2,\deg (\gamma_i))< (\beta,\deg (\gamma_1\cup \gamma_2))$. The last term on the right hand side also vanishes by the induction hypothesis since $\beta_1<\beta$. This proves the lemma.
\end{proof}

It follows that the Morse-to-Floer spectral sequence collapses after the first page, and 
$$H(CF(L,\Lambda_t),\nu^1_{\xi,\mc{W}x^\vee})\cong H^{Morse}(L,\Lambda_t).\vspace{-2em}$$
\end{proof}

\subsection{Hamiltonian isotopy invariance and obstruction to displaceability}\label{hamiltonianinvariancesection}
The main focus of this paper is to deduce when a Lagrangian in a certain class is Hamiltonian non-displaceable from itself. As such, we don't prove complete invariance of input data in the sense that a different perturbation data selection, Maurer-Cartan solution, and Rank 1 representation will give a quasi-isomorphic $A_\infty$ algebra. Any Hamiltonian perturbation of $L$ gives an isomorphism of the moduli spaces of Floer trajectories for the pullback perturbation. The pullback also preserves the symplectic area of each configuration. Thus our definition of $H(CF(L,\Lambda_t),\nu^1_{\xi,\mc{W}x^\vee})$ is a Hamiltonian isotopy invariant for each perturbation data and fibered divisor. More generally, the results in \cite{CW2} on invariance of input data up to $A_\infty$ homotopy equivalence, e.g. divisor (Theorem 5.12), perturbation data (Corollary 5.11), and Maurer-Cartan element (Lemma 4.42)  ``lift'' to our situation when we vary our input data across pullbacks of stabilizing divisors and fibered perturbation data, and for a general Maurer-Cartan element. To see that the non-vanishing of our fibration invariant over a Novikov field obstructs displaceability, we mimic the construction from \cite[Sec 6]{CW2} to sketch the definition of an $A_\infty$ bimodule structure on $CF(L,\phi(L),\Lambda_t)$ for a Hamiltonian diffeomorphism $\phi$ such that $L\cap\phi(L)$ \emph{cleanly} ($L\cap\phi(L)$ is a manifold with tangent space $TL\cap T\phi(L)$). We further sketch the fact that different choices of $\phi$ give quasi-isomorphic $A_\infty$ bimodule structures, and thus give isomorphisms on cohomology.

First, by the Theorems \ref{transversality}\ref{compactnessthm}, we have a coherent Baire set of domain dependent, tamed almost complex structures with which we can define the $A_\infty$ algebra of $\phi(L)$. Namely, the perturbation data for $\phi(L)$ is taken as pullback $\phi^*P_\g$ for any system from the aforementioned lemmas. It follows that the holomorphic curve counts for $(E,\phi(L),(\phi^*P_\g)_\gamma,\phi_*D)$ are in bijection with those of $(E,L,(P_\g)_\gamma,D)$. Thus, it makes sense to talk of $(CF(\phi(L),\Lambda_t),\mu^n_{\phi(L)})$, a Maurer-Cartan solution $\phi_*(\mc{W}x^\vee)$ from Lemma \ref{mclemma}, and thus the homology of this complex.

Next, assuming the clean intersection condition, one can put Morse-Smale functions $F$ on $L\cap \phi(L)$ and form the space of Floer cochains as

$$CF(L,\phi(L),\Lambda_t):=\bigoplus_{x\in \crit(F)}\Lambda_t\cdot x.$$
One can place an $A_\infty$ bimodule structure on this, that is, maps
$$\mu^{a\vert b}:CF(L,\Lambda_t)^{\otimes a}\otimes CF(L,\phi(L),\Lambda_t)\otimes CF(L,\Lambda_t)^{\otimes b}\rightarrow CF(L,\phi(L),\Lambda_t)$$
which satisfy
\begin{align*}
0&=\sum_{i,k}(-1)^{\dagger}\mu^{n-i\vert m}(a^0_1\otimes\cdots\mu^i_L(a^0_{k+1}\otimes\cdots a^0_{k+i})\\[-1ex]
&\hspace{8em}\otimes a^0_{k+i+1}\cdots \otimes m\otimes a^1_{1}\otimes\cdots  a^1_m)\\
&\quad +\sum_{j,l}(-1)^{\dagger}\mu^{n\vert m-j}(a^0_1\otimes\cdots a^0_{n} \otimes m\otimes a^1_{1}\otimes\cdots\\[-1ex]
&\hspace{9em} \otimes\mu^j_{\phi(L)}(a_{l+1}\otimes\cdots a_{l+j})\otimes a_{l+j+1}\otimes\cdots a^1_m)\\
&\quad+\sum_{j,k}(-1)^{\dagger}\mu^{k\vert m-j}(a^0_1\otimes\cdots\mu^{n-k\vert j}(a^0_{k+1}\otimes\cdots  a^0_{n} \otimes m\otimes a^1_{1}\otimes\cdots  a^1_{j})\\[-1ex]
&\hspace{9em}\otimes a^1_{j+l+1}\otimes\cdots a^1_m)
\end{align*}
where $\dagger$ follows the same sign convention as in \eqref{ainftyrelations}. Indeed, the maps $\mu_{\phi}^{0\vert 0}$ can be defined as follows: Let $\Sigma:=\bb{R}\times [0,1]$ as a surface with a complex structure obtained from identifying it with a piece of $\bb{C}$. Given a time dependent (zero-average) Hamiltonian $h_t$ with associated flow $\phi_t$ we define a connection 1-form on $\Sigma\times E$ by $H:=h_t dt$, where $t$ is the $[0,1]$ coordinate. Then $\sigma\times E$ has a minimal coupling form $\omega_E-dH$, and given any almost complex structure $J$ on $E$ we obtain a $J_\Sigma$ on $\Sigma\times E$ from Lemma \ref{connectionacs} that fiberwise agrees with $\phi_t^*J$. Moreover, we have a complex codimension $1$ hypersurface $D_\sigma:=(\phi_{t,*}(D))_t$. Then, $\mu^{0\vert 0}_\phi$ counts nodal treed strips, that consist of Morse flowlines on $L\cap \phi_1(L)$ together with $J_\Sigma$-holomorphic sections $\sigma$ of $\Sigma\times E$ up to translation in the $s$ variable such that for each $\sigma$ we have
\begin{gather*}
\sigma(s,0)\in L\\
\sigma(s,1)\in \phi_1(L)
\end{gather*}
and in addition on leaf resp. root strip components we have
\begin{gather*}
\lim_{s\rightarrow \infty}\in W^+_F(x)\\
\lim_{s\rightarrow -\infty}\in W^-_F(y)
\end{gather*}
We call such a thing a $0\vert 0$ treed strip. We can require $\sigma$ to be \emph{adapted} in the sense of the \emph{(non-constant spheres)} and \emph{(markings)} axioms from Definition~\ref{adapteddef}. More generally we can consider $n\vert m$ treed strips based on a connected nodal curve, as in \cite[Sec 6.2]{CW2}, that have strip components with a designated input and output and in addition consist of adapted holomorphic treed disks as in definition \ref{treeddiskdef} for $L$ resp. $\phi_1(L)$ which lie over $t=0$ resp. $t=1$ with $n$ resp. $m$ leaves. Transversality for treed disk subcomponents is taken care of by Theorem \ref{transversality}. Transversality on non-vertically constant strip components can be realized by taking a $t$ dependent almost complex structure $J_t$ on $E$ that agrees with $J$ resp. $\phi_1^* J$ for $t=0$ resp. $1$ that is adapted to $D_\Sigma$. Transversality for the vertically constant case is similar to the proof of Theorem \ref{liftconfigthm}, since $\Sigma\times E$ is trivializable by parallel transport.

Counting such isolated configurations gives rise to maps $\mu_{\phi}^{n\vert m}$ which satisfy the $A_\infty$ bimodule axioms, and the proof of this fact is similar to that of Theorem \ref{ainftythm} (see \cite[Thm 6.4]{CW2}).

In particular, for $\phi_t\equiv Id$, the bimodule structure on $CF(L,L,\Lambda_t)$ reduces to the $A_\infty$ structure on $CF(L,\Lambda)$.

Finally, given two Hamiltonian isotopies $\phi_t^i$ $i=0,1$, one shows that there is a map of $A_\infty$ bimodules $\Phi^{01}$ by counting treed strips equipped with a parameterization which limit toward a stable manifold on $L\cap \phi_1^1(L)$ in the $\infty$ direction and an unstable manifold on $L\cap \phi_1^0(L)$ in the $-\infty$ direction, as in \cite[Sec 6.5]{CW2}. One shows that $\Phi^{10}\circ\Phi^{01}$ is $A_\infty$ homotopic to the identity by counting treed strips with \emph{two} parametrizations. Thus, the non-vanishing of $HF(L,\Lambda_{nov,t})$ implies the existence of a element $x\in L\cap \phi_1(L)$ for any time $1$ map of a Hamiltonian diffeomorphism where the intersection is clean, and in particular shows that $L$ is non-displaceable.

\section{Applications}\label{applicationsection}

\subsection{A Floer-non-trivial Lagrangian in $\text{Flag}(\bb{C}^3)$}\label{flag3}

The symplectic topology of such a space is studied in \cite{nish}\cite{chokimoh}\cite{pabiniak} by using a toric degeneration of the $U(3)$ Gelfand-Cetlin system. In particular, \cite{nish} show that one can find a Floer non-trivial torus in the interior of the G.-C. polytope (as well as in other flags), and \cite{pabiniak} finds a whole line segment of non-displaceable tori in the monotone flag manifold by looking at a toric degeneration of the G.-C. system, and \cite{chokimoh} expand upon this by showing non-displaceability of Lagrangian non-torus fibers and drawing connections to mirror symmetry. We begin by showing that one can find at least one of these non-displaceable Lagrangians in the fibration context. In the next example we show that there is a whole family of non-displaceable Lagrangians after deforming the symplectic form.

Consider the moduli space of nested complex vector spaces $V_1\subset V_2\subset \C^3$. We can realize this as a symplectic fiber bundle $$\mathbb{P}^1\rightarrow \mathrm{Flag}(\mathbb{C}^3) \xrightarrow{\pi} \mathbb{P}^2,$$ with both the base and fiber monotone. Holomorphic (but not symplectic) trivializations for $\text{Flag}(\C^3)$ can be realized as follows. Start with a chain of subspaces $V_1\subset V_2\subset \C^3$ with $V_1\in \mathbb{P}^2$ represented as $[z_0,z_1,z_2]$ with $z_0\neq 0$. Using the reduced row echelon form, there is a unique point in $\mathbb{P}(V_2)$ with first coordinate zero, $[0,w_1,w_2]$. Let $U_0$ be the open set of $\mathbb{P}^2$ with non-zero first coordinate. We get a trivialization 
\begin{align*}&\Psi_0:\mathrm{Flag}(\mathbb{C}^3)\rightarrow U_0\times\mathbb{P}^1
\\&([z_0,z_1,z_2],V_2)\mapsto ([z_0,z_1,z_2],[w_1,w_2])
\end{align*}
Let $U_1\subset \bb{P}^2$ be the open set with $z_1\neq 0$. The transition map $U_0\times \mathbb{P}^1\rightarrow U_1\times\mathbb{P}^1$ is given by $$g_{01}([w_1,w_2])=[-\frac{z_0w_1}{z_1},w_2-\frac{z_2w_1}{z_1}]$$
which is a well defined element
$$\begin{bmatrix}
\frac{-z_0}{z_1} & 0 \\
\frac{-z_2}{z_1} & 1
\end{bmatrix}$$ in $PGL(2)$. A similar transition matrix works for the other trivializations.

There is a natural symplectic form that we could use given by viewing Flag$(\C^3)$ as a coadjoint orbit, called the Kostant-Kirillov symplectic form. Let $\xi\in SU(3)\cdot t\subset \mf{su}(3)^\vee$ be an element of the orbit and $X^\sharp,Y^\sharp$ be vector fields on the orbit generated by $X,Y\in \mf{su}(3)$. Let
$$\tilde{\omega}_\xi (X^\sharp,Y^\sharp) = \xi([X^\sharp,Y^\sharp]).$$
The kernel of $\tilde{\omega}_\xi$ is $\mf{stab}(\xi)$, the Lie algebra of the stabilizer of $\xi$ for the coadjoint action. Thus, $\tilde{\omega}_\xi$ descends to the quotient $SU(3)/\mf{stab}(\xi)$ as a non-degenerate skew-symmetric form $\omega_\xi$. Such a form is closed by the Jacobi identity (see e.g. \cite[Exercise 5.3.13]{ms1}).

To realize the symplectic fibration structure, one can see \cite[Thm 2.3.3]{guillemin} for instance. Briefly, the projection $\pi$ is equivariant with respect to the $SU(3)$ action, so the fibers $F_p$ inherit an action by $\mf{stab}(\pi(p))\cong SU(2)$ and we can realize $\text{Flag}(\bb{C}^3)$ as an associated bundle to a principal $SU(2)$-bundle $P$. Any $SU(2)$ connection $H$ on $P$ induces one on $\text{Flag}(\bb{C}^3)$, and such a connection is necessarily Hamiltonian. Thus, by Theorem \ref{coupling} there is a minimal coupling form $a_H$ associated to $H$ and we can choose a weak coupling form as
$$\omega_{H,K}:= a_H+K\omega_{\bb{P}^2}$$
where $K>>1$ and $\omega_{\bb{P}^2}$ is the Fubini-Study symplectic structure on $\bb{P}^2$.

\subsubsection{Constructing a fibered Lagrangian over a toric fiber}\label{laginflagssection}
First, we recall results about toric Lagrangians in the base: Any point in the interior of the moment polytope for the toric action on $\bb{P}^2$ corresponds to a Lagrangian torus. In \cite{fooo3}, the authors calculate the full potential for such a torus by developing the virtual fundamental cycle perturbation technique, and show that the barycenter torus (the Clifford torus) is Floer-non-trivial, recovering the result of \cite{cho}. With our perturbation system, one can take the divisor in the base to be the union of the edges of the polytope
$$D_B:=\bigcup_{i=1,2,3} \mc{D}_i.$$ Such a divisor is stabilizing for the standard almost complex structure in the sense of Definition \ref{stabilizingdivisordefinition} for any of these tori for topological reasons. Take $L_B$ of the form 
\begin{equation}\label{toruscoordinates}
[e^{i\theta_1}a,e^{i\theta_2}b,c]
\end{equation}
for some $a,b,c\in \bb{C}^\times$ and $\theta_i\in S^1$. From \cite{choohdiskinstantons} we have that the Maslov index two disks with boundary on $L_B$ in the standard complex structure are regular and given by
\begin{gather*}
u_1(z)= [za,b,b]\\
u_2(z)=[a,zb,c]\\
u_3(z)=[a,b,zc]
\end{gather*}
and up to a M\"obius transformation fixing the boundary of the unit disk in $\bb{C}$. Let $x_M$ denote the maximal index critical point of a Morse-Smale function on $L_B$. By \cite[Prop 9.1]{choohdiskinstantons} there is a natural spin structure on the each moment torus.  If the moment polytope of $\p^2$ given by
$$\mc{P}^2:=\lbrace (\eta_1,\eta_2)\in \bb{R}^2\vert 0\leq \eta_1,\eta_2, \eta_1+\eta_2\leq 1\rbrace,$$ then the full potential is given by
\begin{equation}\label{fooocliffpotential}
\mu^0_{L_B}(x_1,x_2)=\left[ x_1q^{\eta_1}+x_2q^{\eta_2}+\dfrac{1}{x_1x_2}q^{1-\eta_1-\eta_2}\right]x_M
\end{equation}
in the standard toric complex structure, where the exponents are directly proportional to the area of the corresponding disk.

To find a fibered Lagrangian we use a result from \cite[Sec 4.6]{guillemin} to obtain a symplectic trivialization above $\text{Cliff}(\bb{P}^2)$: 
\begin{theorem}(\cite{guillemin}, Theorem 4.6.3)\label{moment}
Let $F\rightarrow E\rightarrow B$ be a symplectic fibration with a $G$-action where the projection is equivariant. Denote $\psi$ as the moment map for the action of $G$ on $B$. Let $\Delta$ be an open set of the moment polytope, and set $U=\psi^{-1}(\Delta)$. After possibly shrinking $U$ there is a symplectic connection $H$ on $\pi^{-1}(U)$ such that the moment map for action on the weak coupling form $a_H+\pi^*\omega_U$ is $\psi\circ\pi$.
\end{theorem}
See Section 4 of \cite{guillemin} for a proof.
We take $G=T^3/diag$, which acts on $\text{Flag}(\bb{C}^3)$ via
$$V_1\subset V_2\mapsto t\cdot V_1\subset t\cdot V_2.$$ Such an action projects to the standard torus action on $\bb{P}^2$. Theorem \ref{moment} gives us a connection $\g$  and an open subset $\Delta$ of the moment polytope for $\bb{P}^2$ containing $\psi(L_B)$ such that parallel transport around a moment fiber $T\subset \bb{P}^2$ is given precisely by the above torus action. Such a connection is flat when restricted to any moment torus $T$ and represents $\pi_1(T)$ trivially. Thus we have a symplectic trivialization of $\pi^{-1}(T)$. In lieu of the ability to change the connection on an open set (see the $G$-equivariant versions of Theorems \ref{connectionextensiontheorem} and \ref{uniquecouplingthm} in \cite[Thms 4.6.1, 4.6.2]{guillemin}), one can extend such a connection to all of $\text{Flag}(\bb{C}^3)$ and still have an isotopy to the original weak coupling form with small fibers.

In the connection given by Theorem \ref{moment} above, the weak coupling form looks like
$$\omega_F-d\delta$$
by a short argument as in \cite[Thm 1.6.1]{guillemin} since the fibers are simply connected. In particular $\delta:\bb{P}^1_p\rightarrow \mf{t}^\vee$ can be viewed as a moment map for the Hamiltonian $T^2$ action on the fiber. However, by construction the $T^2$ action is a subaction of a linear action. Thus it extends to a $PU(2)$ action and therefore must degenerate to an $S^1$ action on $\bb{P}^1$. We will take our candidate Lagrangian
$$L=\delta^{-1}(0)\times L_B.$$
Since $\delta$ has zero average, we can assume that $0$ is the barycenter of the moment polytope so the fiber is a Clifford torus.

\subsubsection{Floer cohomology}\label{flagexamplefloercohomology}
The Lagrangian $L$ is topologically trivial, so we can pick the three-way sum of $S^1$ height functions $f:=f_1+f_2+f_3$ as our Morse-Smale function and then perturb slightly as needed. Further, choose a pseudo-gradient $X_f$ that models the flow of $m$ and splits as a direct sum $$X_{f_3}\oplus X_{f_1+f_2}\in TL_F\oplus H_L.$$

We use the 2nd order potential from Theorem \ref{potthm} to compute the Floer cohomology of $L$. By the transversality in \cite[Thm 6.1]{choohdiskinstantons} and Theorem~\ref{liftconfigthm}, it suffices to use an almost complex structure on the total space that fiberwise agrees with the toric structure and is also a lift of the toric structure from the base to compute the vertically constant and horizontally constant holomorphic disks. By the classification of holomorphic disks in the standard structure on a toric fiber in $\bb{P}^2$, the lifted potential \eqref{liftpot} is
\begin{align*}\mathcal{L}\circ \mu^0_{L_B}(\theta)=\bigg(&\hol_\theta(\mc{L}u_1)q^{\eta_1} r^{e_v(\mathcal{L}u_1)}\\
& +\hol_\theta(\mc{L}u_2)q^{\eta_2} r^{e_v(\mathcal{L}u_2)} +\hol_\theta(\mc{L}u_3)q^{\eta_3} r^{e_v(\mathcal{L}u_3)}\bigg)x_M^M
\end{align*}
where $e_v(\mc{L}u_i)=-\int_{\partial D} \mc{L}u_i^*\delta$ by Corollary \ref{zeroenergycorollary} and $\eta_i$ as in \eqref{fooocliffpotential}. Let $y_i$ denote the evaluation of the holonomy $\hol$ on the $i^{th}$ generator, so that $\hol(u_i)=y_i$ for $i=1,2$.  We have $\hol(u_3)=-\dfrac{1}{y_1y_2}$. The lifts may have non-trivial vertical holonomy, so that
\begin{gather}
\label{holonomyrepflag3}\hol(\mc{L}u_1)=y_1y_3^k\\
\hol(\mc{L}u_2)=y_2y_3^l\\
\hol(\mc{L}u_3)=\dfrac{1}{y_1y_2y_3^j}
\end{gather}
for some integers $j,k,l$. The inclusion of the potential for $\mathrm{Cliff}(\mathbb{P}^1)$ in the fiber above $x_M$ is
\begin{equation}
i_{x_B*}\circ \mu^0_{\mathrm{Cliff}(\mathbb{P}^1)}(y_3)= \left(y_3 r^{e_v(v_1)}+y_3^{-1}r^{e_v(v_2)}\right)x_M
\end{equation}
similarly to \eqref{fooocliffpotential}, where $v_1(z)=[z,1]$, $v_2(z)=[1,z]$, and hence $e_v(v_1)=e_v(v_2)$. By Theorem \ref{potthm} the second order potential for the total Lagrangian is a subsum of
\begin{align}
\widetilde{\mathcal{W}}^2_L(y_1,y_2,y_3)&= y_3 r^{e_v(v_1)}+y_3^{-1}r^{e_v(v_2)}\\
& + y_1y_3^{k}q^{\eta_1} r^{e_v(\mathcal{L}u_1)} +y_2y_3^lq^{\eta_2} r^{e_v(\mathcal{L}u_2)} +\frac{1}{y_1y_2y_3^{j}}q^{\eta_3} r^{e_v(\mathcal{L}u_3)}\label{prepotflag3}
\end{align}
where the $\sim$ means that we have yet to ignore the terms with $\deg_q\neq 0$ of higher total degree. We change the basis of $\T{Hom}(\pi_1(L),\Lambda^{2\times})$ so that $y_1y_3^k\mapsto \tilde{y}_1$ and $y_2y_3\mapsto \tilde{y}_2$ and rename the new variables $y_1$,$y_2$:
\begin{align}\label{potflag3}
\widetilde{\mathcal{W}}_L^2(y_1,y_2,y_3)&= y_3 r^{e_v(v_1)}+y_3^{-1}r^{e_v(v_2)} \\
 \label{potflag32} &+ y_1q^{\eta_1} r^{e_v(\mathcal{L}u_1)} +y_2q^{\eta_2} r^{e_v(\mathcal{L}u_2)}  +\frac{1}{y_1y_2y_3^{m}}q^{\eta_3} r^{e_v(\mathcal{L}u_3)}
\end{align}
for some $m\in\Z$. Since the first part of this expression is symmetric in $y_3$, we can assume that $m\geq 0$ by a further change of basis $y_3^{-1}\mapsto \tilde{y}_3$.

It follows from the construction of $L$ and Corollary \ref{zeroenergycorollary} that the exponent of any $r$ variable in line \eqref{potflag32} is zero. Thus, all terms with $q^{\eta_i} r^{e_v(\mc{L}u_i)}=q^{\eta_i}$ are of the same degree and $\widetilde{\mc{W}}^2_L=\mc{W}^2_L$.

We now specialize to the case that $L_B=\text{Cliff}(\p^2)$, so that $\eta_i=\eta_j=:\eta$. In what follows, we take the derivative of \eqref{potflag3} and look for a critical point in order to apply Theorem \ref{crit}, which is a standard procedure. If the reader is satisfied with assuming the application of Theorem \ref{crit} is valid, then one can skip to the main result of the Subsection \eqref{9.1result}.

The partial derivatives of $\mathcal{W}^2_{L}$ are
\begin{align}
\label{y1}&\partial_{y_1}\mathcal{W}^2_{L}= q^\eta  - \frac{1}{y_1^2y_2y_3^{m}}q^\eta \\
\label{y2}& \partial_{y_2}\mathcal{W}^2_{L}= q^\eta -\frac{1}{y_1y_2^2y_3^{m}}q^\eta \\
\label{y3}& \partial_{y_3}\mathcal{W}^2_{L}=r^{e_v(v_1)}-y_3^{-2}r^{e_v(v_2)}-\dfrac{m}{y_1y_2y_3^{m+1}}q^\eta .
\end{align}

Setting expressions \eqref{y1} and \eqref{y2} equal to $0$ gives partial solutions $y_1=y_2=y^{-m/3}_3$. It remains to solve \eqref{y3}=0. Making the substitutions for $y_1$, $y_2$ and setting equal to zero gives us
\begin{equation*}
r^{e_v(v_1)}-y_3^{-2}r^{e_v(v_2)}-\dfrac{m}{y_3^{5m/3+1}}q^\eta=0.
\end{equation*}
We normalize via the transformation $y_3\mapsto y_3^3$
\begin{equation*}
r^{e_v(v_1)}-y_3^{-2}r^{e_v(v_2)}-\dfrac{m}{y_3^{5m+3}}q^\eta =0
\end{equation*}
and clear the denominator
\begin{equation*}
y_3^{5m+3}r^{e_v(v_1)}-y_3^{5m+1}r^{e_v(v_2)}-m q^\eta =0.
\end{equation*}
Dividing by the appropriate power of $r$ gives
\begin{equation}\label{solvethis}
y^{5m+3}-y_3^{5m+1}-m q^\eta r^{\chi}=0.
\end{equation}
While the power of $r$ is most likely negative, we can pick $K$ large enough in the weak coupling form so that $q^\eta r^{\chi}$ is in the ring $\Lambda^2$.

The ring $\Lambda_t$ becomes a $\Lambda^2$-algebra via the homomorphism $q,r\mapsto t$, so let us instead solve the equation
\begin{equation}\label{solvethist}
y_3^{5m+3}-y_3^{5m+1}-m t^\alpha=0
\end{equation}
with $\alpha>0$. 

The reduction of \eqref{solvethist} mod $t$ has $1$ and $-1$ as solutions. To see that we can find unital solutions in $\Lambda_t$ we apply Hensel's lemma:

\begin{theorem}\cite[Thm 7.3]{eisenbud}
Let $R$ be a ring which is complete with respect to an ideal $\mf{m}$ (in the sense that $R=\underset{\leftarrow}{\lim} R/\mf{m}^i$) and $f(x)\in R[x]$ a polynomial. If $a$ is an approximate root of $f$ in the sense that 
$$f(a)\equiv 0\,(\text{mod } f'(a)^2\mf{m}),$$
then there is a root $b$ of $f$ such that
$$f(b)=0 \text{ and } b\equiv a\,(\text{mod }f'(a)^2\mf{m}).$$
If $f'(a)$ is a nonzero-divisor in $R$, then $b$ is unique. 
\end{theorem}
Completeness of $\Lambda_t$ with respect to $(t)$ is obvious from the construction of the completion as a subring of $\Pi_i \Lambda_t/(t^i)$. Thus, there are solutions $\mathfrak{y_1}\equiv 1\, (mod\, t)$ and $\mathfrak{y_2}\equiv -1\, (mod\, t)$ to \eqref{solvethist} in $\Lambda_t$.

By Proposition \ref{crit}, there is a representation $\xi\in \T{Hom}(\pi_1(L),\Lambda_t^{\times})$ such that
\begin{equation}
\label{9.1result}H(CF(L,\Lambda_t),\nu^1_{\xi,\mc{W}x^\vee})\cong H^*(L,\Lambda_t).
\end{equation}
According to the discussion in Section \ref{hamiltonianinvariancesection}, this shows that the Lagrangian is non-displaceable via Hamiltonian isotopy and recovers the result from \cite{nish}.

\begin{remark}\label{hirzebruchremark}
At first glance, it would seem that the case of a toric Lagrangian $L\subset \bb{P}(\mc{O}\oplus \mc{O}_n)\rightarrow \bb{P}^2$ which fibers over $\cl{\p^2}$ is intimately similar from the point of view of our theory, and this is due to the fact that the connection we used above is the locally the same as the toric one. Surely the higher degree terms differ between $\text{Flag}(\bb{C}^3)$ and $\bb{P}(\mc{O}\oplus \mc{O}_n)$, in a manner that is similar to the difference between the full potentials for $\bb{F}_1$ and $\bb{F}_2$ as in \cite{fooohirzebruchexample, aurouxspeciallagfibrations}.

\end{remark}
\subsection{Families of non-displaceable tori}\label{familiesexample}
In this section we prove Theorem \ref{familiestheorem}: We construct n dimensional families of Floer-non-trivial tori in 
$$\bb{P}^k\rightarrow \bb{P}(\mc{O}\oplus\cdots \mc{O}_{i_k})\rightarrow \bb{P}^n$$
where $\mc{O}_i$ $i\geq 0$ is some rank $1$ holomorphic line bundle with Chern class $i$ times the hyperplane class. The family projects to a toric neighbourhood of $\text{Cliff}(\bb{P}^n)$, which extends (in a certain sense) results in \cite{fooohirzebruchexample} and \cite{viannafamilies}. The idea is as follows: In a toric neighbourhood of $\text{Cliff}(\bb{P}^n)$ we introduce a Hamiltonian perturbation of the toric form given by Theorem~\ref{moment}:
$$\omega_F-d\delta+K\pi^*\omega_{\bb{P}^n}\mapsto \omega_F-d\delta-dh+K\pi^*\omega_{\bb{P}^n}.$$
Here, $h$ is chosen to change the vertical symplectic area of the lifted disks so that all of the terms in line \eqref{potflag32} have the same total degree above a more general toric fiber. Then we show the existence of a unital critical point of the second order potential and apply Theorem \ref{crit}.

First, we see how to put a symplectic connection and weak coupling form on $E:=\bb{P}(\mc{O}\oplus\cdots \mc{O}_{i_k})$. Choosing a Hermitian metric on $\mc{O}_j$ induces a Chern connection that is compatible with the complex structure. This induces a $U(1)\times\cdots U(1)$ connection $H$ on $\mc{O}\oplus\cdots \mc{O}_{i_k}$ and also on $E$. By Theorem \ref{coupling} we get minimal coupling form $a_H$ and a $K$ so that $a_H+K\pi^*\omega_{\bb{P}^n}$ is a symplectic form on $E$ and induces the connection $H$.

\subsubsection{Recall of potential for moment torus in $\bb{P}^n$} Let 
$$\psi([z_0,\dots, z_n])=\dfrac{1}{\sum_{i}\vert z_i\vert^2}(\vert z_0\vert^2,\dots, \vert z_n \vert^2)$$
denote the moment map for the $T^n$ action on $\bb{P}^n$. The moment map for the same action on $(\bb{P}^n,K\omega_{\bb{P}^n})$ is $K\psi$, with the Clifford torus sitting in $(K\psi)^{-1}(\frac{K}{n+1},\dots,\frac{K}{n+1})$. As in the above example, recall that in a toric symplectic manifold we have a classification of $J$-holomorphic disks with boundary in a moment fiber, where $J$ is the toric complex structure \cite[5.3]{choohdiskinstantons}. In $\bb{P}^n$, this looks like 
$$z\mapsto [\phi_0 a_0,\dots,\phi_n a_n]$$
where each of the $\phi_i$ are Blaschke products
$$\phi_i(z)=\Pi_{j=1}^k \dfrac{z-p_{i,j}}{1-\bar{p}_{i,j}z}$$
with $p_{i,j}\in D$ that are pairwise distinct for different $i$ and $[a_0,\dots, a_n]$ is in the moment fiber. As in example \ref{flag3}, the Maslov index $2$ disks up to parameterization are thus given by
\begin{align*}
v_1(z)&=[z a_0,\dots, a_n]\\
v_2(z)&=[a_0,z a_1,\dots, a_n]\\
&\cdots \\
v_{n+1}(z)&=[a_0,\dots, z a_n].
\end{align*}
Following \cite[Ex 5.2]{fooo3} the symplectic areas of these disks are given by the moment map coordinates of the Lagrangian
\begin{align}\label{basesymplecticarea}
\int_D v_1^*K\omega_{\bb{P}^2}&= K\frac{\vert a_0\vert^2}{\sum_{i}\vert a_i\vert^2}=:\alpha_1\\
\int_D v_2^*K\omega_{\bb{P}^2}&= K\frac{\vert a_1\vert^2}{\sum_{i}\vert a_i\vert^2}=:\alpha_2\\
&\cdots \\
\int_D v_{n+1}^*K\omega_{\bb{P}^2}&=K\left(1-\sum_j\frac{\vert a_j\vert^2}{\sum_{i}\vert a_i\vert^2}\right)=1-\sum_{i=1}^{n}\alpha_i.
\end{align}
By the work of \cite{choohdiskinstantons} and others, the potential for $T_{\alpha_1\dots \alpha_n}$ in the integrable toric complex structure with the standard spin structure is
\begin{equation}\label{cliffnpotential}
\mu^0_{\text{Cliff}(\mathbb{P}^n)}=\sum_{i=1}^n y_iq^{e(v_i)}+\frac{1}{y_1\cdots y_n}q^{e(v_{n+1})}.
\end{equation}

\subsubsection{Constructing the family} By the Arnold-Louville theorem (see for example \cite[Thm 1.2]{BolsinovFomenko}) we can take a $T^n$-equivariant symplectic chart of $U\cong  T^n\times B_r (0)$ which maps the torus $T_{\alpha_1\cdots \alpha_{n}}$ at $(\alpha_1,\dots, \alpha_{n})$ to the torus
$$\mc{P}_{\alpha_1,\dots, \alpha_{n}}:=T^n\times \left\lbrace \left(\alpha_1-\dfrac{K}{n+1},\dots, \alpha_{n}-\dfrac{K}{n+1}\right)\right\rbrace$$
in an equivariant way and such that the symplectic form on $U$ is the pullback of $\omega_0=\sum_i  d\alpha_i\wedge d\theta_i$. To trivialize $E$ over the neighbourhood $U$, we use some lift of the $T^n$ action over this neighbourhood. $\pi\vert_U$ is $T^n$-equivariant and we can apply Theorem \ref{moment} to obtain a toric connection over $U$. Next, take a generic trivialization $\Psi$ of this new symplectic form over a simply connected neighbourhood that intersects $\cl{n}$. By $T^k$ invariance of the connection and triviality of the connection around each torus factor, we can extend $\Psi$ to a toric neighbourhood of $\cl{n}$.

As in the above example, we can assume that parallel transport around any base torus $T^n$ is a linear torus action and thus reduces to a $T^k$ action. Since $\delta$ is the moment map for this $T^k$ action with zero average, we can assume that $\delta^{-1}(0)$ is $\cl{k}$ with respect to this action. Define the family of Lagrangians
$$L_{\alpha_1\dots \alpha_{n}}=\mc{P}_{\alpha_1\dots \alpha_{n}}\times \cl{k}$$ in the above trivialization. 

\subsubsection{Modifying the Hamiltonian connection} Let $h:\bb{P}^k\rightarrow \bb{R}$ be a zero average function that is preserved by the $T^k$ action and the level set of $\max \vert h \vert$ contains $\cl{k}$ with $h\vert_{\cl{k}}=\frac{1}{2\pi}$.

We define the function valued $1$-form
\begin{equation}
\bold{h}_{\alpha_1,\dots, \alpha_{n}}(\theta_1,\dots,\theta_n):=\sum_i \left(\alpha_{i}-\frac{K}{n+1}\right)h d\theta_i
\end{equation}
on a sub neighbourhood $U_0\subset U$ containing $\cl{n}$ and extend it to all of $\bb{P}^n$ via a cutoff function $\phi$ which depends on the $\alpha_i$, is $1$ on $U_0$, and vanishes outside $U$.
First we check:
\begin{lemma}\label{isasympformlemma}
$\omega_h:=a_H-d(\phi \bold{h})+K\pi^*\omega_{\bb{P}^n}$ is a symplectic form on $\pi^{-1}(U)$ for small enough $U$, where $K\geq1$ was chosen so that $a+K\pi^*\omega_{\bb{P}^n}$ is a symplectic form on $\bb{P}(\mc{O}\oplus\cdots \mc{O}_{i_k})$.
\end{lemma}
\begin{proof}
In the symplectic trivialization discussed above, we can assume that the weak coupling form looks like $\omega_{\bb{P}^k}+K\sum_i d\theta_i\wedge d\alpha_i$ and gives the trivial connection $T\bb{P}^k\oplus TU$. Inside $U$ we have 
\begin{equation}\label{dh}
d(\phi \bold{h})= \phi\left[\sum_i h d\alpha_i\wedge d\theta_i+dh\sum_i\left(\alpha_i-\frac{K}{3}\right) d\theta_i \right]+d\phi\wedge \bold{h}.
\end{equation} 
For example, set $k=1$ and $n=2$ and take the coordinates $\lbrace \partial_\rho,\partial_{\theta_1},\partial_{\theta_2},\linebreak \partial_\theta,\partial_{\alpha_1},\partial_{\alpha_2}\rbrace$ on $\pi^{-1}(U)$ where $\partial_\rho,\partial_\theta$ are action-angle coordinates on the sphere. The matrix of $\omega_h$ is the negative of
\begin{equation}
\resizebox{.91\textwidth}{!}{$
\begin{bmatrix}
0 & (\alpha_1-\frac{K}{3})\partial_\rho h  & (\alpha_2-\frac{K}{3})\partial_\rho h  & * & * & *\\
(\frac{K}{3}-\alpha_1)\partial_\rho h  & 0 & 0 & * & * & *\\
(\frac{K}{3}-\alpha_2)\partial_\rho h  & 0 & 0 & * & * & *\\
1 & 0 & 0 & 0 & 0 & 0\\
0 & K+\phi h-h(\alpha_1-\frac{K}{3})\partial_{\alpha_1}\phi  & h(\alpha_2-\frac{K}{3})\partial_{\alpha_1}\phi  & 0 & 0 & 0\\
0 & h(\alpha_1-\frac{K}{3})\partial_{\alpha_2}\phi  & K+\phi h+h(\alpha_2-\frac{K}{3})\partial_{\alpha_2}\phi  & 0 & 0 & 0
\end{bmatrix}$}
\end{equation}
where the $*$'s represent the skew-symmetric completion. In the general case, the matrix has a similar shape, with only the anti-diagonal blocks contributing to the determinant. The lower left $(n+k)\times (n+k)$ block has $k$ $1$'s along the diagonal followed by the matrix 
\begin{equation}\label{lowerblockmatrix}
A:=\text{diag}(-K-\phi h)-h\left(\boldsymbol{\alpha}-\frac{\bold{K}}{3}\right)^t\cdot \grad \phi
\end{equation}
where $\boldsymbol{\alpha}-\frac{\bold{K}}{3}$ is a row vector whose entries are $\alpha_i-\frac{K}{3}$. The determinant of the matrix of $\omega_h$ is equal to $\det^2 A$, and we have the general formula
$$\det(D+\bold{u}^t \bold{v})=(1+(D^{-1}\bold{u})\cdot\bold{v}^t)\det D.$$
Thus, we have 
$$\det A=(K+\phi h)^n\left(1+(K+\phi h)^{-1}h\sum_i \left(\alpha_i-\frac{K}{3}\right)\partial_i\phi\right).$$

By the assumption that $\vert h \vert<\frac{1}{2\pi}$ we have $\vert K+\phi h\vert >1-\frac{1}{2\pi}$. To bound the second product factor we can assume that $\phi$ is constructed as follows: Let $\phi'$ be a radial function on $U$ that is $1$ on $B_{r_1}(\frac{\bold{K}}{3})\subset U_0$ and $0$ outside $B_{r_2}(\frac{\bold{K}}{3})\subset U$ and further assume it to be piecewise linear in the radius with slope $\frac{1}{r_1-r_2}$ in the annulus. Let $\phi$ be a smoothing of this function such that $\vert \grad \phi\vert\leq \frac{1}{r_1-r_2}$ with non-zero derivative only on the annulus $B_{r_2+\epsilon}\setminus B_{r_1-\epsilon}$. By Cauchy-Schwarz and the above estimate we have
\begin{align}\label{symplecticformbound}
\vert (K+\phi h)^{-1}h\sum_i (\alpha_i-\frac{K}{3})\partial_i\phi \vert &\leq \frac{1}{2\pi(1-\frac{1}{2\pi})}\Vert \grad \phi \Vert (r_2+\epsilon)\\
\notag &\leq \frac{r_2+\epsilon}{(2\pi-1)(r_2-r_1)}.
\end{align}
If we let $r_1$, $\epsilon$ be small enough, then  we have $\frac{r_2+\epsilon}{r_2-r_1}<2\pi-1$. This bound shows that the determinant is non-zero for small $U_0$ and $\epsilon$. Hence, $\omega_h$ is non-degenerate. It is clearly closed, so the lemma follows.
\end{proof}

Also, we need that
\begin{lemma}
The product $L_{\alpha_1\dots \alpha_{n}}= \cl{k}\times\mc{P}_{\alpha_1\dots \alpha_{n}}$ is Lagrangian with respect to $a_H-d(\phi\bold{h})+K\pi^*\omega_{\bb{P}^n}$ whenever $(\alpha_1,\dots, \alpha_n)\in U$.
\end{lemma}

\begin{proof}
Each fiber of $L_{\alpha_1\dots \alpha_{n}}$ is contained in a level set of $h$, so the second to last term of \eqref{dh} vanishes on $L_{\alpha_1\dots \alpha_{n}}$. Since $\phi$ only depends on the $\alpha_i$, the last term vanishes as well. The first term is proportional to $\omega_{\bb{P}^n}$ and $L_{\alpha_1\dots \alpha_{n}}$ is Lagrangian for the product form, so the lemma follows.
\end{proof}

Denote the connection over $U$ constructed above by $H_h$. We would like a connection on the total space which extends $H_h$. In general one can apply the $T^n$ equivariant version of Theorem \ref{connectionextensiontheorem}, but we construct an explicit connection that agrees with the original Hermitian connection outside of a neighbourhood of $\cl{n}$:

\begin{lemma}[Extension of connection]\label{extensionofwierdconnectionlemma}
There is a Hamiltonian connection $\tilde{H}$ on $\bb{P}(V)$ that agrees with $H_h$ on $\pi^{-1}(U)$ and agrees with the original connection outside a neighbourhood of $\pi^{-1}(U)$ such that $a_{\tilde{H}}+K\pi^*\omega_{\bb{P}^n}$ is symplectic for large enough $K$.
\end{lemma}

\begin{proof}
Extend the symplectic trivialization discussed in the construction of $H_h$ to a larger collar neighbourhood $V$ with $U\subset V$. Denote by $\gamma$ be the zero-average connection $1$-form corresponding to the original connection on $\bb{P}(V)$ from the original construction as a $U(k+1)$ bundle and $\tau$ the connection $1$-form over $V\setminus U$ from the toric trivialization. Let $\rho_1, \rho_2$ be a partition of unity subordinate to the cover $\bb{P}^n\setminus U, V$. Form the connection $\tilde{H}$ over $\bb{P}^n\setminus U$ induced by $\rho_1 \gamma+\rho_2\tau$. Then $\tilde{H}$ is the trivial toric connection near $U$ and so it extends the connection $H_h$. Take $K$ so large that $a_{\tilde{H}}+K\pi^*\omega_{\bb{P}^n}$ is a symplectic form on $\bb{P}(\mc{O}\oplus\cdots \mc{O}_{i_k})$.
\end{proof}
\begin{remark}
The value of $K$ may need to be larger after extending the toric trivialization in Lemma \ref{extensionofwierdconnectionlemma}. However, one can then obtain the result of Lemma \ref{isasympformlemma} with this particular value of $K$.
\end{remark}

\subsubsection{Classification of lifts with $0$ vertical Maslov index}
In order to mimic Theorem \ref{liftconfigthm} and ultimately Theorem \ref{potthm} for $L_{\alpha_1\dots \alpha_n}$, we need a way to holomorphically trivialize $v_i^*(E,L)$ as in Theorem \ref{oka}. Let us work with $v_1$ WLOG. The theorem doesn't directly apply since the connection on $v_1^*(E,L)$ is not $T^k$-valued and does not preserve the fiber holomorphic structure within $\pi^{-1}(U)$, so we modify the argument:

We work with $v_1$ for notational simplicity and WLOG. Suppose that the cutoff function $\phi$ has support in some region $\tilde{V}\subset U$. By the classification of disks in $\bb{P}^n$ as in the beginning of this example, we can assume that $v_1$ is the restriction of a disk map $u_1$ where $u_1$ has boundary in the Lagrangian $\mc{P}_{\tilde{\alpha_1}\dots \alpha_{n}}$, $\tilde{\alpha_1}\geq \alpha_1$ with $(\alpha_1',\dots, \alpha_n)\in U\setminus \tilde{V}$ lying in the region where the connection is trivial.

\begin{figure}
\centering
\begin{minipage}{.5\textwidth}
  \centering

\def\svgwidth{\columnwidth}
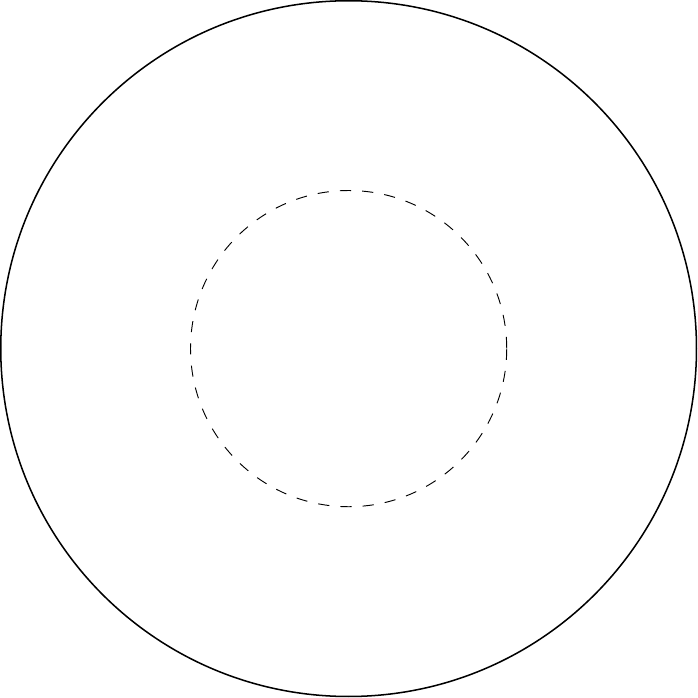
  \caption{Domain for $u_1$.}
  \label{fig:familydiskdomain}
\end{minipage}%
\begin{minipage}{.5\textwidth}
  \centering

\def\svgwidth{\columnwidth}
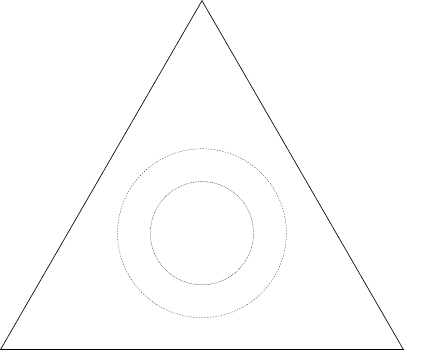
  \caption{Domain for $u_1$.}
  \label{fig:toricpricturefamily}
\end{minipage}
\end{figure}

\begin{lemma}[Trivialization]\label{okafamily}
There exists a biholomorphism $\Phi: (D\times \bb{P}^k,J_0)\rightarrow (u_1^*E,J_h)$ where $J_0$ is induced from the trivial connection and fiberwise agrees with $J_I$ outside an annular region. Moreover, $\Phi (L_{\alpha_1\dots \alpha_n})=L_{\alpha_1'\dots \alpha_n}$ for some $\alpha_1'$.
\end{lemma}
\begin{proof}
As above, let $H$ denote the connection on $u_1^*E$ induced from the original construction of $E$ and the toric trivialization (Theorem \ref{moment}) in the neighbourhood $U$. By Theorem \ref{oka} there is a $(\bb{C}^\times)^{k}$ valued gauge transformation $\mc{G}_\bb{C}$ so that $\mc{G}_\bb{C}(H)$ is a trivial connection. 
A key property of $\mc{G}_\bb{C}$ is that it is $T^k$-valued in the region $u_1^{-1}(U)$: Indeed, the proof of Donaldson's Theorem \ref{heat} involves looking at the evolution equation
\begin{equation}\label{hymtemp}
g^{-1}\frac{\partial g}{\partial t}=-2i\Lambda F_g
\end{equation}
on the bundle of metrics for the associated bundle $\mc{V}$ with the appropriate boundary condition. However, the initial metric $g_H$ induced by $H$ in the region $u_1^{-1}(U)$ is already flat, so the flow is independent of $t$ here. Thus, the solution to \eqref{hymtemp} differs from $g_H$ in this region by a unitary transformation on each line bundle factor, so the claim follows.

Since $h$ is invariant under the $T^k$ action we have $\mc{G}_\bb{C}^*\phi\bold{h}= \phi\bold{h}$ as a function valued $1$-form, and the connection $\mc{G}_{\bb{C}}(\tilde{H})$ is induced by the weak coupling form $\omega_F-d\phi\bold{h}+Kd\alpha_1\wedge d\theta_1$ in the region $U$. Expanding, we have
$$u_1^*\omega_h=\omega_{\bb{P}^k}-\phi hd\alpha_1\wedge d\theta_1- d\phi\wedge \bold{h}-\phi dh\wedge \frac{\bold{h}}{h} +K d\alpha_1\wedge d\theta_1$$
where polar coordinates on the disk are given by $(\alpha_1,\theta_1)$. To save notation, also denote the above form by $\omega_h$. If we follow the reasoning of the proof of Lemma \ref{isasympformlemma}, the two-form $(1-t)\omega_0+t\omega_h=\omega_{th}$ is symplectic for $t\in [0,1]$. We have $\omega_{h}-\omega_0=-d(\phi \bold{h})$, so we can apply a Moser type argument to obtain an isotopy $\Phi_t$ such that $\Phi_t^*\omega_{th}=\omega_0$. The isotopy is the flow of a time dependent vector field $\bold{v}_t$ given by $\phi \bold{h}=\iota_{\bold{v}_t}\omega_{th}$:
$$\bold{v}_t=\frac{(\alpha_1-\frac{K}{n+1})\phi h}{K-t\phi h-th(\alpha_1-\frac{K}{n+1})\partial_{\alpha_1}\phi}\partial_{\alpha_1}$$
on $u_1^*E$. Notably, this vanishes outside $u_1^{-1}(U)$ and exchanges $u^*L_{\alpha_1\dots \alpha_n}$ with some other equatorial section over a concentric circle. We take the map $\Phi:=\mc{G}_\bb{C}^{-1}\circ \Phi_1 $, and $\Phi^{-1}(u_1^*E)$ can still be realized as topological a product via the projection $\tilde{\pi}:=\Phi^*\pi$. The connection $\Phi^{-1}(\tilde{H})$ is flat and is prescribed by the pullback of the weak coupling form on $u_1^{-1}(U)$.  We define $J_0:=\Phi^*J_h$, which is the unique a.c.s. from Lemma \ref{connectionacs} for the connection $\Phi^{-1}(\tilde{H})$ that fiberwise agrees with the $J_I$ outside of $u^{-1}(U)$ and $\Phi^*J^G$ inside this region.
\end{proof}

Obviously the lemma remains true for the other disks $u_i$ and $v_i$.
\begin{corollary}[Vertically constant lifts, their homology classes, and transversality]\label{uniqueliftfamily}
There exists a vertically constant lift $\hat{v}_i$ of $v_i$ to $(E,L_{\alpha_1\dots \alpha_n})$ which is Fredholm regular, and after fixing a point $\hat{v}_i=p\in \pi^{-1}( v_i(1))\cap L_{\alpha_1\dots \alpha_n}$ the lift is unique in its homology class. Moreover, any non-vertically constant lift has positive vertical Maslov index.
\end{corollary}
\begin{proof}
After holomorphically trivializing $u^*_i\bb{P}(\mc{O}\oplus\cdots \mc{O}_{i_k})$ via $\Phi^{-1}$, we can take $\hat{v}_i$ as the restriction of the constant section $\hat{u}_i(z)=p$ to a smaller disk $D_0\subset D$ whose boundary lies at $\hat{u}_i^{-1}(L_{\alpha_1\dots\alpha_n})$. Other holomorphic lifts $\tilde{v}_i$ other holomorphic lifts are given by maps $w_1\times w_2$ where $[w_1]$ generates $H_2(D,\pi(L_{\alpha_1\dots\alpha_n}))$ and $w_2:D_0\rightarrow (\bb{P}^k,\cl{k})$ is a holomorphic disk with respect to some (domain dependent) tamed almost complex structure. Since it is constant in the $\bb{P}^k$ direction, it follows that $\hat{v}_i$ is the unique holomorphic representative of its homology class. Moreover it follows that $\mu(w_2)>0$ if $w_2$ is non-constant.

Transversality at the class $[\hat{v}_i]$ follows for every vertical almost complex structure exactly as in the proof of Theorem \ref{liftconfigthm}. 
\end{proof}

In the spirit of the rest of the paper, we denote $\hat{v}_i=:\mc{L}v_i$ as this Donaldson lift of $v_i$. The holonomy of the fibration around the boundary of each $\mc{L}v_i$ is prescribed precisely by $\phi\bold{h}$. By corollary \ref{zeroenergycorollary} and the definition of $h$, the vertical symplectic areas of the $\mc{L}v_i$ with respect to $\omega_h$ on $\bb{P}(\mc{O}\oplus\cdots\mc{O}_{i_k})$ are:
\begin{align}\label{verticaltwistedarea}
e_v(\mc{L}v_1)&=-\int_{\theta_1\in[0,2\pi]} \bold{h}=\frac{K}{n+1}-\alpha_1\\
e_v(\mc{L}v_2)&=-\int_{\theta_2\in [0,2\pi]}\bold{h}=\frac{K}{n+1}-\alpha_2\\
&\cdots\\
e_v(\mc{L}v_{n+1})&=-\int_{\Delta}\bold{h}=\sum_i\alpha_i-\frac{nK}{n+1}
\end{align}
where in the last integral $\Delta\subset T_{\alpha_1\dots \alpha_{n}} $ is along the diagonal $\lbrace\theta_i=\theta_j\rbrace$ and has a differing sign due to reversed orientation (since $v_n(z)=[\frac{a_0}{z},\frac{a_1}{z},\dots, a_n]$).

Similar to paragraph \ref{flagexamplefloercohomology} we can derive the second order leading potential for $L_{\alpha_1\dots\alpha_n}$. Let $y_i$ represent the holonomy of a representation $\rho\in\text{Hom} (\pi_1(L_{\alpha_1\dots \alpha_n}),\Lambda^{2,\times})$ around the $i^{th}$ torus factor of $\mc{P}(\alpha_1,\dots,\alpha_n)$, and $x_j$ is the holonomy around the $j^{th}$ torus factor of the fiber $\cl{k}$. We use the integrable toric complex structure in the base resp. fiber to write down the vertically constant disks resp. horizontally constant disks. Let $w_i$ denote the disks with boundary in $\cl{k}$ as in the classification in the beginning of this example. Lemma \ref{okafamily} provides a substitute to Theorem \ref{oka} for the holomorphic trivialization in the absence of a $G$-connection, and Corollary \ref{uniqueliftfamily} provides a replacement for Theorem \ref{liftconfigthm} so that the expression of the second order potential still has the form \eqref{peq}. Thus, we have
\begin{align}\label{potfamily}
\mathcal{W}^2_{L_{\alpha_1\dots\alpha_n}}[q,r]&= \sum_j x_j r^{e_v(w_1)}+\frac{1}{x_1\cdots x_k}r^{e_v(w_{k+1})}+\sum_i y_i r^{e_v(\mc{L}v_i)}q^{e(v_i)} \\
 \label{potfamily2} &\quad +\frac{1}{y_1\cdots y_n x_1^{\beta_1}\cdots x_k^{\beta_k}}r^{e_v(\mc{L}v_{n+1})}q^{e(v_{n+1})}.
\end{align}
To avoid clutter we absorb the result of lifting the $v_i$ on the holonomy representation into a change of basis of the form $y_ix_1^*\cdots x_k^*\mapsto y_i$ (which is reflected in the exponents $\beta_1,\dots, \beta_k$). We observe that $e_v(w_i)=e_v(w_j)$ since $L_F=\cl{k}$ and the form $d\bf{h}$ vanishes on the fibers (it is verti-zontal). Moreover, all of the terms involving both $q$ and $r$ have the same degree by the calculations \eqref{verticaltwistedarea} and \eqref{basesymplecticarea}. As in example \ref{flag3} beginning at line \eqref{potflag3}, we can tensor this expression with $t$:
\begin{align}\label{potfamilyt}
\mathcal{W}^2_{L_{\alpha_1\dots\alpha_n}}[t]&= \sum_j x_j t^{e_v(w_1)}+\frac{1}{x_1\cdots x_k}t^{e_v(w_{k+1})}+\sum_i y_i t^{e(\mc{L}v_i)}\\ 
\label{potfamilyt2} &\quad +\frac{1}{y_1\cdots y_n x_1^{\beta_1}\cdots x_k^{\beta_k}}t^{e(\mc{L}v_{n+1})}
\end{align}
where $e(\mc{L}v_i)=e(\mc{L}v_j)$ $\forall i,j$. Similar to the process in the Flag example starting at \eqref{y1}, we can take the partial derivatives of this expression with respect to the $x_i$ and $y_i$ variables, apply the multivariate Hensel's Lemma \ref{henselsmulti} to find a unital nondegenerate critical point. The calculation is completely analogous to one in the next example (Subsection~\ref{L_2 computation}). Finally, we apply Proposition \ref{crit} to see that there is a representation $\eta\in \text{Hom}(\pi_1(L_{\alpha_1\dots\alpha_n}),\Lambda_t)$ so that
$$H(CF(L_{\alpha_1\dots\alpha_n},\Lambda_t),\nu^1_{\eta,\mc{W}x^\vee})\cong H^*(L_{\alpha_1\dots\alpha_n},\Lambda_t).$$
This holds for every $(\alpha_1,\dots,\alpha_n)\in U_0$, so it follows that the entire family of Lagrangians $\pi^{-1}(U_0)$ is Floer non-trivial and thus Hamiltonian non-displaceable by the discussion in Section \ref{hamiltonianinvariancesection}.

\subsection{Initially Complete Flags}\label{fullflags}

The construction  of the Floer-non-trivial three torus in $\T{Flag}(\C^3)$ naturally gives a way to construct non-displaceable tori in higher dimensional flag manifolds. We carry this out here, with a warning that the indices in the potential become quite complicated.

Let $F_n^k$ be the manifold of initially complete flags $V_1\subset\cdots \subset V_k\subset \C^{n+1}$ where $\dim_\C V_i=i$. Identify this homogeneous $U(n)$ space with some coadjoint orbit $G\cdot \xi$ equipped with the $U(n)$-equivariant Kostant-Kirillov symplectic form $\omega_\xi(X^\sharp,Y^\sharp)=\langle \xi, [X,Y]\rangle$. The corresponding weak coupling form will be scaled at each step in the fibration $\p^{n-k}\rightarrow F_n^{k+1}\rightarrow F_n^k$. We construct a Floer-non-trivial Lagrangian in $F_n^k$ by induction on $k$.

\begin{theorem}
There is a Lagrangian torus $L_k\subset F_n^k$ that is unobstructed and Floer-non-trivial.
\end{theorem}

\subsubsection*{Constructing and computing Floer theory for $L_2\subset F_n^2$:}\label{L_2 computation} We illustrate the step from $k=1$ to $k=2$ to derive a model for the higher order cases. $F_n^1=\p^n$, so let us take $L_1=\text{Cliff}(\mathbb{P}^n)$. There is a fibration $\pi_2:F_n^2\rightarrow\p^n$ with fiber $\p^{n-1}$, so by Theorem \ref{moment}, there is a $T^{n}$-invariant open neighbourhood $\mathcal{U}_1$ of $\text{Cliff}(\mathbb{P}^n)$ and a symplectic connection on $F_n^2$ such that the moment map for the action of $T^{n}\subset PU(n+1)$ on $\pi_2^{-1}(\mathcal{U}_1)$ is $\phi\circ \pi_2$, where $\phi:\p^n\rightarrow \mathfrak{t}^{n\vee}$ is the moment map for $\p^n$.

At the level set of $\phi\circ \pi_2$ that lies above $\text{Cliff}(\mathbb{P}^n)$, we have a symplectic trivialization $\phi\circ \pi_2^{-1}(0)\cong\p^{n-1}\times \text{Cliff}(\mathbb{P}^n)$. Thus, pick the Lagrangian $L_2=\delta^{-1}(0)\times\text{Cliff}(\mathbb{P}^n)\subset F_n^2$ where $\delta$ the moment map for parallel transport around $\cl{n}$, as in the construction of $L$ in Example \ref{flag3}. Thus, the fiber is a Clifford torus.

Next, we calculate the second order potential for $L_2$ using Theorem \ref{potthm} and show that it has a critical point. Let $y_i$ be the evaluation of the local system $\rho$ on the first $n$ factors, and $x_j$ be the evaluation on the last $n-1$ factors. The potential for $\text{Cliff}(\mathbb{P}^n)$ is as in \eqref{cliffnpotential} where $e(u_i)=e(u_j)$ $\forall i,j$. The lifted potential for $L_2$ in $F_n^2$ is
\begin{equation}
\mathcal{L}\circ\mu^0_{\text{Cliff}(\mathbb{P}^n)}[q,r]=\sum_{i=1}^n y_ix^{\nu^i}q^{e(u_i)}+\frac{1}{y_1\cdots y_nx^{\nu}}q^{e(u_{n+1})}
\end{equation}
where $\nu^i$ and $\nu$ are multindices with $x^{\nu^i}:=x_1^{\nu_{1}^i}\cdots x_{n-1}^{\nu_{n-1}^i}$ given by the holonomy of the connection around the boundary of $\mathcal{L}u_i$. By changing the basis of $\pi_1(L)$, we can assume that $\nu^i=(0,\dots,0)$. Moreover, there are no $r$ terms since the integral of the holonomy $1$-form around the boundary of each disk is $0$ by Corollary \ref{zeroenergycorollary} and construction of $L_2$. By Theorem \ref{potthm} the second order potential of $L_2$ is
\begin{align*}
\mathcal{W}^2_{L_2}[q,r]&=\sum_{i=1}^n y_iq^{e(u_i)}+\frac{1}{y_1\cdots y_n x^{\nu}}q^{e(u_{n+1})}\\ &\quad+ \sum_{i=1}^{n-1} x_ir^{e_v(w_i)}+\frac{1}{x_1\cdots x_{n-1}}r^{e_v(w_{n})}.
\end{align*}

The search for a critical point yields the equations
\begin{equation}\label{yvar}
\partial_{y_i}\mathcal{W}^2_{L_2}= q^{e(u_i)}+\frac{-1}{y_i\cdot y_1\cdots y_n x^{\nu}}q^{e(u_{n+1})}=0
\end{equation}
and
\begin{align}\label{zvar}
\partial_{x_j}\mathcal{W}^2_{L_2}=&\frac{-\nu_j}{x_jy_1\cdots y_n x^{\nu}}q^{e(u_{n+1})} \\& \label{zvar2} +  r^{e_v(w_j)}+\frac{-1}{x_j\cdot x_1\cdots x_{n-1}}r^{e_v(w_{n})}=0.
\end{align}
Since the $u_i$ resp. $w_i$ are Maslov index $2$ disks that are bounded by $\text{Cliff}(\mathbb{P}^n)$ resp. $\text{Cliff}(\mathbb{P}^{n-1})$, we have that $$e(u_i)=e(u_j),\, \forall i,j$$
and
$$e_v(w_i)=e_v(w_j),\, \forall i,j.$$

With these assumptions, the system \eqref{yvar} gives
\begin{equation}\label{substitution}
g_i:=y^{\mathfrak{e}_i}x^\nu-1=0
\end{equation}
for every $i$, where $y^{\mathfrak{e}_i}=y_iy_1\cdots y_n$. Notably this forces $y_i=y_j$, which we now denote as $y$.

Next, we replace the formal variables $q,r$ with $t$ in \eqref{zvar} and \eqref{zvar2}:
$$\frac{-\nu_j}{x_jy^n x^{\nu}}t^\alpha +1+\frac{-1}{x_j\cdot x_1\cdots x_{n-1}}=0$$
for some $\alpha>0$. Let $\mathfrak{e}=(1,\dots,1)$ and $\mathfrak{e}_j=(1,\dots,\overset{j}{2},1,\dots,1)$, and multiply the above equation by $x^{\nu+\mf{e}_j}$:
\begin{equation}\label{basestepsystem}
f_j:=x^{\nu + \mathfrak{e}_j}-x^\nu-\frac{\nu_j x^{\mathfrak{e}}}{y^n}t^\alpha=0.
\end{equation}
We are now in position to apply the multivariate version of Hensel's lemma to solve the system $(f_1,\dots, f_{n-1},g_1,\dots, g_{n})=0$:

\begin{theorem}\cite[Exercise 7.26]{eisenbud}\label{henselsmulti}
Let $R$ be a ring which is complete with respect to an ideal $\mf{m}$ and $\bold{h}:=(h_1,\dots, h_n)\in R[x_1,\dots, x_n]$ a system of polynomials. If $\bold{a}:=(a_1,\dots, a_n)$ solves
$$\bold{h}(\bold{a})\equiv 0\, (\text{mod }\mf{m}^{\oplus n}),$$
and the Jacobian $\text{Jac}_{\bold{a}}\bold{h}$ is a unit in $R$, then there is a root $\bold{b}$ 
$$\bold{h}(\bold{b})=0 \text{ and } \bold{b}\equiv \bold{a}\,(\text{mod } \mf{m}^{\oplus n}).$$
\end{theorem}

 First, we notice that 
$$x_1=\cdots=x_{n-1}=\aleph_k:=e^{\frac{k2\pi i}{n}}$$
and
$$y_1=\cdots y_n=\beth_{l,k}:=e^{\frac{l2\pi i}{n+1}-\frac{k\vert\nu\vert 2\pi i}{n(n+1)}}$$ gives a solution to \eqref{substitution}, \eqref{basestepsystem} mod $t^\alpha$ for $l,k\in \bb{Z}$. In order to apply Hensel's lemma, we need to show that the Jacobian of the system at this solution is a unit. Let $\mathfrak{f}_i=(0,\dots,\overset{i}{1},0,\dots, 0)$.
\begin{align*}
\frac{\partial f_j}{\partial x_i}&=(\nu_i+1+\delta_{ij})x^{\nu+\mathfrak{e}_j-\mathfrak{f}_i}-\nu_i x^{\nu-\mathfrak{f}_i}-\frac{\nu_j x^{\mathfrak{e}-\mf{f}_i}}{y^n}t^\alpha\\
\frac{\partial f_j}{\partial y_i}&=\frac{n\nu_j x^{\mathfrak{e}}}{y^{n+1}}t^\alpha \\
\frac{\partial g_j}{\partial x_i}&= \nu_i y^{\mathfrak{e}_j}x^{\nu-\mathfrak{f}_i} \; \text{if $\nu_i\neq 0$ or $0$ o/w}\\
\frac{\partial g_j}{\partial y_i}&= (1+\delta_{ij}) y^{\mathfrak{e}_j-\mathfrak{f}_i} x^{\nu}
\end{align*}

For the mapping $\Upsilon(x_1,\dots,x_{n-1},y_1,\dots,y_n)=(f_1,\dots, f_{n-1},g_1,\dots,g_n)$
the matrix $$D_{(\overset{1}{\aleph_k},\dots,\overset{n-1}{\aleph_k},\overset{n}{\beth_{l,k}},\dots,\overset{2n-1}{\beth_{l,k}})}\Upsilon$$ has entries that are listed above. This is invertible for $k=0$ (among other values). Therefore, by Hensel's lemma there is a solution 
$$(\tilde{\aleph}_1,\dots, \tilde{\aleph}_{n-1},\tilde{\beth}_1,\dots,\tilde{\beth}_n)$$
 to \eqref{substitution}, \eqref{basestepsystem} with 
\begin{align*}
\tilde{\aleph}_i &\equiv \aleph_k\; \T{mod} \,t^\alpha\\
\tilde{\beth}_{i} &\equiv \beth_{l,k} \; \T{mod} \,t^\alpha
\end{align*}
for each $i$ and fixed $l,k$. These solutions solve the system \eqref{yvar}, \eqref{zvar}.

By Theorem \ref{crit}, there is a local system $\xi$ such that $$H(CF(L_2,\Lambda_t),\nu^1_{\xi,\mc{W}x^\vee})\cong H^{*}(L_2,\Lambda(t))\neq 0.$$

\subsubsection*{Constructing $L_{k+1}\subset F_n^{k+1}$:} 
Assume we have constructed a Floer-non-trivial Lagrangian $L_l\subset F_n^l$ for each $l\leq k<n$ that fibers over $\text{Cliff}(\mathbb{P}^n)$. The two parts of the induction step are the construction of $L_{k+1}$ and the computation of the disk potential.

For the construction, we first view $F_n^k$ as a fibration $$F_{n-1}^{k-1}\rightarrow F_n^k\xrightarrow{\pi_k} \p^n.$$
There is a Hamiltonian $T^{n}$ action on $F_n^k$ given by 
$$V_1\subset\cdots \subset V_k\subset \C^{n+1}\mapsto t\cdot V_1\subset\cdots \subset t\cdot V_k\subset \C^{n+1}.$$
By Theorem \ref{moment} there is a $T^n$ equivariant subset $V$ with $\text{Cliff}(\mathbb{P}^n)\subset V\subset \p^n$ and a symplectic connection $\g'$ so that the moment map for this action on $a_{\g'}+\pi_k^*\omega_V$ is $\psi\circ \pi_k$, where $\psi$ is the moment map for the standard $T^n$ action on $\p^n$.

Now, we apply the same result \ref{moment} to the fibration $$\p^{n-k}\rightarrow F_n^{k+1}\xrightarrow{\pi} F_n^k$$ where $F^k_n$ is equipped with the adapted symplectic form as above: There is an open set $\pi_k^{-1}(\cl{n})\subset U\subset F_n^k$ and a symplectic connection $\g$ on $\pi^{-1}(U)$ for which the moment map for the $T^n$ action is $\psi\circ \pi_k\circ \pi$. Symplectically, we have that $(\pi_k\circ \pi)^{-1}(\cl{n})\cong \p^{n-k}\times \pi_k^{-1}(\cl{n})$, so we set $L_{k+1}=\text{Cliff}(\mathbb{P}^{n-k})\times L_k$ as above. Particularly, $L_k$ is a product of $k$ Clifford tori with the $i^{th}$ factor having dimension $n-i+1$.

\subsubsection*{Disk potential for $L_{k+1}$} Let $\omega_k=a_{\g'}+K_1\pi_k^*\omega_{FS}$ with $e(u)$ as the energy of a disk with respect to $\omega_k$. The induction hypothesis includes assuming that $L_k$ is unobstructed and Floer-non-trivial in $F_n^k$. Some notation: Let $u_{i\vert j}$ denote the $j^{th}$ holomorphic disk with boundary in the fiber of $L_i$ (the $i^{th}$ Clifford torus) according to the classification by \cite{cho}, and let $x_{i\vert j}$ be the corresponding evaluation of the local system on the boundary of $u_{i\vert j}$. Let $x_{i\vert}^\mathfrak{f}:=x_{i\vert 1}\cdots x_{i\vert n-i+1}$. We assume that the induction hypothesis gives $\mc{W}^{k}_{L_k}[t]$, the $k^{th}$ order terms of $\mu^0_{L_K}$, as
\begin{gather}\label{baseflagpot}
\mc{W}^{k}_{L_k}[t]=\Bigg[\sum_{\substack{1\leq i\leq k\\ 1\leq j\leq n-i+1}}x_{i\vert j} t^{e(u_{i\vert j})}+\sum_{ 1\leq i\leq k}\frac{1}{x_{i\vert}^\mathfrak{f} x_{i+1\vert}^{\nu_{i\vert i+1}}\cdots x_{k\vert}^{\nu_{i\vert k}}}t^{e(u_{i\vert n+2-i})}\Bigg]x_M.
\end{gather}

The notation of the lifting operator is suppressed, i.e. $u_{i\vert j}$ denotes\linebreak $\mathcal{L}^{k-i}u_{i\vert j}$ where $\mathcal{L}$ is the lift of a disk from $(F_n^{l-1},L_{l-1})$ to $(F_n^l,L_l)$. Here, $\nu_{i\vert j}$ are the exponents corresponding to the representation evaluated on the lift $\mathcal{L}^{k-i}u_{i\vert j}$ from $(F_{n}^{i},L_{i})$ to $(F_n^{k},L_{k})$ (and the result of the change of basis that absorbs any result of lifting $u_{i\vert m}$ for $m\leq n+1-i$). Since $L_k$ is constructed from Clifford tori using Theorem \ref{moment} we can assume that any vertical symplectic area vanishes and $e(u_{i\vert m})=e(u_{i\vert l})$. If the fiberwise symplectic form is small enough at each step, we have $e(u_{i\vert l})>e(u_{j\vert m})$ for $i<j$. With this in mind, $\mc{W}^2_{L_{k+1}}[t]$ will not involve all of the variables $x_{i\vert j}$, so a critical point will necessarily be degenerate. Therefore, we must use $\mc{W}_{L_{k+1}}^{k+1}$ to compute the Floer theory of $L_{k+1}$.

Let $\omega_{k+1}=a_{\g}\!+\!K_2\pi^*\omega_k$. For a $J$-holomorphic disk $u:D\!\rightarrow\! (F_n^{k+1},L_{k+1})$, let $e_v(u)$ denote the energy with respect to $a_\g$, and $e(u)$ the energy with respect to $K_2\pi^*\omega_k$. After a suitable change of basis for the $x_{i\vert j}$ $j\leq n+1-i$, the lift of \eqref{baseflagpot} is by definition
\begin{align}
&\mathcal{L}\circ \mc{W}_{L_k}^k[q,r]=\sum_{\substack{1\leq i\leq k\\ 1\leq j\leq n+1-i}}x_{i\vert j} q^{e(u_{i\vert j})}r^{e_v(\mathcal{L}u_{i\vert j})}\\
&\qquad\qquad+\sum_{ 1\leq i\leq k}\frac{1}{x_{i\vert}^\mathfrak{f}x_{i+1\vert}^{\nu_{i\vert i+1}}\cdots x_{k\vert}^{\nu_{i\vert k}}x_{k+1\vert}^{\nu_{i\vert k+1}}}q^{e(u_{i\vert n+2-i})}r^{e_v(\mathcal{L}u_{i\vert n+2-i})}.
\end{align}
By the classification of disks with boundary in $\text{Cliff}(\mathbb{P}^{n-k})$, the argument in Theorem \ref{potthm}, and by the induction assumption, the $(k+1)^{st}$ order potential for $L_{k+1}$ is
\begin{align}
&\mc{W}^{L_{k+1}}_{k+1}[q,r]=\sum_{\substack{1\leq i\leq k\\ 1\leq j\leq n+1-i}}x_{i\vert j} q^{e(u_{i\vert j})}r^{e_v(\mathcal{L}u_{i\vert j})}\\
&\qquad\qquad +\sum_{ 1\leq i\leq k}\frac{1}{x_{i\vert}^\mathfrak{f}x_{i+1\vert}^{\nu_{i\vert i+1}}\cdots x_{k\vert}^{\nu_{i\vert k}}x_{k+1\vert}^{\nu_{i\vert k+1}}}q^{e(u_{i\vert n+2-i})}r^{e_v(\mathcal{L}u_{i\vert n+2-i})}\\ 
&\qquad\qquad + \sum_{i=i}^{n-k}x_{k+1\vert i}r^{e_v(u_{k+1\vert i})}+\frac{1}{x_{k+1\vert}^\mathfrak{f}}r^{e_v(u_{k+1\vert n-k+1})}.
\end{align}

The extension of the argument from Theorem \ref{potthm} that applies to our case can be summed up as follows: The term in the energy involving the connection $1$-form vanishes by the construction of $L$, so energy of the lifted disks is preserved. By the proof of the theorem there are no other disks that lift from $L_k$ and contribute the the $(k+1)^{st}$ order potential, so we are only left with the first order terms from the fiber.

The partial derivative of this potential with respect to the variables $x_{k+1\vert j}$ (after tensoring with $t$ and factoring) is 
\begin{gather}\label{inductionequation}
\partial_{x_{k+1\vert j}} \mc{W}_{L_{k+1}}^{k+1}[t]=
\sum_{1\leq i\leq k}\frac{-\nu^j_{i\vert k+1}}{x_{i\vert}^\mathfrak{f}\cdots x_{k\vert}^{\nu_{i\vert k}}x_{k+1\vert}^{\nu_{i\vert k+1}+\mathfrak{f}_j}}t^\alpha
+ 1+\frac{-1}{x_{k+1\vert j}\cdot x_{k+1\vert}^\mathfrak{f}}.
\end{gather}

Setting $x_{k+1\vert l}=e^{\frac{2\pi i m}{n-k+1}}$ for all $l$ gives a solution to $\eqref{inductionequation} = 0\; (\T{mod}\, t)$. Similarly, setting $x_{i\vert j}$ to some consistent root of unity gives a solution to 
$$D \mc{W}^{k+1}_{L_{k+1}} = 0 \; (\T{mod}\, t).$$ Thus, an application of the multivariate Hensel's lemma gives us a solution to $$D \mc{W}^{k+1}_{L_{k+1}} = 0.$$

By Theorem \ref{crit}, there is a representation $\xi \in \T{Hom}(\pi_1(L_{k+1}),\Lambda_t)$ so that the Floer cohomology at this representation with a suitable Maurer-Cartan solution is isomorphic to singular cohomology with Novikov coefficients. Thus $L_{k+1}$ is Hamiltonian non-displaceable.

\printbibliography

\end{document}

%% file: fig_familydiskdomain.pdf_tex
\begingroup%
  \makeatletter%
  \providecommand\color[2][]{%
    \errmessage{(Inkscape) Color is used for the text in Inkscape, but the package 'color.sty' is not loaded}%
    \renewcommand\color[2][]{}%
  }%
  \providecommand\transparent[1]{%
    \errmessage{(Inkscape) Transparency is used (non-zero) for the text in Inkscape, but the package 'transparent.sty' is not loaded}%
    \renewcommand\transparent[1]{}%
  }%
  \providecommand\rotatebox[2]{#2}%
  \newcommand*\fsize{\dimexpr\f@size pt\relax}%
  \newcommand*\lineheight[1]{\fontsize{\fsize}{#1\fsize}\selectfont}%
  \ifx\svgwidth\undefined%
    \setlength{\unitlength}{334.67884827bp}%
    \ifx\svgscale\undefined%
      \relax%
    \else%
      \setlength{\unitlength}{\unitlength * \real{\svgscale}}%
    \fi%
  \else%
    \setlength{\unitlength}{\svgwidth}%
  \fi%
  \global\let\svgwidth\undefined%
  \global\let\svgscale\undefined%
  \makeatother%
  \begin{picture}(1,1)%
    \lineheight{1}%
    \setlength\tabcolsep{0pt}%
    \put(0,0){\includegraphics[width=\unitlength,page=1]{fig_familydiskdomain.pdf}}%
    \put(2.85019324,0.37249094){\color[rgb]{0,0,0}\makebox(0,0)[lt]{\begin{minipage}{1.69777624\unitlength}\raggedright \end{minipage}}}%
    \put(4.05850858,-0.00983266){\color[rgb]{0,0,0}\makebox(0,0)[lt]{\begin{minipage}{0.52121757\unitlength}\raggedright \end{minipage}}}%
    \put(3.63555371,-0.02264951){\color[rgb]{0,0,0}\makebox(0,0)[lt]{\begin{minipage}{1.17487568\unitlength}\raggedright \end{minipage}}}%
    \put(0.3019,0.50438547){\color[rgb]{0,0,0}\makebox(0,0)[lt]{\lineheight{1.25}\smash{\begin{tabular}[t]{l}$u_1^{-1}(\partial U)$\end{tabular}}}}%
    \put(0,0){\includegraphics[width=\unitlength,page=2]{fig_familydiskdomain.pdf}}%
    \put(0.90684649,0.90724062){\color[rgb]{0,0,0}\makebox(0,0)[lt]{\lineheight{1.25}\smash{\begin{tabular}[t]{l}$\mc{P}_{\alpha_1\dots\alpha_n}$\end{tabular}}}}%
    \put(-0.95178432,-0.18159261){\color[rgb]{0,0,0}\makebox(0,0)[lt]{\begin{minipage}{1.33767734\unitlength}\raggedright \end{minipage}}}%
    \put(0,0){\includegraphics[width=\unitlength,page=3]{fig_familydiskdomain.pdf}}%
    \put(0.25554135,0.15412745){\color[rgb]{0,0,0}\makebox(0,0)[lt]{\lineheight{1.25}\smash{\begin{tabular}[t]{l}supp($\phi$)\end{tabular}}}}%
    \put(0,0){\includegraphics[width=\unitlength,page=4]{fig_familydiskdomain.pdf}}%
  \end{picture}%
\endgroup%

%% file: fig_toricpicturefamily.pdf_tex
\begingroup%
  \makeatletter%
  \providecommand\color[2][]{%
    \errmessage{(Inkscape) Color is used for the text in Inkscape, but the package 'color.sty' is not loaded}%
    \renewcommand\color[2][]{}%
  }%
  \providecommand\transparent[1]{%
    \errmessage{(Inkscape) Transparency is used (non-zero) for the text in Inkscape, but the package 'transparent.sty' is not loaded}%
    \renewcommand\transparent[1]{}%
  }%
  \providecommand\rotatebox[2]{#2}%
  \newcommand*\fsize{\dimexpr\f@size pt\relax}%
  \newcommand*\lineheight[1]{\fontsize{\fsize}{#1\fsize}\selectfont}%
  \ifx\svgwidth\undefined%
    \setlength{\unitlength}{206.10288845bp}%
    \ifx\svgscale\undefined%
      \relax%
    \else%
      \setlength{\unitlength}{\unitlength * \real{\svgscale}}%
    \fi%
  \else%
    \setlength{\unitlength}{\svgwidth}%
  \fi%
  \global\let\svgwidth\undefined%
  \global\let\svgscale\undefined%
  \makeatother%
  \begin{picture}(1,0.81473179)%
    \lineheight{1}%
    \setlength\tabcolsep{0pt}%
    \put(0,0){\includegraphics[width=\unitlength,page=1]{fig_toricpicturefamily.pdf}}%
    \put(0.34948836,0.25352511){\color[rgb]{0,0,0}\makebox(0,0)[lt]{\lineheight{1.25}\smash{\begin{tabular}[t]{l}supp($\phi$)\end{tabular}}}}%
    \put(0.36767935,0.40717213){\color[rgb]{0,0,0}\makebox(0,0)[lt]{\lineheight{1.25}\smash{\begin{tabular}[t]{l}$U$\end{tabular}}}}%
    \put(0.70274282,0.19789791){\color[rgb]{0,0,0}\makebox(0,0)[lt]{\lineheight{1.25}\smash{\begin{tabular}[t]{l}$\mc{P}_{\alpha_1\alpha_2}$\end{tabular}}}}%
    \put(1.73503375,0.04764168){\color[rgb]{0,0,0}\makebox(0,0)[lt]{\begin{minipage}{0.16375316\unitlength}\raggedright \end{minipage}}}%
    \put(0,0){\includegraphics[width=\unitlength,page=2]{fig_toricpicturefamily.pdf}}%
    \put(1.78057494,0.78684513){\color[rgb]{0,0,0}\makebox(0,0)[lt]{\begin{minipage}{0.74411295\unitlength}\raggedright \end{minipage}}}%
    \put(0.70274293,0.35925151){\color[rgb]{0,0,0}\makebox(0,0)[lt]{\lineheight{1.25}\smash{\begin{tabular}[t]{l}$\mc{P}_{\alpha'_1\alpha_2}$\end{tabular}}}}%
    \put(0,0){\includegraphics[width=\unitlength,page=3]{fig_toricpicturefamily.pdf}}%
    \put(0.54859281,0.04230735){\color[rgb]{0,0,0}\makebox(0,0)[lt]{\lineheight{1.25}\smash{\begin{tabular}[t]{l}$u_1$\end{tabular}}}}%
    \put(1.53960794,-0.1215959){\color[rgb]{0,0,0}\makebox(0,0)[lt]{\begin{minipage}{0.32278938\unitlength}\raggedright \end{minipage}}}%
  \end{picture}%
\endgroup%